\documentclass[11pt]{amsart}
\usepackage[top=5cm,bottom=4.2cm,right=3cm,left=3cm]{geometry}
\usepackage{mathrsfs}
\usepackage{yfonts}
\usepackage{amssymb}
\usepackage{latexsym}
\usepackage{amsmath}
\usepackage{graphicx,color}
\usepackage[all]{xy}
\usepackage{pstricks}
\usepackage{pst-plot}
\usepackage{epsfig}
\theoremstyle{plain}
\newtheorem{thm}{Theorem}[section]
\newtheorem{propo}[thm]{Proposition}
\newtheorem{lem}[thm]{Lemma}
\newtheorem{cor}[thm]{Corollary}

\theoremstyle {definition}

\theoremstyle{remark}
\newtheorem*{rem}{Remark} 

\definecolor{grey30}{rgb}{0.3,0.3,0.3}
\definecolor{grey67}{rgb}{0.67,0.67,0.67}

\newcommand{\nin}{\notin} 

\newcommand{\ci}{\subseteq} 

\newcommand{\tif}{\text{if }}

\newcommand{\tand}{\text{ and }}

\newcommand{\tfor}{\text{for }}

\newcommand{\twith}{\text{with }}

\newcommand{\ga}{\alpha}
\newcommand{\gb}{\beta}
\newcommand{\gc}{\chi}
\newcommand{\gd}{\delta}
\newcommand{\gep}{\varepsilon}
\newcommand{\gf}{\varphi}
\newcommand{\gga}{\gamma}
\newcommand{\gh}{\eta}

\newcommand{\gm}{\mu}
\newcommand{\gn}{\nu}

\newcommand{\gp}{\pi}

\newcommand{\gs}{\sigma}
\newcommand{\gr}{\rho}
\newcommand{\gt}{\tau}

\newcommand{\gx}{\xi}
\newcommand{\gy}{\psi}
\newcommand{\gz}{\zeta}

\newcommand{\gD}{\Delta}
\newcommand{\gF}{\Phi}

\newcommand{\gO}{\Omega}
\newcommand{\gP}{\Pi}

\newcommand{\gU}{\Upsilon}

\newcommand{\gY}{\Psi}
 
\def\B#1{\textbf{#1}}
\newcommand{\C}[1]{{\mathcal{#1}}} 
\newcommand{\D}[1]{{\mathbb{#1}}} 

\newcommand{\refT}[1]{Theorem ~\ref{#1}}

\newcommand{\refP}[1]{Proposition ~\ref{#1}}
\newcommand{\refC}[1]{Corollary ~\ref{#1}}
\newcommand{\refE}[1]{Equation ~\eqref{#1}}

{\par \samepage}%
{\par}

\newcommand{\ol}{\overline}

\newcommand{\pc}{\mathcal{PC}}
\newcommand{\rs}{\hat{\mathbb{C}}}
\newcommand{\ra}{\rightarrow}
\renewcommand{\Im}{\operatorname{Im}}
\renewcommand{\Re}{\operatorname{Re}}
   
\newcommand{\Dom}{\operatorname{Dom}}   
\newcommand{\cp}{\textup{cp}}
\newcommand{\cv}{\textup{cv}}
\newcommand{\IS}{\mathcal{I}\hspace{-1pt}\mathcal{S}}

\newcommand{\Csh}{\mathcal{C}^\sharp}
\newcommand{\rr}{\operatorname{\mathcal{R}}}
\newcommand{\irr}{HT_{N}}

\newcommand{\diam}{\operatorname{diam}\,}
\newcommand{\darun}{\operatorname{int}\,}
\newcommand{\ex}{\operatorname{\mathbb{E}xp}}
 
\newcommand{\eps}{\varepsilon}

\newcommand{\vfi}{\varphi}
\newcommand{\cc}{\mathbb{C}}
\newcommand{\ea}{e^{2\pi \ga \mathbf{i}}}
\newcommand{\p}{\mathcal{P}}
\newcommand{\Bk}{\boldsymbol{k}}
\newcommand{\co}[1]{^{\circ {#1}}}

\begin{document}
\title[Statistical properties of quadratic polynomials]%
{Statistical properties of quadratic polynomials with a neutral fixed point}
\author[Artur Avila]{Artur Avila}
\email[Artur Avila]{artur@math.jussieu.fr}
\address[Artur Avila]
{CNRS UMR 7586, Institut de Math\'ematiques de Jussieu - Paris Rive Gauche, 
B\^atiment Sophie Germain, Case 7012, 75205 Paris Cedex 13, France \& 
IMPA, Estrada Dona Castorina 110, 22460-320, Rio de Janeiro, Brazil}

\author[Davoud Cheraghi]{Davoud Cheraghi}
\email[Davoud Cheraghi]{d.cheraghi@imperial.ac.uk}
\address[Davoud Cheraghi]{Department of Mathematics, Imperial College London, United Kingdom}
\thanks{The second author acknowledges funding from Leverhulme Trust while carrying out this research.}
\subjclass[2010]{Primary: 37F50, Secondary: 37F25, 58F11}

\begin{abstract}
We describe the statistical properties of the dynamics of the quadratic polynomials 
$P_\ga(z)=\ea z+z^2$ on the complex plane, with $\ga$ of \textit{high return} times.
In particular, we show that these maps are \textit{uniquely ergodic} on their \textit{measure 
theoretic attractors}, and the unique invariant probability is a \textit{physical measure} 
describing the statistical behavior of typical orbits in the \textit{Julia set}. 
This confirms a conjecture of Perez-Marco on the unique ergodicity of \textit{hedgehog} 
dynamics, in this class of maps.
\end{abstract}

\maketitle
\renewcommand{\thethm}{\Alph{thm}}
\section*{Introduction}\label{intro}
In this paper we are interested in the asymptotic distribution of the orbits of quadratic polynomials
\[P_\ga(z)=\ea z+z^2\]
acting on the complex plane, for irrational values of $\ga$.

When considering conservative dynamical systems (i.e., those preserving a smooth density),
the Ergodic Theorem assures the existence of a basic
statistical description (stationarity) of typical (with respect to Lebesgue measure) orbits, and the initial
focus of analysis tends to be the nature of the ergodic decomposition of Lebesgue measure.
However, for non-conservative dynamical systems, stationarity is far from assured
in principle.  A nice situation emerges when one can identify \textit{physical measures} $\mu$,
which describe the behavior of large subsets of the phase space, in the sense that their basins (the set of orbits
whose Birkhoff averages of any continuous observable are given by the spatial average with respect to $\mu$) have
positive Lebesgue measure.  Ideally, one would like to be able to describe the behavior of
almost every orbit using one of only finitely many physical measures: Indeed, whether such situation arises
frequently is the main content of the famous Palis conjecture \cite {Pal00}.

In the case of rational maps, the analysis is rather simple in the Fatou set (in particular, all orbits 
present stationary behavior), but the behavior in the Julia set is much less well understood. 
When the Julia set is not the whole Riemann sphere (e.g., for polynomial maps), a general theorem
of Lyubich \cite {Ly83b} assures that the $\omega$-limit set of almost every orbit is contained in 
the \textit{postcritical} set $\pc$,
defined as the closure of the forward orbits of all critical values contained in the Julia set, so it is of
importance to understand the asymptotic distribution of orbits in the postcritical set.  
Particularly, if all orbits in $\pc$ admit a common asymptotic distribution $\mu$ (i.e., the dynamics 
restricted to the postcritical set is uniquely ergodic), then almost every orbit in the Julia set must have 
this same asymptotic distribution, and $\mu$ will be a physical measure provided the Julia set 
has positive Lebesgue measure.

Of course, in many cases the Julia sets of quadratic polynomials have zero Lebesgue measure, 
and this is true in particular for almost every irrational $\ga$, see \cite {PZ04}.  
However, it has been shown by Buff and Ch\'eritat \cite {BC12} that for some irrational values of $\ga$ 
the Julia set of $P_\ga$ has positive Lebesgue measure. 
Their analysis is dependent on a renormalization approach introduced by Inou and Shishikura 
\cite {IS06} to control the postcritical set of $P_\ga$ whenever $\ga$ is of \textit{high type}: 
$\ga \in \irr$ for $N$ suitably large.
Here the class $\irr$ is defined in terms of the modified continued fraction expansion of $\ga$ as the 
set of all
\[
\ga=a_{-1}+\cfrac {\varepsilon_0} {a_0+\cfrac {\varepsilon_1} {a_1+\cfrac {\varepsilon_2} 
{a_2+\cdots}}}
\]
with $a_{-1}\in \mathbb{Z}$, $\varepsilon_i=\pm 1$ and $a_i \geq N \geq 2$, $i \geq 0$.

A systematic approach to study the dynamics of quadratics $P_\ga$ using Inou-Shishikura 
renormalization has been laid out by Cheraghi in \cite{Ch10-I}. 
A new analytic technique is introduced in \cite{Ch10-II} to prove optimal estimates on the changes of 
coordinates (Fatou coordinates) that appear in the renormalization. 
These have led to numerous important results on the dynamics of quadratics $P_\ga$ with $\ga$ 
of high type.
In particular, the combination of \cite{Ch10-I} and \cite{Ch10-II} provide a detailed geometric 
description of the postcritical sets $\pc(P_\ga)$ (implying zero area), 
and explain the topological behavior of the orbits of these maps. 
Here, we make use of the renormalization and the optimal estimates on the changes of coordinates 
in order to describe the asymptotic distribution of the orbits of $P_\ga$.

\begin{thm}\label{T:unique-ergodicity}
There exists $N \geq 2$ such that for $\ga\in \irr$, $P_\ga:\pc(P_\ga)\to \pc(P_\ga)$ is 
uniquely ergodic.
\end{thm}
For the remaining of this section, we fix $N$ as in the previous theorem.
\begin{cor}\label{C:physical}
For every $\ga \in \irr$, if $J(P_\ga)$ has positive Lebesgue measure then $P_\ga|J(P_\ga)$ 
admits a unique physical measure, which describes the behavior of almost every orbit in $J(P_\ga)$.
\end{cor}

While the use of the Inou-Shishikura renormalization operator of course restricts the set of 
parameters that we can address, it is believed that some renormalization operator with similar 
qualitative features should be available through the whole class of irrational numbers, so the 
arguments developed here might eventually be applied to the general case.
Irrational rotations of dynamically significant types such as bounded type, Brjuno, non-Brjuno, 
Herman, non-Herman, and Liouville are present in $\irr$. 
Less speculatively, considering the class $\irr$ is enough to conclude (due to the work of Buff-Ch\'eritat)
that Corollary \ref {C:physical} is indeed meaningful, in the sense that it applies non-trivially to some
quadratic polynomials.\footnote {We point out that ``Feigenbaum´´
quadratic polynomials with a Julia set of positive Lebesgue
measure have been recently constructed by Avila and Lyubich.  In those cases the
dynamics in the postcritical set is very simple and obviously uniquely ergodic (indeed it is almost periodic),
so they also admit physical measures.}

In order to identify precisely the unique invariant measure on the postcritical set, 
we need to discuss in more detail the dynamics of the maps $P_\ga$ near the origin.
One can identify two basic cases with fundamentally different behavior, distinguished 
by whether zero belongs to $\pc(P_\ga)$ or not, which turns out to depend on the local 
dynamics of $P_\ga$ at zero.  
The map $P_\ga$ is called \textit{linearizable} at zero, if there exists a local conformal change 
of coordinate $\phi$ fixing zero and conjugating $P_\ga$ to the rotation of angle $2\pi \ga$ around zero: 
$R_\ga\circ \phi=\phi\circ P_\ga$ on the domain of $\phi$, where $R_\ga(z)=\ea z$.
It was a nontrivial problem to determine the values of $\ga$ for which $P_\ga$ is linearizable. 
We refer the interested reader to \cite{Mi06} for its rich history, and only need to mention that the 
answer to this problem depends on the arithmetic nature of $\ga$. 
For Lebesgue almost every $\ga\in [0,1]$ such linearization exists, and on the other hand, for a 
generic choice of $\ga\in[0,1]$, $P_\ga$ is not linearizable. 
When $P_\ga $ is linearizable at zero, the maximal domain on which $P_\ga$ is conjugate to a rotation 
is called the \textit{Siegel disk} of $P_\ga$.

By a result of Ma\~n\'e \cite{Man87}, the orbit of the critical point is recurrent, and when $P_\ga$ 
is linearizable, the orbit of the critical point accumulates on the whole boundary of the Siegel disk, 
while when it is not linearizable, this orbit accumulates on zero.  
In either case, this allows us to construct natural invariant measures supported in the postcritical set: 
in the linearizable case, one takes the harmonic measure on the boundary of the Siegel disk, 
viewed from zero, while in the non-linearizable case, one takes the Dirac mass at the origin.  
It is worthwhile to point out that Buff and Ch\'eritat construct examples of Julia sets of positive 
Lebesgue measure in both linearizable and non-linearizable cases, so there are indeed physical 
measures of both types.

Theorem~\ref{T:unique-ergodicity} implies that there is no periodic point in $\pc(P_\ga)$, except 
possibly zero.
Here, we prove an interesting property of the dynamics of linearizable maps counterpart to the  \textit{small cycle property} of nonlinearizable quadratics obtain by Yoccoz \cite{Yoc95}.
That is, when $P_\ga$ is not linearizable at $0$, every neighborhood of $0$ contains  
infinitely many periodic cycles of $P_\ga$. 

\begin{thm}\label{T:small-cycle-property}
When $P_\ga$ is linearizable at $0$ for some $\ga\in \irr$, every neighborhood of the Siegel disk of $P_\ga$ 
contains infinitely many periodic cycles of $P_\ga$. 
\end{thm}

Our results can also be seen as providing insight onto the dynamics of so-called ``hedgehogs''.
Let $f$ be a holomorphic germ defined on a neighborhood of zero with $f(0)=0$ and $f'(0)=\ea$ for an irrational $\ga$. 
Perez-Marco in \cite{PM97} introduced local invariant sets for $f$, called \textit{Siegel compacta}, that were used to 
study the local dynamics of $f$ at zero.
More precisely, he proves that for $f$ as above and a Jordan domain $U\ni 0$ where $f$ and its inverse 
are one-to-one on a neighborhood of the closure of $U$, there exists a unique compact 
connected invariant set $K$ with $0\in K\ci \overline{U}$ and $K\cap \partial U\neq \emptyset$. 
A distortion property of the iterates of $f$ on $U$ is translated to the property of unique ergodicity of  
$f:\partial K\to \partial K$, which was conjectured to hold in this full generality.
When $f$ is not linearizable at zero, $K$ has no interior ($\partial K=K$) and is called the \textit{hedgehog} 
of $f$ on $U$. 
We show that the boundary of every Siegel-compacta of $P_\ga$ must be either an invariant curve in the Siegel disk of $P_\ga$ or is contained in $\pc(P_\ga)$. 
Hence, we can confirm the following partial result on this conjecture.
\footnote{The conjecture was announced in several workshops around 1995. 
It has been mistakenly reported in \cite[Section 3.4]{Yoc99} that this conjecture has been proved 
in full generality by R.~Perez-Marco. 
However, our communication with him confirms that there was never any proof of this statement.}

\begin{thm}\label{C:hedgehog-dynamics}
For every $\ga\in \irr$ and every Siegel compacta $K$ of $P_\ga$, the map 
$P_\ga:\partial K\to \partial K$ is uniquely ergodic.
\end{thm} 

Beside the quadratic polynomials, all results of this paper also apply to maps in the Inou-Shishikura class 
(see Section~\ref{SS:IS-class} for the definition) with rotation numbers of high type. 
In particular, one infers the appropriately interpreted statements for a large class of rational maps. 

The result of Inou and Shishikura \cite{IS06} and the analytic technique introduced in \cite{Ch10-II} 
have led to recent major advances on the dynamics of quadratic polynomials, and their 
numerous applications are still being harvested. 
They have been used to confirm a fine relation between the sizes of the Siegel disks and 
the arithmetic of the rotation $\ga$ in \cite{CC13}, and have resulted in a breakthrough on the 
local connectivity of the Mandelbrot set in \cite{ChSh14}. 
Most of the current paper is devoted to analyzing the delicate relation between the arithmetic 
of $\ga$ and the geometry of the renormalization tower.
Indeed, once we carry out this analysis, the proofs of the above theorems only occupy 
two to three pages each. 
We expect our analysis of this interaction assist with answering the remaining questions concerning 
the dynamics of the quadratic polynomials $P_\ga$.
\renewcommand{\thethm}{\thesection.\arabic{thm}}
\section{Near parabolic renormalization and an invariant class}\label{sec:prelim} 
\subsection{Inou-Shishikura class}\label{SS:IS-class}
Consider the cubic polynomial $P(z)=z(1+z)^2$. 
We have $P(0)=0$ and $P'(0)=1$. 
Also, $P$ has a critical point at $\cp_P=-1/3$ which is mapped to the critical value at $\cv_P=-4/27$, and another critical point at $-1$ which is mapped to zero.

Consider the filled in ellipse 
\[E= \Big\{x+\B{i} y\in \cc \; \Big | \; (\frac{x+0.18}{1.24})^2+(\frac{y}{1.04})^2\leq 1 \Big\},\]
and let 
\begin{equation}\label{E:U}
U= g(\rs \setminus E), \text{ where } g(z)=-\frac{4z}{(1+z)^2}.
\end{equation}
The domain $U$ contains $0$ and $\cp_P$, but not the other critical point of $P$ at $-1$.

Following \cite{IS06}, we define the class of maps
\begin{displaymath}
\IS\!=\Big\{h=P\circ \vfi^{-1}\!\!:U_h \rightarrow \cc \;\Big|%
\begin{array}{l} 
\text{$\vfi\colon U \ra U_h$ is univalent onto, $\vfi(0)=0$, $\vfi'(0)=1$,}\\ 
\text{and $\vfi$ has a quasi-conformal extension onto $\cc$.} 
\end{array}
\Big\}.
\end{displaymath}
Every map in $\IS$ has a fixed point of multiplier one at $0$, and a unique critical point at 
$\cp_h=\vfi(-1/3)\in U_h$ which is mapped to $\cv_h=-4/27$.
Elements of $\IS$ have the same covering structure as the one of $P$ on $U$. 

For $\ga\in \D{R}$, let $R_\ga$ denote the rotation of angle $\ga$ about zero; $R_\ga(z)=\ea z$.
By pre-composing the elements of $\IS$ with rotations $R_\ga$, $\ga\in \D{R}$, we define the class of maps   
\[\IS_\ga=\{ h \circ R_\ga \mid h\in \IS\}.\]
Let us also normalize the quadratic family to the form 
\[ Q_\ga(z)=\ea z+\frac{27}{16} e^{4\pi \ga \B{i}}z^2,\]
so that it has a fixed point of multiplier $\ea$ at zero, and its critical value is at $-4/27$.

Consider a holomorphic map $h\colon \!\!\Dom h \ra \cc$, where $\Dom h$ denotes the domain of definition 
(always assumed to be open) of $h$. 
Given a compact set $K\subset \Dom h$ and an $\eps>0$, a neighborhood of $h$ (in the compact-open topology) 
is defined as the set of holomorphic maps $g: \Dom g\to \D{C}$ such that $K \subset \Dom g$ and for all $z\in K$,  
$|g(z)-h(z)|<\eps$.
Then, a sequence $h_n:\Dom h_n\ra \cc$, $n=1,2,\dots$, \textit{converges} to $h$, if for every 
neighborhood of $h$ defined as above, $h_n$ is contained in that neighborhood for sufficiently large $n$. 
Note that the maps $h_n$ are not necessarily defined on the same set.


Every $h \in \IS_{\ga}$, with $h(z)=f_0(\ea z)$ for some $f_0\in \IS$ and $\ga\in \mathbb{R}$, fixes $0$ with multiplier  
$h'(0)=\ea$. 
Provided $\ga$ is small enough and non-zero, $h$ has a non-zero fixed point in $\Dom h$, denoted by $\sigma_h$, 
that has split from $0$ at $\ga=0$.
The $\sigma_h$ fixed point depends continuously on $f_0$ and $\ga$, with asymptotic expansion 
$\sigma_h=-4\pi \ga \B{i}/f_0''(0)+o(\ga)$, as $\ga$ tends to $0$. 
Clearly, $\sigma_h \ra 0$ as $\ga \ra 0$.  
Indeed, the choices of the domain $U$ and the polynomial $P$ guarantees (using the area theorem) that 
$f_0''(0)$ is uniformly away from $0$. 

\begin{lem}[\cite{IS06}]\label{L:bounds-on-second-derivative}
The set $\{h''(0)\mid h \in \IS\}$ is relatively compact in $\cc\setminus \{0\}$.
\end{lem}  

We summarize the basic local dynamics of maps in $\IS_{\ga}$, for small $\ga$, in the following theorem. 
See Figure~\ref{F:petal}.
\begin{thm}[Inou--Shishikura \cite{IS06}]\label{T:Ino-Shi1} 
There exists $\ga_* >0$ such that for every $h\colon U_h \ra \cc$ in $\IS_\ga \cup \{Q_\ga\}$ with $\ga \in (0,\ga_*]$,
there exist a Jordan domain $\p_h \subset U_h$ and a univalent map $\Phi_h\colon \p_h \ra \cc$ 
satisfying the following properties:
\begin{itemize}  
\item[(a)] The domain $\p_h$ is bounded by piecewise smooth curves and is compactly contained in $U_h$. 
Moreover, $\cp_h$, $0$, and $\sigma_h$ belong to the boundary of $\p_h$, while $\cv_h$ belongs to the interior of $\p_h$.
\item[(b)] $\gF_h(\p_h)$ contains the set $\{w\in \mathbb{C}\mid \Re w \in (0,1]\}$. 
\item[(c)] $\Im \Phi_h(z) \ra +\infty$ when $z\in \p_h\ra 0$, and $\Im \Phi_h(z)\ra-\infty$ when 
$z \in \p_h \ra \sigma_h$.
\item[(d)] $\Phi_h$ satisfies the Abel functional equation on $\p_h$, that is, 
\[\Phi_h(h(z))=\Phi_h(z)+1, \text{ whenever $z$ and $h(z)$ belong to $\p_h$}.\] 
\item[(e)] The map $\Phi_h$ satisfying the above properties is unique, once normalized by setting 
$\Phi_h(\cp_h)=0$. 
Moreover, the normalized map $\Phi_h$ depends continuously on $h$.     
\end{itemize}
\end{thm}
In the above theorem and in the following statements, when $h=Q_\ga$, $U_h$ is taken as $\D{C}$.
 
The class $\IS$ is denoted by $\mathcal{F}_1$ in \cite{IS06}. 
The properties listed in the above theorem follow from Theorem 2.1 as well as Main Theorems 1 and 3 in \cite{IS06}.
We state some crucial geometric properties of the domains $\p_h$ in the following proposition. 
See \cite[Prop.~1.4]{Ch10-I}  or \cite[Prop.~12]{BC12} for different proofs of it. 

\begin{propo}[\cite{Ch10-I}, \cite{BC12}] \label{P:petal-geometry}
There exist $\ga_*'>0$ and positive integers $\B{k}, \B{k}'$ such that for every map $h\colon U_h \ra \cc$ in $\IS_\ga\cup\{ Q_\ga\}$ 
with $\ga \in (0,\ga_*']$, the domain $\p_h \subset U_h$ in the above theorem may be chosen to satisfy the additional properties:
\begin{itemize}  
\item[(a)] there exists a continuous branch of argument defined on $\p_h$ such that 
\[\max_{w,w'\in \p_h} |\arg(w)-\arg(w')|\leq 2 \pi \B{k}',\]
\item[(b)] $\Phi_h(\p_h)=\{w \in \cc \mid 0 < \Re(w) < \ga^{-1} -\Bk\}$. 
\end{itemize}
\end{propo}

The map $\Phi_h: \mathcal{P}_h\to \cc$ obtained in the above theorem is called the \textit{perturbed Fatou coordinate}, 
or the \textit{Fatou coordinate} for short, of $h$.
In this paper, by this coordinate we mean the map $\gF_h:\p_h\to \D{C}$, where $\p_h$ satisfies the extra properties in the 
above proposition. 
See Figure~\ref{F:petal}.

\begin{figure}[ht]
\begin{center}
\begin{pspicture}(9,5.5)
\rput(4.5,2.75){\includegraphics[width=8cm]{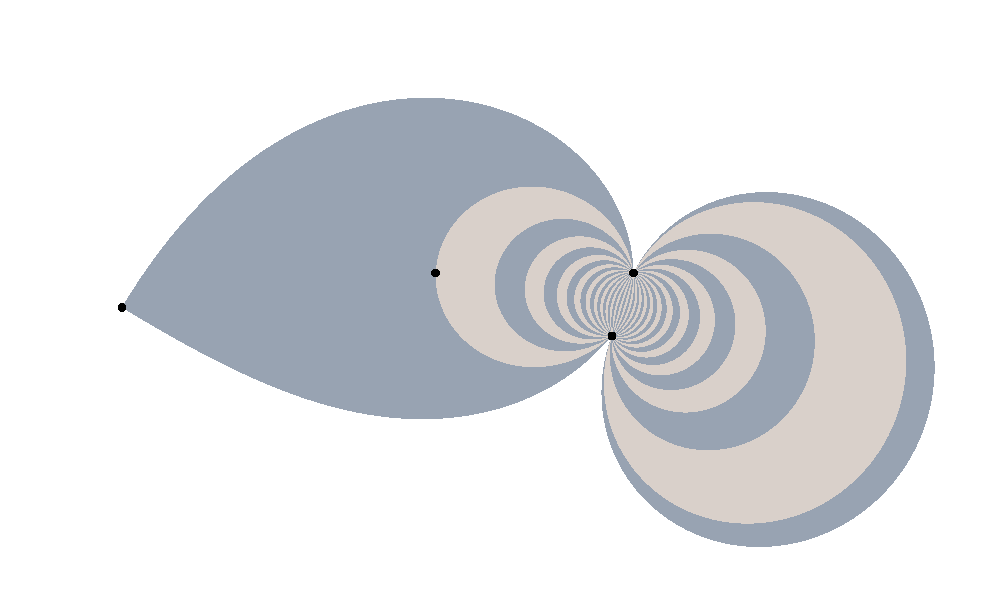}}
\rput(2.0,2.65){$\cp_h$}
\rput(3.65,2.95){$\cv_h$}
\rput(5.65,3.4){$0$}
\psline(5.2,2.05)(5.36,2.4)
\rput(5.1,1.9){$\sigma_h$}
\end{pspicture}
\end{center}
\caption{The domain $\C{P}_h$ and the special points associated to some $h\in\IS_\ga$. 
The map $\Phi_h$ sends each colored croissant to an infinite vertical strip of width one.}
\label{F:petal}
\end{figure}
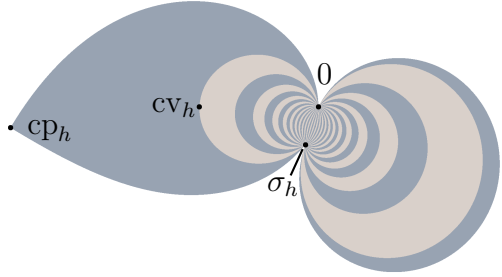


\subsection{Near parabolic renormalization}\label{SS:renormalization-def}
Let $h\colon U_h \ra \cc$ be in $\IS_{\ga} \cup \{Q_\ga\}$, with $\ga \in (0,\ga_*']$, and let $\Phi_h\colon \p_h \ra \cc$ 
denote the Fatou coordinate of $h$ defined in the previous section. 
Define 
\begin{equation}\label{E:sector-def}
\begin{gathered}
\C{C}_h=\{z\in \p_h : 1/2 \leq \Re(\Phi_h(z)) \leq 3/2 \: ,\: -2< \Im \Phi_h(z) \leq 2 \}, \\
\Csh_h=\{z\in \p_h : 1/2 \leq \Re(\Phi_h(z)) \leq 3/2 \: , \: 2\leq \Im \Phi_h(z) \}.
\end{gathered}
\end{equation}
By definition, $\cv_h\in \darun(\C{C}_h)$ and $0\in \partial(\Csh_h)$. 
We are also assuming that $\ga$ is small enough so that $1/\ga -\B{k}\geq 3/2$ 
(see Equation~\eqref{E:high-type-restriction}).  
See Figure~\ref{F:sectorpix}.

Assume for a moment that there exists a positive integer $k_h$, depending on $h$, such that the following four properties hold.
\begin{itemize}
\item For every integer $k$, with $1 \leq k \leq k_h$, there exists a unique connected component of $h^{-k}(\Csh_h)$ 
which is compactly contained in $\Dom h$ and contains $0$ on its boundary. We denote this component by 
$(\Csh_h)^{-k}$. 
\item For every integer $k$, with $1\leq  k \leq k_h$, there exists a unique connected component of $h^{-k}(\C{C}_h)$ which has 
non-empty intersection with $(\Csh_h)^{-k}$, and is compactly contained in $\Dom h$. 
This component is denoted by $\C{C}_h^{-k}$. 
\item The sets $\C{C}_h^{-k_h}$ and $(\Csh_h)^{-k_h}$ are contained in\footnote{The notation 
$\langle r \rangle$ stands for the closest integer to $r\in \D{R}$.} 
\[\{z\in\p_h \mid  \frac{1}{2}< \Re \Phi_h(z) < \langle \frac{1}{\ga} \rangle -\Bk-\frac{1}{2}\}.\] 
\item The maps $h: \C{C}_h^{-k}\to \C{C}_h^{-k+1}$, for $2\leq k \leq k_h$, and $h: (\Csh_h)^{-k}\to (\Csh_h)^{-k+1}$, for 
$1\leq k \leq k_h$, are univalent. 
The map $h: \C{C}_h^{-1} \to \C{C}_h$ is a degree two proper branched covering.
\end{itemize}
Let $k_h$ denote the smallest positive integer for which the above conditions hold, and define
\[S_h=\C{C}_h^{-k_h}\cup(\Csh_h)^{-k_h}.\]
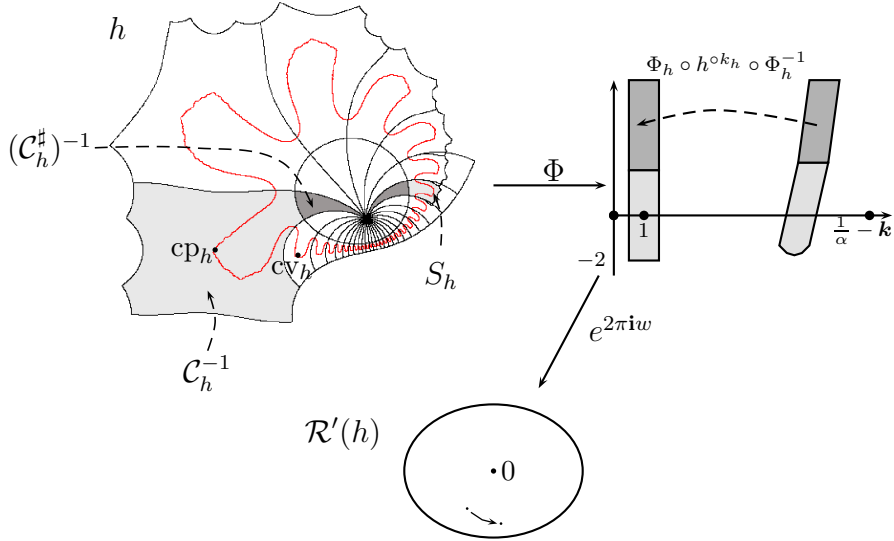
\begin{figure}[ht]
\begin{center}
 \begin{pspicture}(-.5,1.2)(11.4,9)
\epsfxsize=6.3cm
\rput(3.5,5.9){\epsfbox{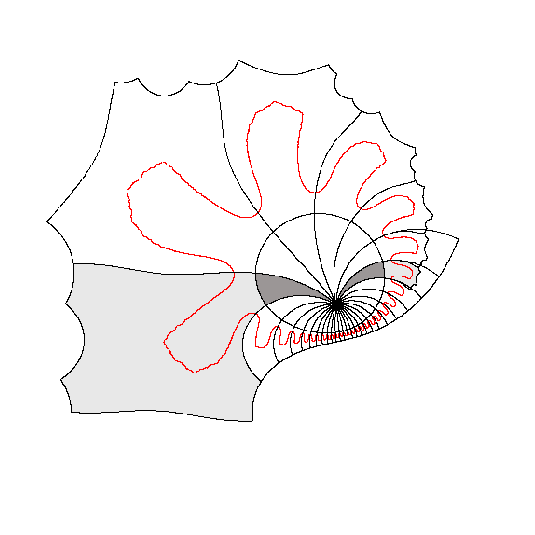}}  
  \psset{xunit=1cm}
  \psset{yunit=1cm}
    \pscurve[linewidth=.6pt,linestyle=dashed,linecolor=black]{->}(5.3,5.1)(5.3,5.5)(5.2,5.8)
    \pscurve[linewidth=.6pt,linestyle=dashed,linecolor=black]{->}(2.2,3.7)(2.3,4.1)(2.2,4.5)
    \pscurve[linewidth=.6pt,linestyle=dashed,linecolor=black]{->}(.7,6.4)(2,6.4)(3.2,6.2)(3.6,5.6)
    \rput(5.3,4.7){$S_h$}
    \rput(2.2,3.4){$\C{C}_h^{-1}$}
    \rput(.1,6.4){$(\Csh_h)^{-1}$}
    \psdots[dotsize=2pt](2.3,5.04)(3.4,4.97)
    \rput(2.,5){\small{$\cp_h$}}
    \rput(3.33,4.77){\small{$\cv_h$}}
\newgray{Lgray}{.99}
\newgray{LLgray}{.88}
\newgray{LLLgray}{.70}
\psdots(7.6,5.5)(8,5.5)(11,5.5)
\pspolygon[fillstyle=solid,fillcolor=LLLgray](7.8,7.3)(7.8,6.1)(8.2,6.1)(8.2,7.3)
\pspolygon[fillstyle=solid,fillcolor=LLgray](7.8,6.1)(7.8,4.9)(8.2,4.9)(8.2,6.1)

\pspolygon[fillstyle=solid,fillcolor=LLLgray](10.2,7.3)(10.05,6.2)(10.45,6.2)(10.6,7.3)
\pspolygon[fillstyle=solid,fillcolor=LLgray](10.05,6.2)(9.8,5.1)(9.9,5)(10,4.97)(10.1,5)(10.2,5.1)(10.45,6.2)
\psdots(7.6,5.5)(8,5.5)
\psline{->}(7.6,5.5)(11.3,5.5)
\psline{->}(7.6,4.7)(7.6,7.3)
\rput(8,5.3){\tiny{$1$}}
\rput(10.9,5.3){\tiny{$\frac{1}{\ga}-\Bk$}}
\rput(7.3,4.9){\tiny{$-2$}}

\psline{->}(6,5.9)(7.5,5.9)
\rput(6.8,6.1){$\Phi$}

\pscurve[linestyle=dashed]{<-}(7.9,6.7)(8.9,6.9)(10.3,6.7)
\rput(9.1,7.5){\tiny{$\gF_h \circ h\co{k_h}\circ \gF_h^{-1}$}}

\psline{->}(7.4,4.7)(6.6,3.2)
\rput(7.7,4.){$e^{2\pi \B{i} w}$}
\psellipse(6,2.1)(1.2,.9)
\psdots[dotsize=2pt](6,2.1)
\rput(4,2.6){$\rr' (h)$}
\rput(6.2,2.1){\small{$0$}}
\rput(1,8){$h$}
\psline[linewidth=.5pt]{->}(5.7,1.55)(5.85,1.45)(6.05,1.4)
\NormalCoor
\psdot[dotsize=1pt](5.65,1.6)
\psdot[dotsize=1pt](6.1,1.4)
 \end{pspicture}
\caption{The figure shows the sets $\C{C}_h$, $\Csh_h$,..., $\C{C}_h^{-k_h}$, $(\Csh_h)^{-k_h}$, and the sector $S_h$. 
The induced map $\gF_h \circ h\co{k_h}\circ \gF_h^{-1}$ projects via $e^{2\pi \B{i} w}$ to a well-defined map 
$\rr (h)$ on a neighborhood of $0$. 
The amoeba curve around $0$ is a large number of iterates of $\cp_h$ under $h$.}
\label{F:sectorpix}
\end{center}
\end{figure}

Consider the map 
\begin{equation}\label{E:renorm-def}
E_h=\Phi_h \circ h\co{k_h} \circ \Phi_h^{-1}:\Phi_h(S_h) \ra \cc. 
\end{equation}
By the functional equation in Theorem~\ref{T:Ino-Shi1}-d), $E_h(z+1)=E_h(z)+1$, when both $z$ and $z+1$ belong to the boundary of 
$\gF_h(S_h)$. 
Hence, $E_h$ projects via $z=\frac{-4}{27}e^{2 \pi \B{i} w}$ to a map $\rr (h)$ defined on a set containing a 
punctured neighborhood of $0$. 
However, zero is a removable singularity of this map, and one can see that $\rr(h)$ must be of the form  
$z \mapsto e^{2 \pi \frac{-1}{\ga}\B{i}}z+ O(z^2)$ near zero. 
The map  $\rr (h)$, restricted to the interior of $\frac{-4}{27}e^{2\pi \B{i}(\Phi_h(S_h))}$, is called the 
\textit{near-parabolic renormalization} of $h$. 
We may simply refer to the near-parabolic renormalization as \textit{renormalization} for short, 
\footnote{Inou and Shishikura give a somewhat different definition of this renormalization operator using slightly 
different regions $\C{C}_h$ and $\Csh_h$ compared to the ones here. 
However, the resulting maps $\rr(h)$ are the same modulo their domains of definition.
More precisely, there is a natural extension of $\gF_h$ onto the sets $\C{C}_h^{-k} \cup (\Csh_h)^{-k}$, for $0\leq k\leq k_h$, 
such that each set $\Phi_h(\C{C}_h^{-k}\cup (\Csh_h)^{-k})$ is contained in the union 
\[D^\sharp_{-k} \cup D_{-k} \cup D''_{-k} \cup D'_{-k+1} \cup D_{-k+1} \cup D^\sharp_{-k+1}\] 
in the notations used in \cite[Section 5.A]{IS06}.}. 
Note that $\gF_h$ maps the critical value of $h$ to one, and the projection $w\mapsto \frac{-4}{27}e^{2 \pi \B{i} w}$ 
maps integers to $-4/27$.
Thus, the critical value of $\rr(h)$ is at $-4/27$. 
See Figure~ \ref{F:sectorpix}. 

Define 
\begin{equation}\label{E:V}
V=P^{-1}\big (B(0,\frac{4}{27}e^{4\pi}) \big ) \setminus  \big ( (-\infty,-1]\cup B \big)
\end{equation}
where $B$ is the component of $P^{-1}(B(0,\frac{4}{27}e^{-4\pi}))$ containing $-1$. 
By an explicit calculation (see \cite[Prop.~5.2]{IS06}) one can see that the closure of $U$ is contained in the interior of $V$.   
See Figure~\ref{F:poly}. 

\begin{figure}[ht]
\begin{center}
  \begin{pspicture}(8,3.2)
  \rput(4.5,1.6){\epsfbox{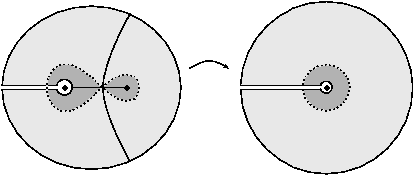}}
      \rput(6.12,1.6){{\small $\times$}}
      \rput(5.9,1.4){{\small $cv_P$}}
      \rput(6.7,1.8){{\small $0$}} 
      \rput(2.7,1.6){{\small $\times$}}
      \rput(2.7,1.3){{\small $cp_P$}}
      \rput(3.2,1.8){{\small $0$}}
      \rput(1.85,1.3){{\small $-1$}}
      \rput(4.5,2.2){{\small $P$}}
      \rput(1.1,2.7){{\small $V$}} 
\end{pspicture}
\caption{A schematic presentation of the polynomial $P$; its domain, and its range. 
Similar colors and line styles are mapped on one another.}
\label{F:poly}
\end{center}
\end{figure}
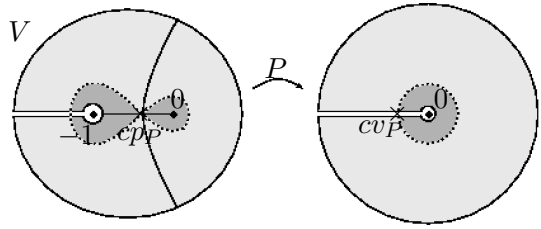

The following theorem \cite[Main thm~3]{IS06} guarantees that the above definition of 
renormalization $\rr$ can be carried out for certain perturbations of maps in $\IS$. 
In particular, this implies the existence of $k_h$ satisfying the four properties needed for the definition of renormalization. 

\begin{thm}[Inou-Shishikura]\label{T:Ino-Shi2} 
There exist a constant $\ga^*>0$ such that if $h \in \IS_\ga$ with $\ga \in (0,\ga^*]$, then $\rr(h)$ is 
well-defined and belongs to the class $\IS_{-1/\ga}$. 
That is, there exists a univalent map $\psi:U\to \cc$ with $\psi(0)=0$ and $\psi'(0)=1$ so that 
\[\rr(h)(z)=P \circ \psi^{-1}(e^{\frac{-2\pi}{\ga}\textnormal{\B{i}}} z), \; \forall z\in \psi(U)\cdot e^{\frac{2\pi}{\ga}\textnormal{\B{i}}}.\] 
Furthermore,  $\psi:U\to \cc$ extends to a univalent map on $V$.    

Similarly, for $\ga\in (0,\ga^*]$, $\rr(Q_\ga)$ is well-defined and belongs 
to $\IS_{-1/\ga}$.
\end{thm}

For $h\in \IS_\ga$ or $h=Q_\ga$ with $\ga\in [-\ga^*, 0)$, the conjugate map $\hat{h}=s\circ h\circ s$,
where $s(z)=\ol{z}$ is the complex conjugation, satisfies $\hat{h}(0)=0$ and $\hat{h}'(0)=e^{-2\pi \ga \B{i}}$. 
Since the class of maps $\IS$ is invariant under conjugation by $s$, $\hat{h}$ belongs to the class $\IS_{-\ga}\cup \{Q_{-\ga}\}$. 
In particular, by the above theorems the near parabolic renormalization of $\hat{h}$ is defined.    
Through this, we extend the domain of definition of $\rr$ to contain the set of maps $h\in \IS_\ga\cup \{Q_\ga\}$, $\ga\in [-\ga^*,0)$. 

The near parabolic renormalization and a small variation of Theorem~\ref{T:Ino-Shi2} is 
extended to cover unisingular holomorphic maps in \cite{C14}. 
A numerical study of the parabolic renormalization of Inou-Shishikura has been carried out in 
\cite{LaYa11}.  

\subsection{Fatou coordinates}\label{SS:extending-F-coord}

In this section we state some basic properties of the Fatou coordinates that will be used throughout 
this paper. 
One may refer to \cite{Ch10-I} for detailed arguments on these. 
We say that a smooth curve $\gga:(0,1) \to \D{C}\setminus \{0\}$ lands at $0$ if $\gga(t) \to 0$ as $t\to 0$. 
Moreover, we say that $\gga$ lands at $0$ at a well-defined angle if there is a branch of $\arg$ 
defined on $\gga(t)$ for all small values of $t$ and $\lim_{t\to 0} \arg \gga(t)$ exists.  
The next proposition follows from the estimates on Fatou coordinates in \cite{Ch10-I}. 
As it is not stated in that paper, we present a proof of it in Section~\ref{SS:estimates-F-coord}.

\begin{propo}\label{P:landing-angle}
For all $h\in \IS_\ga \cup \{Q_\ga\}$ with $0 < |\ga| \leq \ga^*$ and every $r\in [0, 1/\ga -\Bk]$, 
the curve $\gF_h^{-1}(\mathbb{R}+ r\textnormal{\B{i}})$ lands at $0$ at a well-defined angle. 
\end{propo}

It follows from the above proposition that for all $h\in \IS_\ga \cup \{Q_\ga\}$ with 
$0 < |\ga| \leq \ga^*$ the sector $S_h$ is bounded by piecewise smooth curves two of which land 
at $0$ at well-defined angles. 
Then, one can see that for every $h\in \IS_\ga \cup \{Q_\ga\}$ with $0 < |\ga| \leq \ga^*$ we have 
\begin{equation}\label{E:bounded-remaining}
k_h \geq 2.
\end{equation}
On the other hand we have the following upper bound on $k_h$. 
\begin{propo}[\cite{Ch10-I}]\label{P:turning}
There exists $\Bk''\in \D{Z}$ such that for every $h$ in $\IS_\ga\cup\{ Q_\ga\}$ with $0 < |\ga| \leq \ga^*$, $k_h\leq \Bk''$. 
\end{propo}

\begin{propo}\label{P:well-contained-in-domain}
There is a constant $\gd>0$ such that for every $h$ in $\IS_\ga\cup\{ Q_\ga\}$ 
with $0 < |\ga| \leq \ga^*$ and every $z\in \cup_{i=0}^{k_h} h\co{i}(S_h) \cup \C{P}_h$, 
we have $|z|\leq 1/\gd$, $B_\gd(z) \subset \Dom h$, and 
$B_\gd(\C{C}_h^{-k_h}) \subset \Dom h\setminus \{0\}$. 
\end{propo}

\begin{proof}
Fix $\ga$ satisfying the hypothesis. 
According to Inou-Shishikura \cite{IS06}, for every $h\in \IS_\ga\cup\{ Q_\ga\}$ the set 
$\cup_{i=0}^{k_h} h\co{i}(S_h) \cup \C{P}_h$ is compactly contained in $\Dom h$. 
Indeed, they show that \cite[Section 5.N, the value of $\gh$]{IS06}, one may use the set 
$\{z\in \C{P}_h \mid 1/2 \Re \gF_h \leq 3/2,  -13 \leq \Im w \leq 2 \}$ in place of $\C{C}_h$ (considered here) to define the near parabolic renormalization of $h$. 
That is, corresponding pre-images of this (larger) set are defined and contained in $\Dom h$. 
In particular, the pre-images of the set $\C{C}_h \cup \Csh_h$, up to $k_h$, are compactly 
contained in $\Dom h$. 
Similarly, for every $h$, the set $\C{C}_h^{-k_h}$ is compactly contained in $\Dom h \setminus \{0\}$.
By the pre-compactness of the class $\IS_\ga$ and the continuous dependence of $\gF_h$ on $h$, 
we conclude that there are constants $\gd$ and $c$ satisfying the conclusion of the lemma for $h$
in $\IS_\ga \cup \{Q_\ga\}$. 
It remains to see what happens as $\ga\to 0$. 

As $h \in \cup_{\ga\in (0,\ga^*]} \IS_\ga \cup \{Q_\ga\}$ tends to some map 
$h_0\in \IS_0 \cup \{Q_0\}$, the Fatou coordinate $\gF_h$ tends to the attracting and 
repelling Fatou coordinates of $h_0$. 
According to Inou and Shishikura \cite{IS06}, the sets $\C{C}_{h_0}$, $\Csh_{h_0}$, 
and their corresponding pre-images are defined for all $h_0\in \IS_0 \cup \{Q_0\}$, 
with $S_{h_0}$ contained in the domain of the repelling Fatou coordinate of $h_0$. 
Moreover, these domains are compactly contained in the domain of $h_0$.
Then, one defines the (parabolic) renormalization of $h_0$ as in the previous section.
By the work of Inou-Shishikura, the Fatou coordinate and the renormalization depend continuously 
on the map $h\in \cup_{\ga\in [-\ga^*,\ga^*]} \IS_\ga \cup \{Q_\ga\}$. 
This implies that there are positive constants $\gd$ and $c$ satisfying the properties 
in the lemma. 
As this is the only place we use the (parabolic) renormalization of maps in $\IS_0\cup \{Q_0\}$, 
we do not explain this further and refer the reader to \cite{IS06} for more details on this.
\end{proof}

Let $h\in \IS_\ga \cup \{Q_\ga\}$ with $0 < |\ga| < \ga^*$.
Define the set 
\[\C{D}_h':= \gF_h(\C{P}_h) \cup \bigcup_{j=0}^{k_h+ \langle 1/\ga \rangle -\B{k}-2} (\gF_h(S_h)+j).\]
Using the dynamics of $h$ we may extend the domain of definition of $\gF_h^{-1}$. 
\begin{lem}\label{L:extending-inverse-F-coord}
The map $\gF_h^{-1}: \gF_h(\C{P}_h) \to \C{P}_h$ extends to a holomorphic map 
\[\gF_h^{-1}: \C{D}_h' \to (\cup_{i=0}^{k_h} h\co{i}(S_h) \cup \C{P}_h) \setminus \{0\},\]
such that for all $w\in \mathbb{C}$ with $w, w+1\in \C{D}_h'$ we have 
\[\gF_h^{-1}(w+1)= h \circ \gF_h^{-1}(w). \] 
\end{lem}

\begin{proof}
If $z \in S_h$, then for all integers $j$, with $0 \leq j \leq k_h+ \langle 1/\ga \rangle -\B{k}-2$, 
$h\co{j}(z)$ is defined and belongs to $\Dom h$. 
That is because, since $h$ is near-parabolic renormalizable, the iterates $z, h(z), \dots, h\co{k_h}(z)$ are defined with  
$h\co{k_h}(z) \in \C{P}_h$ and $\Re \gF_h(h\co{k_h}(z))\in [1/2, 3/2]$. 
Then it follows from the conjugacy relation for $\gF_h$ and Proposition~\ref{P:petal-geometry} that for all $j$ with 
$k_h \leq j \leq k_h+ \langle 1/\ga \rangle -\B{k}-2$, $h\co{j}(z)$ is defined and belongs 
to $\C{P}_h$. 

Define the map $\gF_h^{-1}$ on $\C{D}_h'$ as 
\begin{equation}\label{E:extended-IFC}
\gF_h^{-1}(w)=
\begin{cases}
\gF_h^{-1}(w) & \tif  0<  \Re w < \ga^{-1}-\B{k} \\
h \co{j} \circ \gF_h^{-1}(w-j) & \tif w \in \gF_h(S_h)+j.
\end{cases}
\end{equation}
The conjugacy relation in Theorem~\ref{T:Ino-Shi1}-d implies that the above map is a well-defined holomorphic map on $\C{D}_h'$, 
and satisfies the desired functional equation on $\C{D}_h'$.
However, $\gF_h^{-1}$ is not univalent on $\C{D}_h'$. 
\end{proof}

Let us define the notation
\[\ex(\zeta) = \frac{-4}{27} e^{2\pi \B{i} \zeta}, \quad \ex:\D{C} \to \D{C}\setminus \{0\}.\]
Then, we may lift the maps $\gF_h^{-1}: \C{D}_h' \to \D{C}\setminus \{0\}$ and 
$s\circ \gF_h^{-1}: \C{D}_h' \to \D{C}\setminus \{0\}$ under the covering map $\ex: \D{C} \to \D{C}\setminus \{0\}$. 
That is, there is a map $\gc_{h}: \C{D}_h' \to \D{C}$ such that 
\begin{equation*}
\forall w \in \C{D}_h, \quad 
\begin{cases}
\ex \circ \gc_{h}(w)= \gF_h^{-1}(w), & \tif \ga\in (0, 1/2) \\
\ex \circ \gc_{h}(w)= s \circ \gF_h^{-1}(w), & \tif  \ga \in (-1/2,0).
\end{cases}
\end{equation*}
Each $\gc_{h}$ is either holomorphic or anti-holomorphic.
The lift $\gc_{h}$ is determined up to translations by integers. 

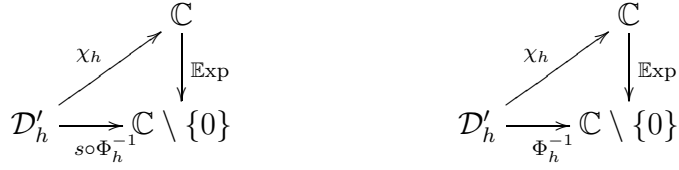
\begin{figure}[h]
\centering
\begin{minipage}[b]{0.35\linewidth}
\[
\xymatrix{
& \mathbb{C} \ar[d]^\ex \\ 
\C{D}_h' \ar[r]_{s\circ \gF_h^{-1}} \ar[ur]^{\gc_h}  &  \mathbb{C}\setminus \{0\}} 
\]
\end{minipage}
\begin{minipage}[b]{0.35\linewidth}
\[
\xymatrix{
& \mathbb{C} \ar[d]^\ex \\ 
\C{D}_h' \ar[r]_{\gF_h^{-1}} \ar[ur]^{\gc_h}  &  \mathbb{C}\setminus \{0\}} 
\]
\end{minipage}
\caption{The lift $\gc_h$ depends on the sign of $\ga$.}
\end{figure}

\begin{propo}\label{P:sector-geometry}
There exits an integer $\hat{\B{k}}$ such that for every $h\in \IS_\ga \cup \{Q_\ga\}$ with $0 < |\ga| \leq \ga^*$ and any choice of 
the lift $\gc_{h}$ we have 
\[\sup \{ |\Re w- \Re w'| :  {w,w' \in \gc_h(\C{D}_h')} \} \leq \hat{\Bk}.\]
\end{propo}
 
\begin{proof}
By Propositions~\ref{P:petal-geometry}, there is a uniform bound on the total spiral of the 
set $\C{P}_h$ about zero. 
By the pre-compactness of the class $\IS_\ga$, this implies that there is a uniform bound on the total spiral of each set $h\co{i}(S_h)$, for $0\leq i \leq k_h$, about zero. 
(See the proof of Proposition~\ref{P:well-contained-in-domain} for further details.)
In other words, the lifts of these sets under $\ex$ has uniformly bounded horizontal width. 
Combining with the uniform bound on $k_h$ in Proposition~\ref{P:turning}, this implies the 
existence of the constant $\hat{\Bk}$ satisfying the conclusion of the lemma. 
\end{proof}

Since $\gF_h^{-1}(\C{D}_h')$ is contained in the image of $h$,  for any choice of the lift $\gc_h$ and 
every $w\in \gc_h(\C{D}_h')$, we must have $\Im w > -2$. 
Hence, combining with the above proposition, we conclude that there is a choice of $\gc_h$, denoted by $\gc_{h,0}$ 
such that 
\[\gc_{h,0}(\C{D}_h') \subset \{w\in \D{C} \mid  1 \leq \Re w \leq \hat{\B{k}}+2, \Im w > -2\} .\]

Define the set 
\[\C{D}_h:= \gF_h(\C{P}_h) \cup \bigcup_{j=0}^{k_h+ \hat{\B{k}} +\B{k}+ 2} (\gF_h(S_h)+j).\]

\begin{lem}\label{L:enough-iterates}
For every $h\in \IS_\ga \cup \{Q_\ga\}$ with 
$0 < |\ga| \leq \min\{\ga^*, 1/(2\B{k} +\hat{\B{k}}+4.5)\}$, we have 
\[\C{D}_h \subset \C{D}_h'.\] 
\end{lem}

\begin{proof}
The condition on $\ga$ guarantees that $\langle 1/\ga \rangle \geq 2\Bk+ \hat{\Bk}+4$. 
Hence, $k_h+ \hat{\B{k}} +\B{k}+ 2 \leq k_h+ \langle 1/\ga \rangle -\B{k}-2$, and 
by the definitions of the sets $\C{D}_h$ and $\C{D}_h'$, the inclusion follows.
\end{proof}

\begin{rem}
Although the above lemma easily follows from the definitions and the condition imposed on $\ga$, it has been emphasized to 
make it clear where the high type condition is used in this paper. 
That is, to extend $\gF_h^{-1}$ on $\C{D}_h$, we need that the map $h$ can be iterated at 
least $k_h + \hat{\B{k}} + \B{k}+ 2$ times on $S_h$.
This is crucial if one wishes to prove the main theorems of this article for all rotation numbers using 
an invariant class under a similar renormalization operator. 
Note that the constants $\B{k}$ and $\hat{\B{k}}$ depend only on the class of maps invariant under the renormalization. 
\end{rem}


\subsection{Modified continued fraction and infinitely renormalizable maps}\label{S:continued-fraction--renormalizations}
For irrational $\ga\in \D{R}$, define 
\begin{equation}\label{E:rotations}
\ga_0=d(\ga, \D{Z}), \text{ and } \ga_{i+1}=d(1/\ga_i, \D{Z}), \text{ for } i\geq 0,
\end{equation} 
where $d$ denotes the Euclidean distance on $\D{R}$. 
Then, choose $a_{-1}\in \D{Z}$ with $\ga-a_{-1} \in (-1/2,+1/2)$, and $a_i\in \D{Z}$ with 
$1/\ga_{i}-a_i\in (-1/2, +1/2)$, for $i=0,1,2,\dots$. 
Define $\gep_0=1$ if $\ga-a_{-1} \in (0, 1/2)$ and $\gep_0=-1$ if $\ga-a_{-1}\in (-1/2, 0)$.
Similarly, for $i=0,1,2,\dots$, let 
\begin{equation*}
\gep_{i+1}= 
\begin{cases}
1 & \text{if } \frac{1}{\ga_i}-a_i\in (0, 1/2), \\
-1 & \text{if } \frac{1}{\ga_i}-a_i\in (-1/2, 0).
\end{cases}
\end{equation*}
Note that for all $i\geq 0$, $\ga_i\in (0,1/2)$ and $a_i\geq 2$.  
It follows that $\ga$ is given as the infinite continued fraction 
\[\ga=a_{-1}+\cfrac{\gep_0}{a_0+\cfrac{\gep_1}{a_1+\cfrac{\gep_2}{a_2+\dots}}}.\]

Recall the class of high type numbers $\irr$ defined in the introduction. 
Fix an integer 
\[N\geq 1/\ga^*+1/2.\] 
Using the formula $\ga_{i-1}=1/(a_{i-1}+\gep_i\ga_i)$, $\ga\in \irr$ implies that for all $i\geq 0$ we have $\ga_{i}\in (0,\ga^*]$.
We also need to assume the constant $N$ is large enough so that 
\begin{equation}\label{E:high-type-restriction}
N\geq 2 \Bk+\hat{\Bk}+4.5.
\end{equation}
This guarantees that for every $n \geq 0$, we have $\ga_n \leq 1/(\Bk+\hat{\Bk}+4)$ needed in Lemma~\ref{L:enough-iterates}. 

Let $\ga\in \irr$ and $f_\ga \in \IS_\ga \cup \{Q_\ga\}$. 
Define 
\[
f_0= 
\begin{cases}
f_{\ga} & \text{ if } \gep _0=+1 \\
s\circ f_{\ga} \circ s   &   \text{ if } \gep_0=-1 
\end{cases}
,\]
where $s(z)=\ol{z}$ denotes the complex conjugation.
Now, $f_0$ has asymptotic rotation $\ga_0\in (0,\ga^*]$ at $0$. 
Then, by Theorem~\ref{T:Ino-Shi2}, we may inductively define the sequence of maps 
\[
f_{n+1}= 
\begin{cases}
\rr (f_n) & \text{ if } \gep _n=-1 \\
s \circ \rr(f_n) \circ s   &   \text{ if } \gep_n=+1 
\end{cases}
.\] 
Let $U_n=U_{f_n}$ denote the domain of definition of $f_n$, for $n\geq 0$. 
It follows that, for every $n\geq 1$,  
\[f_n:U_n\to \cc, f_n\in \IS_{\ga_n},   f_n(0)=0, f_n'(0)= e^{2\pi \ga_n \B{i}}.\]

The reason for considering the above notion of continued fractions instead of the standard one is that the 
set of high type numbers in this expansion is strictly bigger than the set of high type numbers in the 
standard expansion. 
It is the nature of near parabolic renormalization that makes this notion of continued fraction more 
suitable to work with. 
\section{Symbolic dynamics near the attractor}\label{S:tower}
\subsection{Changes of coordinates}\label{SS:change-coordinates}
Recall the sequence of maps $f_n$, $n\geq 0$, defined in Section~\ref{S:continued-fraction--renormalizations}. 
In this section we shall define some changes of coordinates between the dynamic planes of these maps. 
Because of the complex conjugation $s$ that appears in the definition of the sequence $f_n$, 
extra care is needed in defining these changes of coordinates. 

For $n\geq 0$, let $\Phi_n=\Phi_{f_n}$ denote the Fatou coordinate of $f_n\colon U_n\ra \cc$ 
defined on the set $\p_n=\p_{f_n}$; see Theorem~\ref{T:Ino-Shi1} and Proposition~\ref{P:petal-geometry} and the definition after them. 

For every $n\geq 0$, let $\C{C}_n$ and $\Csh_n$ denote the corresponding sets for $f_n$ defined in \eqref{E:sector-def} 
(i.e., replace $h$ by $f_n$). 
Denote by $k_n$ the smallest positive integer with  
\[S_n^0=\C{C}_n^{-k_n}\cup (\Csh_n)^{-k_n}\subset 
\{z\in \p_n \mid 1/2 < \Re \Phi_n(z)< a_n -\Bk-1/2\}.\]
By definition, the critical value of $f_n$ is contained in $f_n\co{k_n}(S_n^0)$.
For each $n\geq 0$ we define the set 
\[\C{D}_n:= \gF_n(\C{P}_n) \cup \bigcup_{j=0}^{k_n+\B{k}+\hat{\B{k}}+2} (\gF_n(S_n^0)+j).\]
With $h=f_n$ in Section~\ref{SS:extending-F-coord} and Lemma~\ref{L:enough-iterates}, 
we have holomorphic maps  
\[\gF_n^{-1}: \C{D}_n \to \Dom f_n \setminus \{0\},\]
such that 
\[\gF_n^{-1}(w+1)= f_n \circ \gF_n^{-1}(w), \quad \forall w, w+1 \in \C{D}_n. \]
We denote the corresponding lifts $\gc_{f_n,0}$ by $\gc_{n,0}$. 
Then, note that for $n\geq 1$, 
\begin{equation*}
\forall w \in \C{D}_n, \quad 
\begin{cases}
\ex \circ \gc_{n,0}(w)= \gF_n^{-1}(w), & \tif \gep_n= -1 \\
\ex \circ \gc_{n,0}(w)= s \circ \gF_n^{-1}(w), & \tif \gep_n= +1,
\end{cases}
\end{equation*}
and 
\begin{equation}\label{E:width-of-image-chi}
\gc_{n,0}(\C{D}_n) \subset \{w\in \D{C} \mid  1 \leq \Re w \leq \hat{\B{k}}+2, \Im w > -2\}\subset \gF_{n-1}(\C{P}_{n-1}).
\end{equation}
Each $\gc_{n,0}$ is either holomorphic or anti-holomorphic, depending on the sign of $\gep_n$.
Define the maps $\gc_{n,i}= \gc_{n,0}+i$, for $i \in \D{Z}$. 
The condition \eqref{E:high-type-restriction} on $\ga$ implies that 
for all $n\geq 1$ and all integers $i$ with $0 \leq i \leq a_{n-1}$, we have 
\begin{equation}
\gc_{n,i}: \C{D}_n \to \C{D}_{n-1}.
\end{equation}
Indeed, we show in the next lemma that an stronger form of inclusion holds. 
 
\begin{lem}\label{L:well-contains}
There exists a constant $\gd_0>0$, depending only on the class $\IS$, such that for every 
$n\geq 1$ and every $i$ with $0\leq i \leq a_{n-1}$, we have 
\[\forall w\in \gc_{n,i}(\C{D}_n), B_{\gd_0}(w) \subset \C{D}_{n-1}.\]
\end{lem}

\begin{proof}
By Equation~\eqref{E:width-of-image-chi}, $\gc_{n,0}(\C{D}_n) \subset \C{D}_{n-1}$. 
By Proposition~\ref{P:well-contained-in-domain}, $B_\gd(\gF_n^{-1}(\C{D}_n))$ 
is contained in the domain of $f_n$ and $\ex \circ \gF_{n-1} (S_{n-1}^0)= \Dom f_n$.
This implies that there is a uniform constant $\gd_0$ such that  
\[B_{\gd_0}(\gc_{n,0}(\C{D}_n)) \subset  \gF_{n-1}(S_{n-1}^0)+ \D{Z}.\]
On the other hand, by the first inclusion in Equation~\eqref{E:width-of-image-chi} and the lower 
bound on $k_n$ in Equation~\eqref{E:bounded-remaining}, for all $n\geq 1$, 
all integers $i$ with $0 \leq i \leq a_n$, and for all $w\in \gc_{n,i} (\C{D}_n)$ we must have 
\[
1\leq \Re w \leq \hat{\Bk}+2 + a_n 
\leq \hat{\Bk} + k_n+ a_n  
< (a_n -\Bk -1/2) + k_n +\Bk +\hat{\Bk} +1/2. 
\] 
This finishes the proof of the proposition by making $\ga_0$ less than or equal to $1/2$. 
\end{proof}

For $n \geq 1$, define $\psi_n\colon \p_{n}\ra \p_{n-1}$ as 
\[\psi_n =  \gF_{n-1}^{-1}\circ  \gc_{n,0}   \circ  \gF_n.\]
Each $\psi_n$ is a holomorphic or anti-holomorphic map, depending on the sign of 
$\gep_n$, and extends continuously to $0\in \partial \p_n$ by mapping it to $0$. 
Define the compositions 
\begin{gather*} 
\gY_1= \psi_1: \C{P}_1 \to \C{P}_0, \gY_2=\psi_1 \circ \psi_2 \colon \p_2 \to \p_0, \\
\Psi_n=\psi_1 \circ \psi_2 \circ \dots \circ \psi_n\colon \p_n \to \p_0, \tfor n\geq 3.
\end{gather*}
These are holomorphic or anti-holomorphic maps, depending on the sign of $(-1)^n \gep_1 \gep_2 \dots \gep_n$.

For every $n\geq 0$ and $i\geq 2$, define the sectors 
\begin{gather*}
S_n^1=\psi_{n+1}(S_{n+1}^0)\subset \p_n, \\
S_n^i=\psi_{n+1}\circ \dots \circ\psi_{n+i}(S_{n+i}^0)\subset \p_n, \text{ for } i\geq 2.
\end{gather*}
All these sectors contain $0$ on their boundaries. 
For the readers convenience, the lower index of each map $\gy_n$, $\gY_n$, $\gc_{n,i}$ 
determines the level of its domain of definition, that is, for example, $\gy_n$ is defined on a set 
that is on the dynamic plane of $f_n$.    
Similarly, the set $S_n^i$ is contained in the dynamic plane of $f_n$ (and is at depth $i$).
However, we shall mainly work with $S_i^0$ and $S^i_0$, for $i\geq 0$. 

There are two collections of changes of coordinates $\gy_n$ and $\gY_n$, for $n\geq 1$, 
as well as $\gc_{n,i}$, for $n\geq 1$ and $0\leq i \leq a_n$ that function in parallel. 
The former set of changes of coordinates are more convenient for the combinatorial 
study of the dynamics of the map using the tower of maps $f_n$. 
This is presented in Sections~\ref{SS:orbit-relations} and \ref{SS:petals}.
The latter set of changes of coordinates are more suitable for the analytic aspects of 
the associated problems. 
This analysis appears in Section~\ref{S:geometry-arithmetic}. 
The relations between the two collection are discussed in Section~\ref{SS:lifts-vs-iterates}.  


\subsection{Orbit relations}\label{SS:orbit-relations}
By the definition of renormalization, one time iterating $\rr(h)$ corresponds to several times iterating $h$, 
via the changes of coordinates between the dynamic planes of $h$ and $\rr(h)$. 
This is made more precise in the next lemma. 

\begin{lem}\label{L:renorm}
Let $n\geq 0$ and $z\in\p_n$ be a point with $w=\ex\circ \Phi_n(z)\in \Dom \rr(f_n)$.
There exists an integer $\ell_z$ with  $1\leq \ell_z \leq a_n -\Bk+k_n-1/2$, 
such that
\begin{itemize}
\item the orbit $z,f_{n}(z),f_{n}\co{2}(z),\dots,f_{n}\co{\ell_z}(z)$ is defined, and $f_{n}\co{\ell_z}(z)\in \p_{n}$;
\item $\ex\circ \Phi_{n}(f_{n}\co{\ell_z}(z))= \rr(f_n)(w)$.
\end{itemize}
\end{lem}

\begin{proof}
Since $w\in \Dom \rr(f_{n})$, $\rr(f_n)(w)$ is defined.
By the definition of renormalization, there are 
$\zeta \in \Phi_{n}(S^0_{n})$ and $\zeta' \in \Phi_{n}( \C{C}_{n}\cup\Csh_{n})$, 
such that
\begin{equation*} 
\ex(\zeta)=w,\quad \ex(\zeta')=\rr(f_n)(w),\text { and} \quad \zeta'=\Phi_{n}\circ f_{n}\co{k_{n}}\circ\Phi_{n}^{-1}(\zeta).
\end{equation*}
Since $\ex (\Phi_{n}(z))=w$, there exists an integer $\ell$ with 
\[-k_{n}+1\leq\ell\leq a_n -\Bk-1/2,\] 
such that $\Phi_{n}(z)+\ell=\zeta$.

By the functional equation in Theorem~\ref{T:Ino-Shi1}-d), we have
\begin{align*}
\zeta' &=\Phi_{n}\circ f_{n}\co{k_{n}}\circ\Phi_{n}^{-1}(\zeta)\\
          &=\Phi_{n}\circ f_{n}\co{k_{n}}\circ\Phi_{n}^{-1}(\Phi_{n}(z)+\ell)\\
          &=\Phi_{n}\circ f_{n}\co{k_{n}+\ell}(z).
\end{align*}

Letting $\ell_z=k_{n}+\ell$, we have 
\begin{gather*} 
1 \leq  \ell_z \leq k_{n}+ a_n -\Bk-1/2, \\
f_{n}\co{\ell_z}(z)=\Phi_{n}^{-1}(\zeta') \in \p_{n},\\
\ex\circ\Phi_{n}(f_{n}\co{\ell_z}(z))
=\ex \circ \Phi_{n}(\Phi_{n}^{-1}(\zeta'))
=\ex (\zeta')
=\rr(f_n)(w).
\end{gather*}
\end{proof}

On the practical side, if $\gep_n=-1$, then $\rr(f_n)=f_{n+1}$ and the above lemma holds for $f_{n+1}$ instead of $\rr(f_n)$. 
If $\gep=+1$, then one may replace $\rr(f_n)$ by $f_{n+1}$ and $w$ by $s(w)$ in the above lemma. 

It should be clear that there are many choices for $\ell_z$ in the above lemma. 
We need precise formulas relating the number of iterates on consecutive renormalization levels. 
This may be achieved by imposing some condition on the beginning and termination point of an orbit of $f_n$ that reduces to 
one iterate of $\rr(f_n)$.  
Here, we ask an orbit of $f_n$ to start and terminate in $\psi_{n+1}(\p_{n+1})$. 
Then combining these relations together, for several values of $n$, we relate appropriate 
iterates of $f_0$ to one iterate of $f_n$, through the change of coordinates $\gY_n$. 

Let $R_\ga$ denote the rotation of angle $2 \pi \ga$ about $0$; $R_\ga(z)=e^{2\pi \ga i} z$, 
$z\in \D{C}$.
The \textit{closest return times} of $R_\ga$ are define as the sequence of positive integers 
$q_0< q_1 < q_2< \dots$ as follows.
Let $q_0=1$, $q_1=a_0$, and for $i\geq 1$, $q_{i+1}$ is the smallest integer bigger than 
$q_{i-1}$ such that $|R_\ga\co{q_{i+1}}(1)-1| < |R_\ga\co{q_i}(1)-1|$. 

Define 
\[\p_n'=\{w\in\p_n\mid 0 <\Re \Phi_{n}(w) < \ga_n^{-1} -\Bk-1\}.\]
We have $f_n(\p_n') \subset \p_n$.

\begin{lem}\label{L:one-level}
For every $n\geq 1$ we have 
\begin{itemize}
 \setlength{\itemsep}{.2em}
\item[(a)]for every $w\in \p_n'$, $f_{n-1}\co{a_{n-1}}\circ \psi_n(w)=\psi_n\circ f_n(w)$, and 
\item[(b)]for every $w\in S^0_n$, $f_{n-1}\co{(k_n a_{n-1}+1)}
\circ \psi_n(w)=\psi_n\circ f_n\co{k_n}(w)$. 
\end{itemize}
\end{lem}

\begin{lem}\label{L:conjugacy}
For every $n\geq 1$ we have 
\begin{itemize}
 \setlength{\itemsep}{.2em}
\item[(a)]for every $w\in \p_n'$, $f_0\co{q_n}\circ\Psi_n(w)=\Psi_n \circ f_n(w)$,
\item[(b)]for every $w\in S^0_n$, $f_0\co{(k_nq_n+q_{n-1})}\circ\Psi_n(w)=\Psi_n\circ f_n\co{k_n}(w)$, and
\item[(c)]similarly, for every $m<n-1$, $f_n\colon \p_n' \ra \p_n$ and $f_n\co{k_n}:S_n^0\ra (\C{C}_n\cup\Csh_n)$ 
are conjugate to some appropriate iterates of $f_m$ on the set $ \psi_{m+1} \circ \dots \circ \psi_n(\p_n)$. 
\end{itemize}
\end{lem}

To find the correct number of iterates in the above lemmas, we compare the maps $f_n$ near $0$ 
to the rotations of angle $2\pi \ga_n$ about $0$. 
That is, the relations hold near zero, and hence, must hold on the region where the equations 
are defined. 
For a detail proof of the above lemmas one may refer to \cite{Ch10-I} or \cite{Ch10-II}. 


\subsection{Petals covering the postcritical set}\label{SS:petals}
For every $n\geq 1$, we define the sets 
\[I^n_a=
\bigcup_{i=0}^{k_n+ a_n-\Bk-2} f_0\co{(i q_n)}(S_0^n),\quad 
I^n_b=f_0\co{q_{n-1}}(I^n_a).\]
Note that by Equation~\eqref{E:bounded-remaining} for all $n\geq 1$ we have $k_n\geq 2$.
Set 
\[I^n= I^n_a \cup I^n_b.\]
The set $I^n$ defined through the above mechanism has some crucial dynamical properties 
necessary for our purpose. 
We investigate these in the remaining of this section.

\begin{lem}\label{L:petal}
For every $n\geq 1$, the sets $I^n_a$, $I^n_b$, and $I^n$ are connected subsets of $\D{C}$. 
Moreover, each of $I^n_a$, $I^n_b$, and $I^n$ is bounded by piecewise smooth curves two of 
which land at $0$ at well-defined angles.
\end{lem}

\begin{proof}
By Proposition~\ref{P:landing-angle}, the set $S_n^0$ is bounded by piecewise smooth curves two 
of which landing at $0$ at well-defined angles. 
Moreover, one of these boundary curves is mapped to the other boundary curve by $f_n$. 
Combining with Lemma~\ref{L:conjugacy}, we conclude that 
$S_0^n=\gY_n(S_n^0)$ is bounded by piecewise smooth curves two of which land at $0$ at 
well-defined angles, and one boundary curve is mapped to the other by $f_0\co{q_n}$. 
In particular, each $f_0\co {i q_n}(S_0^n)$ is defined and is a connected set.
Moreover, the consecutive sets $f_0\co {i q_n}(S_0^n)$ and $f_0\co {(i+1) q_n}(S_0^n)$ share a 
boundary curve. 
This implies that $I^n_a$ is a connected set bounded by piecewise smooth curves landing at $0$ at 
well-defined angles. 
Clearly, $I^n_b= f_0\co{q_{n-1}}(I^n_a)$ must enjoy the same properties. 

It remains to show that $I^n$ is a connected set.
There is a point $z$ in the interior of $S_n^0$ (sufficiently close to $0$) and  
an integer $i$ with $k_n \leq i \leq a_n+k_n-\Bk-2$ such that $w=f_n\co{i}(z)\in S_n^0$. 
The points $z'=\gY_n(z)$ and $w'=\gY_n(w)$ belong to $S_0^n$.
Using Lemma~\ref{L:conjugacy}, we conclude that $f_0\co{(q_{n-1} + i q_n)}(z')=w'$. 
Thus, $w'\in S_0^n \subset I^n_a$ and $w' \in f_0\co{q_{n-1}} \circ f_0\co{i q_n}(S_0^n)\subset I^n_b$. 
This implies that $I^n_a$ and $I^n_b$ have non-empty intersection, and therefore, 
their union must be connected.
\end{proof}
 
For $n\geq 1$, define the union 
\[\gU^n= \bigcup_{i=0}^{q_n-1} f_0\co{i}(I^n) \cup \{0\}.\]
The sets $S_0^n$ are bounded by piecewise smooth curves two of which land at zero at well-defined 
angles. 
Comparing $f_0$ with the rotation of angle $2\pi \ga_0$ near $0$, one can verify that  
$\gU^n$ contains a neighborhood of $0$. 

When $f\in \IS_\ga \cup \{Q_\ga\}$ is linearizable at $0$, $\gD(f)$ denotes the Siegel disk of 
$f$ centered at $0$, and when it is not linearizable at $0$ we define $\gD(f)$ as the empty set.  

Some similarly defined union of sectors, denoted by $\gO^n_0$, have been studied in 
\cite{Ch10-I,Ch10-II} in detail. 
Indeed, each $\gU^n$ involves $q_n-1$ more iterates of $S_0^n$ than the set $\gO^n_0$. 
This modification is made to achieve some nice combinatorial features that were not 
available through the sets $\gO_0^n$. 
It is proved in \cite[Propositions 2.4 and 4.9]{Ch10-I} that the sets 
$\gO^n_0$ form a nest of domains shrinking to $\C{PC}(f_0) \cup \gD(f_0)$. 
The same arguments may be repeated here to prove this result for the domains $\gU^n$.
We shall skip repeating these arguments here as they are not the main focus of this paper, but state 
them in the next two propositions for reference.
(Alternatively, one can see that for every $n\geq 2$, $\gU^{n+1} \subseteq \gO^n_0 \subseteq \gU^{n-1}$, 
and therefore, $\cap_{n\geq 1} \gU^n =\cap_{n\geq 1} \gO^n_0$.)

Recall that $\pc(f_j)$ denotes the postcritical set of $f_j$. 

\begin{propo}\label{P:pc-neighbor}
Let $\ga\in \irr$. Then, 
\begin{itemize}
\item[(a)] for every $n\geq 0$, 
\[\pc(f_n) \subset \bigcup_{i=0}^{k_n+ a_n -\Bk-2} f_n\co{i}(S_n^0) \cup \{0\};\] 
\item[(b)] for every $n\geq 1$, 
\[\pc(f_0) \ci \gU^n.\]
\end{itemize}
\end{propo}

\begin{propo}\label{P:Siegel-disk-enclosed}
Assume that $f\in \IS_\ga \cup \{Q_\ga\}$ in linearizable at $0$ and $\ga\in \irr$. 
Then, for every $n_0\geq 1$ we have $\cap_{n \geq n_0} \gU^n= \gD(f) \cup \C{PC}(f)$ and 
$\partial \gD(f) \subset \C{PC}(f)$.
\end{propo}

\begin{propo}\label{P:invariance-of-pc}
For every $n\geq 1$ we have $\pc(f_0) \cap \gY_n(\p_n) = \gY_n (\pc(f_n) \cap \p_n)$.
\end{propo}

\begin{proof}
Recall that $\cv_n$ denotes the critical value of  $f_n$ and $\gF_n(\cv_n)=1$.
As $\ex(1)=\cv_n$ for all $n$, it follows from the definition of $\gy_n$ that there is a 
non-negative integer $j_n$ with $\gy_n(\cv_n)= f_{n-1}\co{j_n}(\cv_{n-1})$. 

First we show that for every $n\geq 1$ we have 
\begin{equation*} 
\C{PC}(f_{n-1}) \cap \gy_n(\C{P}_n)= \gy_n(\C{PC}(f_n) \cap \C{P}_n).
\end{equation*}
Fix $z \in \gy_n(\C{P}_n)$ and $w \in \C{P}_n$ with $\gy_n(w)=z$. 
It is enough to show that $z\in \C{PC}(f_{n-1})$ if and only if $w\in \C{PC}(f_n)$. 

Assume that $w\in \C{PC}(f_n)$.  
As $\C{P}_n$ is an open set, there is an increasing sequence of positive integers $n_i$, for $i\geq 0$, 
such that the iterates $f_n\co{n_i}(\cv_n)$ belong to $\C{P}_n$ and converge to $w$. 
Then, one infers from Lemma~\ref{L:conjugacy} that there is an increasing sequence of positive integers 
$m_i$ such that $f_{n-1}\co{m_i}(\gy_n(\cv_n))=\gy_n(f_n\co{n_i}(\cv_n))$.
Thus, $f_{n-1}\co{(j_n+m_i)}(\cv_{n-1})=f_{n-1}\co{m_i}(\gy_n(\cv_n))$ converges to $z$. 
That is, $z\in \C{PC}(f_{n-1})$. 

Let $x_0=cv_n, x_1, x_2, \dots$ denote the (ordered) points in the orbit of $\cv_n$ that 
are contained in $\C{P}_n$.
Define the sequence of positive integers $l_i$, for $i\geq 0$, so that $x_{i+1}= f_n\co{l_i}(x_i)$.
By Proposition~\ref{P:pc-neighbor}-a), for every $i\geq 0$, either $l_i=1$ or $2\leq l_i \leq k_n$.
For every $i$ with $l_i=1$, $\gy_n(x_{i+1})=f_{n-1}\co{a_{n-1}}(\gy_n(x_n))$ and by the definition of renormalization, $a_{n-1}$ is the smallest positive integer $s$ with $f_{n-1}\co{s}(\gy_n(x_n))\in \gy_n(\C{P}_n)$. 
That is, all intermediate iterates are outside of $\gy_n(\C{P}_n)$. 
Similarly, when some $l_i\geq 2$ then the intermediate iterates $f_n\co{1}(x_i)$, \dots, 
$f_n\co{l_i-1}(x_i)$ are outside of $\C{P}_n$. 
This implies that $\gy_n(x_{i+1})=f_{n-1}\co{(l_ia_{n-1}+1)}(\gy_n(x_i))\in \gy_n(\C{P}_n)$, 
and  $l_i a_{n-1}+1$ is the smallest positive integer $s$ with $f_{n-1}\co{s}(\gy_n(x_i))\in\gy_n(\C{P}_n)$. 

Now assume that $z$ belongs to $\C{PC}(f_{n-1})$. 
As $\gy_n(\C{P}_n)$ is an open neighborhood of $z$, by the above paragraph, there is an increasing 
sequence of integers $n_0 < n_1 < n_2 < \dots$ such that $\gy_{n}(x_{n_i})$ converge to $z$. 
Therefore, the sequence $x_{n_i}$ converges to $w$. 
That is, $w\in \C{PC}(f_n)$. 

The equality in the proposition holds for $n=1$ by the above equation. 
Assume that we have $\pc(f_0) \cap \gY_{n-1}(\C{P}_{n-1}) = \gY_{n-1} (\pc(f_{n-1}) \cap \C{P}_{n-1})$ 
(induction hypothesis for $n-1$). 
As each $\gY_n: \C{P}_n \to \C{P}_0$ is univalent and $\gY_n= \gY_{n-1} \circ \gy_n$, we obtain  
\begin{align*}
\pc(f_0) \cap \gY_n(\p_n) 
&= \pc(f_0) \cap \big (\gY_{n-1}(\C{P}_{n-1})  \cap \gY_{n-1}(\gy_n(\C{P}_n))\big) && \\
&= \big (\pc(f_0) \cap \gY_{n-1}(\C{P}_{n-1}) \big)  \cap \gY_{n-1}(\gy_n(\C{P}_n))&& \\
&=\gY_{n-1} (\pc(f_{n-1})\cap\C{P}_{n-1})\cap\gY_{n-1}(\gy_n(\C{P}_n)) &&\text{(induction hypothesis)}\\
&= \gY_{n-1} ( \pc(f_{n-1}) \cap \gy_n(\C{P}_n)) &&  \\
&=\gY_{n-1} (\gy_n(\C{PC}(f_n) \cap \C{P}_n )) && \text{(above equation)} \\
&= \gY_n (\pc(f_n) \cap \p_n)
\end{align*}
This finishes the proof of the proposition. 
\end{proof}

\begin{lem}\label{L:return-on-tail}
For every $n\geq 1$, $f_0\co{q_n}( \pc(f_0) \cap I^n_b) \subset I^n$. 
\end{lem}

\begin{proof}
Fix $z \in \pc(f_0) \cap I^n_b$. 
By the definition of $I_b^n$, $z$ belongs to $f_0\co{q_{n-1}}  \circ f_0\co{i q_n}(S_0^n)$, for some 
$i$ with $0\leq i\leq k_n+ a_n -\Bk-2$. 
We consider two cases below.

If $i < k_n+ a_n-\Bk-2$, then 
$f_0\co{q_n}(z)\in f_0\co{q_{n-1}}  \circ f_0\co{(i+1) q_n}(S_0^n)\subset I^n_b \subset I^n$.  

If $i= k_n+ a_n -\Bk-2$, by Lemma~\ref{L:conjugacy}, $z$ belongs to  
\begin{align*}
f_0\co{q_{n-1}}\circ f_0\co{(k_n+ a_n -\Bk-2)q_n}(S_0^n)
& = f_0\co{(a_n - \Bk-2)q_n } \circ f_0\co{k_n q_n+ q_{n-1}} (S_0^n)\\
&= f_0\co{(a_n-\Bk -2)q_n} \circ \gY_n \circ f_n\co{k_n}(S_n^0)\\
&=  \gY_n \circ f_n\co{(k_n+ a_n - \Bk-2)}(S_n^0).
\end{align*}
On the other hand, since $z$ belongs to $\pc(f_0)$ and 
$f_n\co{(k_n+ a_n - \Bk-2)}(S_n^0)$ is contained in $\p_n$, 
by Proposition~\ref{P:invariance-of-pc} there is 
\[w\in f_n\co{(k_n+ a_n - \Bk-2)}(S_n^0) \cap \pc(f_n), \text{ with } \gY_n(w)=z.\] 
In particular, 
\[\Re \gF_n(w)\in [ a_n - \Bk-3/2, a_n -\Bk-1/2].\]
As $w\in \pc(f_n)$, by Proposition~\ref{P:pc-neighbor}, its forward orbit remains in 
the union of the sectors $ f_n\co{i}(S_n^0)$, for $0\leq i \leq k_n+ a_n - \Bk-2$.  
Hence, by the mapping property of $f_n$ on these sectors in the definition of the renormalization, 
we must have $w\in \cup_{j=0}^{k_n-1} f_n\co{j}(S_n^0) \cap \p_n$.
Then, by Lemma~\ref{L:conjugacy}, this implies that 
$\gY_n(w)\in \cup_{j=0}^{k_n-1} f_0\co{(j q_n)}(S_0^n)$. 
Therefore, 
\[f_0\co{q_n}(z)=
f_0\co{q_n}(\gY_n(w))\in \cup_{j=1}^{k_n} f_0\co{(j q_n)}(S_0^n) \subset I_a^n \subset I^n.\] 
This finishes the proof of the lemma. 
\end{proof}

We also need the following lemma.

\begin{lem}\label{L:beginning-related-to-end}
For every $n\geq 1$ and every $i$ with $0\leq i\leq q_{n-1}-1$ we have 
\[f_0\co{i} (f_0\co{(mq_n)} (S_0^n)) \subset f_0\co{(q_n-q_{n-1}+i)}(I^n_b), 
\text{ where } m=k_n+ a_n  -\Bk -2.\]
\end{lem}

\begin{proof}
It is enough to prove the inclusion for $i=0$.
We may rearranging the iterates as in
\[f_0\co{(mq_n)}(S_0^n)=f_0\co{(q_n-q_{n-1})} (f_0\co{q_{n-1}} \circ f_0\co{((m-1)q_n)} (S_0^n)).\]
On the other hand, as $f_0\co{((m-1)q_n)} (S_0^n) \subset I^n_a$, we have 
$f_0\co{q_{n-1}} \circ f_0\co{((m-1)q_n)} (S_0^n) \subset I^n_b$. 
\end{proof}

The next proposition is the ultimate statement in this section we need for the proofs of the main 
results of this paper. 

\begin{propo}\label{P:orbits-organized}
For every $\ga\in \irr$ and every integer $n\geq 1$ we have the following.  
For every non-zero $z\in \pc(f_0)$ there is an integer $\ell$ with $0 \leq \ell \leq q_n-1$ such that the 
following hold: 
\begin{itemize}
\item[a)] for all $j$ with $0\leq j \leq q_n-\ell-1$ we have $f_0\co{j}(z) \in f_0\co{(\ell+j)}(I^n)$;
\item[b)] for all $j$ with $q_n-\ell \leq  j \leq q_n-1$ we have $f_0\co{j}(z) \in f_0\co{(j-q_n+\ell)}(I^n)$.
\end{itemize}
\end{propo}

\begin{proof}
By Proposition~\ref{P:pc-neighbor}, $z$ belongs to $\gU^n$. 
We may rewriting this set as 
\begin{align*}
\gU_n \setminus \{0\}=\bigcup_{i=0}^{q_n-1} f_0\co{i}  (I^n_a\cup I^n_b) 
&=\bigcup_{i=0}^{q_n-1} f_0\co{i}  (I^n_a) \cup  \bigcup_{i=0}^{q_n-1} f_0\co{i} (I^n_b) \\
&=\bigcup_{i=0}^{q_{n-1}-1} f_0\co{i}  (I^n_a) \cup  \bigcup_{i=0}^{q_n-1} f_0\co{i} (I^n_b) \\
&=I^n_a \cup \bigcup_{i=1}^{q_{n-1}-1} f_0\co{i}  (I^n_a) \cup  \bigcup_{i=0}^{q_n-1} f_0\co{i} (I^n_b)
\end{align*}
Now, we consider four cases below. 

\medskip

1) If $z\in I^n_a$, we let $\ell=0$. 
Here, $z\in I^n_a\subset I^n$, and therefore, we have part a) in the proposition. 
There is no $j$ satisfying b).

\medskip

2) If $z\in f_0\co{i}(I^n_b)$, for some $i$ with $0\leq i \leq q_n-1$, we let $\ell=i$. 
Then, $z\in f_0\co{\ell}(I^n_b)\subset f_0\co{\ell}(I^n)$, and hence, a) holds.  
On the other hand, since $z\in f_0\co{\ell}(I^n_b)$, by Lemma~\ref{L:return-on-tail}, 
$f_0\co{(q_n-\ell)}(z)\in f_0\co{q_n}(I^n_b)\subset I^n$. 
This implies the statement b) in the proposition.

\medskip

It remains to prove the proposition for $z\in  \cup_{i=1}^{q_{n-1}-1} f_0\co{i}  (I^n_a)$.
Fix $i$ with $z\in  f_0\co{i}  (I^n_a)$. 
We treat this in two cases below. 

\medskip

3) If $z\in f_0\co{i}(f_0 \co{(m q_n)}(S_0^n))$, for some $m$ with 
$0\leq m \leq k_n+ a_n  -\Bk -3$. 
With $\ell=i$, we have $z\in f_0\co{\ell}(I^n_a)\subset f_0\co{\ell}(I^n)$. 
This implies a). 

On the other hand, 
\[f_0\co{(q_n-\ell)}(z)\in f_0\co{q_n} (f_0\co{(mq_n)}(S_0^n))=f_0\co{((m+1)q_n)}(S_0^n)\subset I^n_a \subset I^n.\] 
Part b) follows from the above inclusion. 

\medskip

4) If $z\in f_0\co{i}(f_0 \co{(m q_n)}(S_0^n))$, for $m=k_n+ a_n -\Bk -2$. 
We let $\ell= q_n-q_{n-1}+i$. 
By Lemma~\ref{L:beginning-related-to-end}, $z\in f_0\co{\ell}(I^n_b)\subset f_0\co{\ell}(I^n)$, 
and hence, a) holds. 
On the other hand, by Lemma~\ref{L:return-on-tail}, $z\in f_0\co{\ell}(I^n_b)$ implies that 
\[f_0\co{(q_n-\ell)}(z)\in f_0\co{q_n}(I^n_b)\subset I^n.\]
This implies the statement b) in the proposition.
\end{proof}

\subsection{Lifts versus iterates} \label{SS:lifts-vs-iterates}

In this section we give an alternative definition of the sectors $I^n$, and its forward iterates, 
in terms of the lifts $\gc_{n,i}$. 
We give an alternative definition in Proposition~\ref{P:lifts-vs-iterates} which makes later 
arguments simpler. 

\begin{lem}\label{L:the-first-lift}
For every $n\geq 1$ and every integer $i$ with $0 \leq i \leq k_n+ a_n - \B{k}-2$ we have 
\begin{itemize}
\item[a)] 
\[\gF_{n-1} \circ f_{n-1}\co{(i a_{n-1} )}(S_{n-1}^1) 
\subset \{w \in \D{C} \mid -2 < \Im w , 0 \leq \Re w \leq  \hat{\B{k}}+2\},\] 
\item[b)]
\[\gF_{n-1} \circ f_{n-1}\co{(i a_{n-1} +1)}(S_{n-1}^1) 
\subset \{w\in \D{C} \mid  -2 < \Im w , 1 \leq \Re w \leq  \hat{\B{k}}+3\};\] 
\item[c)] 
\[f_0\co{(i q_n)} (S^n_0) = 
\gF_0^{-1} \circ \gc_{1,0} \circ \gc_{2,0} \circ \dots \circ \gc_{n-1,0} \circ \gF_{n-1} \circ 
f_{n-1}\co{(i a_{n-1})}(S_{n-1}^1), \] 
\item[d)]
\[f_0\co{(i q_n+q_{n-1})} (S^n_0) = 
\gF_0^{-1} \circ \gc_{1,0} \circ \gc_{2,0} \circ \dots \circ \gc_{n-1,0}\circ \gF_{n-1} \circ 
f_{n-1}\co{(i a_{n-1} +1)}(S_{n-1}^1).\]
\end{itemize}
\end{lem}

\begin{proof}
{\em Part a)}
Fix $n\geq 1$. 
We consider two cases: 

\medskip

{\em Case 1:} Assume that $i$ satisfies $0 \leq i \leq k_n$.

\noindent First, by an inductive argument we show that for all $w \in S^0_n$ we have 
\begin{equation}\label{E:extended-conjugacy}
f_{n-1}\co{(i a_{n-1})} \circ \gy_n (w)=\gF_{n-1}^{-1} \circ \gc_{n,0} (\gF_n(w)+i).
\end{equation}
Note that for all $w \in S_n^0$, $\gF_n(w)+i \in \C{D}_n$ and hence the right hand side of the 
above equation is defined. 
Moreover, the right hand side of the equation is either a holomorphic or anti-holomorphic function 
of $w$. 

By definition, we have $\gy_n=\gF_{n-1}^{-1} \circ \gc_{n,0} \circ \gF_n$, and hence the equation holds for $i=0$. 

Assume that Equation~\eqref{E:extended-conjugacy} holds for all integers less than or equal to 
some $0 \leq i < k_n$. 
We wish to show that it holds for $i+1$.
For $|w|$ small enough, the left hand side of the equation is defined for $i+1$. 
Moreover, for $w$ on a smooth curve on $\partial S_n^0$ landing at $0$ there is $w'$ on $\partial S^0_n$ with $f_n(w')=w$. 
By Lemma~\ref{L:one-level} for $w'$, we obtain 
\begin{align*} 
f_{n-1}\co{((i+1) a_{n-1})} \circ \gy_n (w') &= f_{n-1}\co{(i a_{n-1})} \circ \gy_n (w) \\
&= \gF_{n-1}^{-1} \circ \gc_{n,0} (\gF_n(w)+i)= \gF_{n-1}^{-1} \circ \gc_{n,0} (\gF_n(w')+1+i).
\end{align*}
That is, the equation holds for $i+1$ on a curve landing at $0$ on $S_n^0$. 
Hence, by the uniqueness of the analytic continuation, it must hold for $w \in S^0_n$ close to $0$. 
On the other hand, since by the open mapping property of holomorphic and anti-holomorphic maps, it follows that 
the set of points on which the equality holds forms an open and closed subset of $S_n^0$. 
This implies that the equality must hold on $S_n^0$. 

Equation~\eqref{E:extended-conjugacy} and $S_{n-1}^1= \gy_n(S_n^0)$, implies that 
$\gF_{n-1} \circ f_{n-1}\co{(i a_{n-1})}(S_{n-1}^1)$ is contained in $\gc_{n,0}(\C{D}_n)$. 
Combining with the inclusion in Equation~\eqref{E:width-of-image-chi}, we conclude the inclusion in 
a) for these values of $i$. 

\medskip

{\em Case 2:} Assume that $i$ satisfies $k_n \leq i \leq k_n+ a_n - \B{k}-2$. 
For all $i$ with $k_n \leq i \leq k_n+ a_n - \B{k}-2$, $f_n\co{i}(w) \in \C{P}_n$, and by 
Lemma~\ref{L:one-level}, we have 
\[f_{n-1}\co{(i a_{n-1} +1)} (\gy_n(w))= \gy_n (f_n\co{i}(w))=\gF_{n-1}^{-1} \circ \gc_{n-1,0} \circ \gF_n (f_n\co{i}(w)).\]
Therefore, it follows from Equation~\eqref{E:width-of-image-chi} that 
\begin{multline*}
\gF_{n -1}(f_{n-1}\co{(i a_{n-1} +1)} (\gy_n(w))) 
\subset \gc_{n-1,0}(\gF_n(\C{P}_n)) \\
\subset \gc_{n-1,0}(\C{D}_n)
\subset \{w \in \D{C} \mid -2 < \Im w , 1 \leq \Re w \leq  \hat{\B{k}}+2\}
\end{multline*}
Hence, the relation $\gF_{n-1} \circ f_{n-1}(z)= \gF_{n_1}(z)+1$ implies the inclusion in a) for 
these values of $i$ (i.e.\ $1$ is replaced by $0$).

\medskip

{\em Part b)} 
By the restriction in Equation~\eqref{E:high-type-restriction} on the rotation, 
we have $\hat{\B{k}}+3 \leq \ga_n^{-1}-\B{k} -1$.  
This implies that the sets $f_{n-1}\co{(i a_{n-1})}(S_{n-1}^1)$ and 
$f_{n-1}\co{(i a_{n-1} +1)}(S_{n-1}^1)$ are contained in $\C{P}_{n-1}'$. 
As $\gF_{n-1}$ conjugates $f_{n-1}$ to the translation by one on $\C{P}_{n-1}'$, 
the inclusion in b) follows from the one in a). 

\medskip

{\em Part c)}
Recall the changes of coordinate $\gy_n=\gF_{n-1}\circ \gc_{n,0} \circ  \gF_n$, and their 
compositions $\gY_{n}$, for $n\geq 1$. 
It follows from the inclusion in part a) that for every $w\in S^1_{n-1}$, 
$\gF_{n-1} \circ f_{n-1}\co{(i a_{n-1})}(w)$ is contained in the domain of $\gc_{n-1,0}$. 
Thus, $\gY_{n-1} \circ f_{n-1} \co{(i a_{n-1})}(w)$ is defined. 
On the other hand, by Lemma~\ref{L:conjugacy}-b), 
for every $w\in S_0^n$, $f_0\co{k_n q_n + q_{n-1}}(w)$ is defined and 
$\Re f_n\co{k_n}(w) \in [1/2,3/2]$. 
In particular, this implies that for every $i$ with $0 \leq i \leq k_n$, 
$f_0\co{(i q_n)}(w)$ is defined. 
Then, one infers from Lemma~\ref{L:conjugacy}-a), that for all $w\in S_0^n$ and all $i$ with 
$k_n \leq i \leq a_n+k_n-\Bk -2$, $f_0\co{(i q_n)}(w)$ is defined. 

We want to show that the equality 
\[f_0\co{(i q_n)}\circ \gY_n(w)= \gY_{n-1} \circ f_{n-1}\co{(i a_{n-1})} \circ \gy_n(w),  w\in S_n^0,\]
holds for all $i$ in the lemma.  
The set $S_n^0$ is bounded by piecewise smooth curves two of which land at $0$ where 
one boundary curve is mapped to another by $f_n$. 
By Lemma~\ref{L:one-level}, this implies that $S_{n-1}^1$ is bounded by piecewise smooth curves, 
one of which is mapped to another by $f_{n-1}^{a_{n-1}}$. 
Similarly, by Lemma~\ref{L:conjugacy} the set $S_0^n$ is bounded by piecewise smooth curves, 
one of which is mapped to another by $f_0\co{q_n}$. 
These imply that the above equation, for each $i$, is valid on a boundary curve of $S_n^0$. 
These imply that the above equation holds on a curve on the boundary of $S_n^0$.
By the uniqueness of holomorphic and anti-holomorphic continuations, we conclude that the 
above equation must hold on the connected set $S_n^0$. 
This finishes the proof of the lemma. 

\medskip

{\em Part d)}
By the inclusion in part b), the set $\gF_{n-1} \circ f_{n-1}\co{(i a_{n-1} +1)}(S_{n-1}^1)$ is 
contained in the domain of $\gc_{n-1,0}$. 
Thus, the right hand side of the equation is defined. 
On the other hand, since $f_{n-1}\co{(i a_{n-1} +1)}(S_{n-1}^1) $ is contained in $\C{P}_{n-1}'$, 
the equality follows from the equality in part c) and the conjugacy relation in 
Lemma~\ref{L:conjugacy}-a). 
\end{proof}

Define the sets 
\[J_{n-1}= \bigcup_{j=0}^{1}  \bigcup_{i=0}^{k_n+ a_n -\B{k}-2} f_{n-1}\co{(i a_{n-1} +j)}(S_{n-1}^1), 
n\geq 1.\]
By Lemma~\ref{L:the-first-lift}, for all $n\geq 1$, we have
\begin{equation}\label{E:width-of-J}
\gF_{n-1}(J_{n-1}) \subset \{w\in \D{C} \mid  -2 < \Im w, 1 \leq \Re w \leq \hat{\B{k}}+3\}  \subset \C{D}_{n-1},
\end{equation}
and  
\[I^n=
\gF_0^{-1} \circ \gc_{1,0} \circ \gc_{2,0} \circ \dots \circ \gc_{n-1,0} \circ \gF_{n-1} ( J_{n-1})
= \gY_{n-1}(J_{n-1}).\]

\begin{propo}\label{P:lifts-vs-iterates}
For every $n\geq 1$ and every $i$ with $0 \leq i \leq q_n$, there are integers $i_j$ 
with $0 \leq i_j \leq a_{j-1}$, for $1 \leq j \leq n$, such that 
\begin{equation}\label{E:lifts-vs-iterates}
\forall w\in J_{n-1},  f_0\co{i}(\gY_{n-1}(w)) = \gF_0^{-1} \circ \gc_{1,i_1} \circ \gc_{2,i_2} \circ \dots \circ \gc_{n-1,i_{n-1}} (\gF_{n-1} (w)+i_n),
\end{equation}
in particular, 
\[f_0\co{i}(I^n) =
\gF_0^{-1} \circ \gc_{1,i_1} \circ \gc_{2,i_2} \circ \dots \circ \gc_{n-1,i_{n-1}} (J_{n-1}+i_n).\]
\end{propo}

\begin{proof}
The latter part of the proposition follows from the former part.
For the former part, first we need to show that both sides of Equation~\eqref{E:lifts-vs-iterates} 
are defined.
Fix $n\geq 1$. 
By Equations~\eqref{E:width-of-J} and \eqref{E:high-type-restriction}, for all integers $i_n$ with 
$0\leq i_n \leq  a_{n-1}$, we have 
\[\gF_{n-1} (J_{n-1}) +i_n \subset \C{D}_{n-1}.\]
Hence, $\gc_{n-1, i_{n-1}}$ is defined on $\gF_{n-1} (J_{n-1}) +i_n$.
Now, we have  
\[\gc_{n-1, i_{n-1}}(\gF_{n-1} (J_{n-1}) +i_n) \subset \gc_{n-1, i_{n-1}}(\C{D}_{n-1}).\] 
Combining the above inclusion with Equation~\eqref{E:width-of-image-chi}, we conclude that 
for all integers $i_{n-1}$ with $0\leq i_{n-1}\leq a_{n-2}$, the set 
$\gc_{n-1, i_{n-1}}(\gF_{n-1} (J_{n-1}) +i_n)$ is contained in $\C{D}_{n-2}$. 
Therefore, the expression 
$\gc_{n-2, i_{n-2}} \circ \gc_{n-1, i_{n-1}}(\gF_{n-1}(J_{n-1})+i_n)$ is defined. 
Continuing this argument, one infers that the right hand side of Equation~\eqref{E:lifts-vs-iterates} is 
defined. 
On the other hand, since $f_{n-1}$ may be iterated at least $a_{n-1}$ times on $J_{n-1}$ and 
$J_{n-1}$ is contained in $\C{P}_{n-1}$, Lemma~\ref{L:conjugacy} implies that 
$f_0$ may be iterated at least $a_{n-1} q_{n-1}+q_{n-2}=q_n$ times on $\gY_{n-1}(J_{n-1})=I^n$. 
Thus, for all such $i$, the left hand side of Equation~\eqref{E:lifts-vs-iterates} is defined.

By the conjugacy property of the coordinates $\gF_{m}^{-1}$ on $\C{D}_{m}$, 
each composition on the right hand side of Equation~\eqref{E:lifts-vs-iterates} corresponds to 
some non-negative iterate of $f_0$ on the left hand side of the equation.
Comparing $f_0$ with the rotation $R_{\ga_0}$ near zero, one can see that for all given $i$ 
there are integers $i_j$ as in the proposition such that the equality holds near $0$. 
By the uniqueness of the analytic continuation, they must hold on the connected set $J_{n-1}$.
\end{proof}

\begin{rem}
All the lemmas and propositions in Sections~\ref{SS:petals} and \ref{SS:lifts-vs-iterates} hold for 
all maps $f\in \IS_{\ga}$ with $\ga\in \irr$.
That is, it is not necessary that $f$ is a quadratic polynomial.
\end{rem}

\section{Geometry of the renormalization tower and the arithmetic of $\ga$}\label{S:geometry-arithmetic}
\subsection{Nearby sectors}\label{SS:nearby-orbits}.
Recall the sequence of numbers $\ga_i$, $i\geq 0$, defined in \refE{E:rotations}. 
Let $\gb_0=1$, and $\gb_k=\gP_{i=1}^k \ga_i$, $k\geq 1$.
The irrational number $\ga$ is called a Brjuno number if 
\[\sum_{j=0}^{+\infty}  \gb_{j-1} \log \ga_j^{-1} < +\infty.\] 
An equivalent characterization of the Brjuno numbers in terms of the best rational approximants of 
$\ga$ is $\sum_{j=0}^{+\infty} (q_j^{-1} \log q_{j+1})< +\infty$. 
By Siegel-Brjuno-Yoccoz theorem \cite{Sie42,Brj71,Yoc95} the quadratic map $Q_\ga$ is linearizable 
at $0$ if and only if $\ga$ is a Brjuno number. 
This optimality result has been extended to the maps in $\IS_\ga$, $\ga\in \irr$, in \cite{Ch10-I}. 
See the remark after Proposition~\ref{P:siegel-sectors}.

In this section we analyze the sizes of the sectors $f_0\co{i}(I^n)$, $0\leq i\leq q_{n-1}$, defined in 
Section~\ref{SS:petals}.
Our ultimate goal is to prove the following two propositions.  

For $\gd>0$ let $B_\gd(0)$ denote the open disk of radius $\gd$ centered at $0$.
For every $n\geq 1$ and $\gd>0$ define the set
\[G(n,\gd)= \{ i\in \D{Z}   \mid 0\leq i\leq q_n-1, f_0\co{i} (I^n) \subset B_\gd(0) \}.\]
Let $|X|$ denote the cardinality of a given set $X$.

\begin{propo}\label{P:cremer-sectors}
Let $\ga$ be a non-Brjuno number in $\irr$. 
Then, for every $\gd>0$ we have 
\[\limsup_{n\to\infty}  \frac{|G(n,\gd)|}{q_n}=1.\]
\end{propo}
Indeed we state and prove an stronger statement in Proposition~\ref{P:cremer-sectors-stronger}.
However, as we show in Section~\ref{S:Unique-Ergodicity}, the above statement is enough to derive 
the unique ergodicity (stated in the introduction) for non-Brjuno values of $\ga$.

For Brjuno values of $\ga$, let $\gD(f_i)$ denote the Siegel disk of $f_i$, $i\geq 0$. 
The $\gd$ neighborhood of $\gD(f_0)$ is denoted by $B_\gd(\gD(f_0))$.
Given $n\in \D{N}$ and $\gd>0$, define
\[H(n,\gd)=\{ i \in \D{Z} \mid 0 \leq i \leq q_n-1, f_0\co{i} (I^n) \subset B_\gd(\gD(f_0)\}.\]

\begin{propo}\label{P:siegel-sectors}
For every Brjuno $\ga\in \irr$  we have the following:
\begin{itemize}
\item[(a)] For every $\gd>0$ 
\[\limsup_{n\to\infty}  \frac{|H(n,\gd)|}{q_n}=1.\]
\item[(b)] For every $\gep>0$ there are $n_0\in \D{Z}$ and $\gd'>0$ such that for every $n\geq n_0$ and every $\gd< \gd_0$ 
we have 
\[\forall k\in H(n, \gd), \quad \diam (f_0\co{k}(I^n) \setminus \gD(f_0)) \leq \gep.\]
\end{itemize} 
\end{propo}
When $\gd$ is small, the sets $H(n,\gd)$ may be empty for small values of $n$, but eventually these sets 
become non-empty by the first part of the above proposition. 

The remaining of this section is devoted to the proofs of the above propositions. 
The argument has two flavors: an arithmetic part and an analytic part.
Readers interested in the proofs of the unique ergodicity using the above propositions may safely 
skip the remaining of this section and go directly to Section~\ref{S:Unique-Ergodicity}.  

\begin{rem}
In \cite{Ch10-I} it is proved that for every $\gd>0$ the sets $G(n,\gd)$ and $H(n,\gd)$ are non-empty 
for large values of $n$.
The analysis presented in this paper uses an improved distortion estimate on the Fatou coordinates 
$\gF_n$ that has been established 
in \cite{Ch10-II}, but was not available at the time of writing \cite{Ch10-I}. 
\end{rem}


\subsection{An arithmetic lemma}\label{SS:arithmetic}
Consider the sequence of numbers $\ga_i$, $i\geq 0$, and let 
\[\gb_0=1, \gb_k=\gP_{i=1}^k \ga_i, k\geq 1.\]
Fix an arbitrary $B\in \D{R}$ and define the sequence of numbers 
\begin{equation}\label{E:bisequence}
\begin{aligned}
&B_{k,k}=-2 , \qquad k=0,1,2,\dots ; \\
&B_{k,i-1}=\ga_{i} B_{k,i}+\log \ga_i^{-1}-B, \qquad 1\leq i\leq k.
\end{aligned}
\end{equation}
For $T \geq 0$ and integers $l \geq 0$ define
\[\C{L}(\ga,T,l)= \{k\in \{l+1, l+2, \dots\} \mid B_{k,i}\geq T\ga_{i}^{-1},\, \tfor  l< i< k\}.\]
By definition, $l+1\in \C{L}(\ga, T, l)$ as there is no condition to be satisfied. 
Hence, every $\C{L}(\ga, T, l)$ is a non-empty set. 
 
\begin{lem}\label{L:good-levels}
For every irrational $\ga$ and every $B\geq 0$ we have,
\begin{itemize}
\item[(a)] for every $k\geq 0$
\[B_{k,0}\leq \sum_{i=1}^{k}\gb_{i-1}\log\ga_i^{-1}  \leq B_{k,0}+2B+2 ;\] 
\item[(b)] if $\ga$ is non-Brjuno, for every $T\in\D{R}$, the set $\C{L}(\ga, T, 0)$ has infinite cardinality. 

\item[(c)] if $\ga$ is a Brjuno number such that for some $T\geq 0$ and every $l \geq 0$, 
$\C{L}(\ga,T,l)$ has finite cardinality, then 
\[\liminf_{j\to \infty} \lim_{m\to \infty} B_{m,j}-T\ga_j^{-1}< + \infty.\]
\end{itemize}
\end{lem}

Items (a) and (b) of the above lemma, when one starts with $B_{k,k}=0$ instead of $-2$ 
in \refE{E:bisequence} and $B=T$, appear in \cite[Section 1.6]{Yoc95}. 
The same argument works here as well. 
For the sake of completeness, and their use in part (c), we present a proof of them here.  

\begin{proof}
{\em Part (a):}  
First, by an inverse induction we see that for every $k\geq 1$ 
\begin{equation}\label{E:arithmetic-height}
B_{k,i-1}= -2\gb_{i-1}^{-1}\gb_k+\gb_{i-1}^{-1} \sum_{j=i}^{k}\gb_{j-1}(\log\ga_j^{-1}-B),\quad
 1\leq i\leq k.
\end{equation}
That is, for $i=k$ the formula becomes 
\[B_{k,k-1}=-2 \gb_{k-1}^{-1} \gb_k+ \gb_{k-1}^{-1} \gb_{k-1} (\log \ga_{k}^{-1}-B)=
 -2 \ga_k+ \log \ga_k^{-1}-B, \] 
which is valid by Equation~\eqref{E:bisequence}. 
Now assume that Equation~\eqref{E:arithmetic-height} holds for $i=m+1$. 
Then, by Equation~\eqref{E:bisequence}, we have   
\begin{align*}
B_{k,m-1}&=\ga_m B_{k,m}+ \log \ga_m^{-1} -B \\    
&= \ga_m (-2 \gb_{m}^{-1} \gb_k + \gb_m^{-1} \sum_{j=m+1}^{k}\gb_{j-1}(\log\ga_j^{-1}-B))+\log \ga_m^{-1} -B \\
&= -2 \gb_{m-1}^{-1} \gb_k + \gb_{m-1}^{-1} \sum_{j=m}^{k}\gb_{j-1}(\log\ga_j^{-1}-B),
\end{align*}
which finishes the proof of the induction step.

In particular, the formula \eqref{E:arithmetic-height} for $i=1$ becomes 
\begin{align*}
B_{k,0}&=-2\gb_k+\sum_{j=1}^{k}\gb_{j-1} \log\ga_j^{-1}-   B \sum_{j=1}^{k}\gb_{j-1}.
\end{align*}
On the other hand, since each $\ga_i \in (0,1/2)$, we have 
\begin{equation}\label{E:trivial-bound}
\sum_{j=1}^{k}\gb_{j-1} \leq  \sum_{j=1}^{\infty} (1/2)^{j-1}=2.
\end{equation} 
The above two equations prove the first part of the lemma. 

\medskip

{\em Part (b):} Assume on the contrary that $\C{L}(\ga,T,0)$ has finite cardinality for some $T\in \D{R}$. 
Let $k_0$ denote the largest element in $\C{L}(\ga,T,0)$.

Given a positive integer $N_0>k$, define the integers $N_0>N_1>N_2>\cdots >N_r\leq k_0$ according to 
\[B_{N_{l-1},N_l}< T\ga_{N_l}^{-1},\quad \tfor l=1,2,\dots, r.\]
By Equation~\eqref{E:arithmetic-height} for $k=N_{l-1}$ and $i=N_l+1$, the above equation implies that 
\[\sum_{j=N_l+1}^{N_{l-1}}\gb_{j-1}(\log \ga_j^{-1}-B)
=\gb_{N_l} B_{N_{l-1}, N_l} + 2 \gb_{N_{l-1}} 
\leq T\gb_{N_l-1} + 2\gb_{N_{l-1}}.\]
Adding the above sums together for $l=1,2, \dots, r$, we obtain
\[\sum^{N_0}_{j=N_r+1}\gb_{j-1}\log \ga_j^{-1}\leq  2\sum_{l=1}^{r}\gb_{N_{l-1}}+
T\sum_{l=1}^{r}\gb_{N_l-1}+B \sum^{N_0}_{j=N_r+1}  \gb_{j-1}.\]
As $N_0\to \infty$, using the bound in \eqref{E:trivial-bound}, we conclude that 
\[\sum^{\infty}_{j=k_0+1}\gb_{j-1}\log \ga_j^{-1}\leq 2 (2+T+B),\]
contradicting our assumption on the type of $\ga$.
\medskip

{\em Part (c):}
Define $\ell_0=0$, and $\ell_{j+1}=\max \C{L}(\ga,T,\ell_j)$, for $j\geq 0$. 
By definition, $\ell_{k+1}\geq \ell_k+1$, and hence $\ell_k\to +\infty$ as $k\to +\infty$. 

Since $\ga$ is a Brjuno number, by Equation~\eqref{E:arithmetic-height} and the uniform bound in
\eqref{E:trivial-bound}, for every $i\geq 0$, $\lim_{m\to +\infty} B_{m,i}$ exists and is finite.  
We claim that for every $k\geq 0$, 
\[\lim_{m\to +\infty}B_{m,\ell_k} -T\ga_{\ell_k}^{-1}\] 
is uniformly bounded from above independent of $k$. 
This would imply the statement in part (c).

First we see that  
\begin{equation}\label{E:increments}
\forall j\geq 0,   B_{\ell_{j+1}, \ell_j}< T\ga_{\ell_j}^{-1}.
\end{equation}
If the above statement is not correct, for some $j\geq 0$ we have 
\[ B_{\ell_{j+1}, \ell_j} \geq T\ga_{\ell_j}^{-1}\geq -2=B_{\ell_j, \ell_j}.\]
Then, by the recursive relation in \eqref{E:bisequence}, and the definition of $\ell_j$, 
for all $i$ with $\ell_{j-1} < i < \ell_j$ we must have 
\[B_{\ell_{j+1},i} \geq B_{\ell_j, i} \geq T \ga_i^{-1}.\]
However, the above inequality contradicts the choice of $\ell_j= \max \C{L}(\ga, T, \ell_{j-1})$ since 
$\ell_{j+1}> \ell_j$ and satisfies the inequality $B_{\ell_{j+1},i} > T\ga_i^{-1}$ for 
all $\ell_{j-1}<i<\ell_{j+1}$. 

Fix $k\in \D{N}$ and let $n>k$ be an arbitrary integer. 
Recall that $B_{k,k}=-2$, for $k\geq 0$.
By Equation~\eqref{E:increments}, for every $j$ with $k+1\leq j\leq n$ we have 
\[|B_{\ell_{j}, \ell_{j-1}}- B_{\ell_{j-1}, \ell_{j-1}}| < T\ga_{\ell_{j-1}}^{-1}+2.\]
Then, recursively multiplying the above equation by $\ga_{i}$ and then adding and subtracting 
$\log \ga_i^{-1}-B$ within the absolute value, for values of $i=\ell_{j-1}, \dots \ell_k+1$, we 
come up with the inequality
\begin{align*}
|B_{\ell_{j},\ell_{k}}-B_{\ell_{j-1}, \ell_k}|
&< (T\ga_{\ell_{j-1}}^{-1}+2)\ga_{\ell_{j-1}}\ga_{\ell_{j-1}-1}\dots \ga_{\ell_k+1}\\
&= T \gb_{\ell_{j-1}-1} \gb_{\ell_k}^{-1}+ 2 \gb_{\ell_{j-1}} \gb_{\ell_k}^{-1}.
\end{align*}
Then, by the triangle inequality, we have  
\begin{align*}
|B_{\ell_n,\ell_k}-B_{\ell_k, \ell_k}|
&\leq \sum_{j=k+1}^{n} |B_{\ell_j, \ell_k}- B_{\ell_{j-1}, \ell_k}| \\
&\leq \sum_{j=k+1}^{n} (T \gb_{\ell_{j-1}-1} \gb_{\ell_k}^{-1}+ 2 \gb_{\ell_{j-1}} \gb_{\ell_k}^{-1})\\
&\leq T \ga_{\ell_k}^{-1}+ T \sum_{j=k+2}^{+\infty} (\gb_{\ell_{j-1}-1} \gb_{\ell_k}^{-1}) 
+ 2 \sum_{j=0}^{+\infty} \gb_j \\
& \leq  T \ga_{\ell_k}^{-1}+ 2 T + 4. 
\end{align*}
It follows from the above inequalities that for every $n>k$, 
\[B_{\ell_n, \ell_k}- T\ga_{\ell_k}^{-1} \leq 2T+4+2.\]

If an integer $m \nin \{\ell_k,\ell_{k+1}, \ell_{k+2},\dots \}$, choose $n$ with $\ell_{n-1}<m<\ell_n$. 
By the definition of the sequence $\ell_n$ we have 
\[B_{\ell_n,m}\geq T\ga_m^{-1}>B_{m,m}.\]
This implies that 
\[B_{m, \ell_k} -T\ga_{\ell_k}^{-1} \leq B_{\ell_n, \ell_k}-T \ga_{\ell_k}^{-1}\leq 2T+ 6\]

All in all, we have shown that for all $m\geq \ell_k$ we have the uniform bound 
\[B_{m,\ell_k} - T\ga_{\ell_k}^{-1} \leq 2 (T+2)+2.\]
This finishes the proof of part (c).
\end{proof}


\subsection{Estimates on Fatou coordinates}\label{SS:estimates-F-coord}
To control the geometry of the sectors $f_0\co{i}(I^n)$, for $0\leq i\leq q_n-1$, through Proposition~\ref{P:lifts-vs-iterates}, 
we need some estimates on the Fatou coordinates.  
In this section we assemble two crucial estimates, a global one and an infinitesimal one, that we need for the analysis 
in this paper.  

Recall the Fatou coordinate $\gF_h: \p_h \to \D{C}$ of a map $h\in \IS_\ga$, $\ga\in (0, \ga_*]$, from Theorem~\ref{T:Ino-Shi1}.  
We denote the non-zero fixed point of $h$ that lies on the boundary of $\p_h$ by $\gs_h$. 
Consider the covering map  
\begin{equation}\label{E:covering-formula}
T_h(w)= \frac{\gs_h}{1-e^{-2\pi \ga w\B{i}}}: \D{C} \to  \hat{\D{C}} \setminus \{0, \gs_h\}.
\end{equation}
The map $T_h$ commutes with the translation by $1/\ga$. 

The connected components of $T_h^{-1}(\p_h)$ are simply connected sets that are disjoint from $\D{Z}/\ga$. 
Moreover, the projection of each such component onto the imaginary axis covers the whole imaginary axis.  
In particular, we denote the connected component of $T_h^{-1}(\p_h)$ that separates $0$ from $1/\ga$ by $\tilde{\p}_h$.
We may lift the map $\gF_h^{-1}: \C{D}_h \to \D{C}\setminus \{0,\gs_h\}$ under the covering map 
$T_h: \D{C} \to  \hat{\D{C}} \setminus \{ 0, \gs_h\}$ to obtain a map $L_h: \C{D}_h \to \D{C}$ such that 
\begin{equation}\label{E:intermediate-lift-of-F}
T_h  \circ  L_h(\gz)= \gF_h^{-1}(\gz), \forall \gz \in \C{D}_h.
\end{equation}
The above lifts are determined upto a translation by an element of $\D{Z}/\ga$. 
We choose the one that maps $\gF_h(\C{P}_h)$ to $\tilde{\C{P}}_h$.
In other words, $L_h$ is a unique extension of $T_h \circ \gF_h$, mapping $\gF_h(\C{P}_h)$ to $\tilde{\C{P}_h}$.  

Recall the lifts  
\[\gc_{h,0} :\C{D}_h \to \D{C}\] 
defined in Section~\ref{SS:extending-F-coord}.
We need to control the derivative of this map. 

One may lift $h$ via $T_h$ to obtained a holomorphic map 
$F_h$ defined on $L_h(\C{D}_h)$. 
Indeed, the lift is univalent on $\tilde{\C{P}}_h$ and has univalent  extension onto a larger set. 
The map $F_h$ on $\tilde{\p}_h$ is close to the translation by one, with explicit estimates that can be worked out using classical 
distortion estimates on $h$. 
It follows that $L_h^{-1}$ conjugates $F_h$ to the translation by one, and it is a classical non-trivial problem to prove 
estimates on $L_h$ and in particular $|L_h'-1|$, depending on the bounds on $|F_h(w)-(w+1)|$.  
One may refer to \cite{Yoc95,Sh98,Sh00,Ch10-I,Ch10-II}, for further details on this. 
Then, through the decomposition of $\gF_h^{-1}$ to $L_h  \circ T_h$ one obtains relevant estimates on the map $\gF_h$. 
The following estimates are the finest estimates known to us. 
They have been established in \cite{Ch10-I} and \cite{Ch10-II}.

\begin{propo}[\cite{Ch10-I}]\label{P:global-estimate}
For every $D\in \D{R}$ there exists $M>0$ such that for all $h$ in $\IS_\ga \cup \{Q_\ga\}$ with $0 < \ga \leq  \ga^*$ we have 
\[\forall \gz \in \C{D}_h \; \twith \Im \gz \geq D/\ga,  \quad |L_h(\gz) - \gz | \leq M \log (1+ \frac{1}{\ga}).\]
\end{propo}

\begin{propo}[\cite{Ch10-II}]\label{P:fine-estimate}
For all $D>0$ there exists $M>0$ such that for all $h\in \IS_\ga \cup \{Q_\ga\}$ with $0 < \ga \leq \ga_*$, we have 
\[\forall \gz \in \C{D}_h \; \twith \Im \gz \geq D/\ga, \quad |\gc_{h,0}'(\gz)-\ga|\leq M\ga e^{-2\gp \ga \Im \gz}.\]
\end{propo}

\begin{rem}
Proposition~\ref{P:global-estimate} is stated as Proposition~5.15 in \cite{Ch10-I}. 
In the notation of \cite{Ch10-I}, $x_h$ is bigger than or equal to $y_h$ (by definition) and by Proposition 5.14 of that paper 
$y_h\geq \ga^{-1} -\Bk$. 
That is, the inequality in Proposition~\ref{P:global-estimate} is proved for $\gz\in \gF_h(\C{P}_h)$. 
On the other hand, for $\gz\in \C{D}_h \setminus \gF_h(\C{P}_h)$ one uses the equation $F_h\circ  L_h(\gz)= L_h(\gz+1)$ and 
the uniform bounds on $F_h$ to bound $L_h(\gz) - \gz$.  
Since there are uniformly bounded number of iterates of $F_h$ involved ($\B{k}''+\B{k}+\hat{\B{k}}+2$), the estimate also holds on 
$\C{D}_h$.    
Note that the condition $\ga<\ga_2$ in that proposition is already incorporated in Proposition~\ref{P:petal-geometry} of this paper 
under the constant $\ga_*'$. 
Hence, we do not impose any further condition on $\ga$ here. 

Proposition~\ref{P:fine-estimate} stated above is proved in \cite[Proposition~3.3]{Ch10-II}.
Indeed, Proposition~3.3 in \cite{Ch10-II} proves an stronger statement where the dependence of $M$ on $D$ is established and 
the inequality holds on a larger domain. 
The latter part of Proposition~\ref{P:fine-estimate} follows from the proof of Proposition~3.3 in \cite{Ch10-II}.
\end{rem}

\begin{proof}[Proof of Proposition~\ref{P:landing-angle}]
Fix $r\in [0, 1/\ga -\B{k}]$. 
The curve $t\mapsto \gF_h^{-1}(t+ r\B{i})$ lands at $0$ at a well-defined angle as 
$t\to +\infty$ if and only if $\lim_{t\to +\infty} \Re \gc_{h,0}(t+r \B{i})$ exists and is finite. 
We use Proposition~\ref{P:fine-estimate} with $D=1$ to obtain a constant $M_1$ and the estimate on the 
derivative of $\gc_{h,0}$ above the horizontal line $1/\ga$. 
For all $t_1 > t_2 \geq 1/\ga$ we have 
\[
|\Re \gc_{h,0}(t_1+r\B{i}) -\Re \gc_{h,0}(t_2+r \B{i})| 
\leq \int_{t_2}^{t_1} M_1 \ga e^{-2\pi \ga t} dt 
\leq \frac{M_1}{2\pi} e^{-2\pi \ga t_2}. 
\]
Since $e^{-2\pi \ga t_2} \to 0$ as $t_2\to +\infty$, we conclude that  
$\Re \gc_{h,0}(t+ r\B{i})$ tends to a finite limit as $t\to +\infty$. 
\end{proof}

Recall the sets $\C{D}_n$ defined in Section~\ref{SS:change-coordinates}. 
Let $\gr_n(z) |dz|$ denote the Poincar\'e metric on $\C{D}_n$, i.e.\ the hyperbolic metric of constant curvature $-1$. 
The changes of coordinates $\gc_{n,i}: \C{D}_n \to \C{D}_{n-1}$ have the following nice property with respect to these metrics.  

\begin{lem}\label{L:contraction}
There exists a constant $\gr \in (0,1)$ such that for every $n\geq 1$ and all integers $i$ with 
$0 \leq i \leq a_i$, we have $\|\gc_{n,i}'\| \leq \gr$, where the norm is calculated with respect to the hyperbolic metrics on $\C{D}_n$ and $\C{D}_{n-1}$.
\end{lem}

The above lemma appears in \cite{Ch10-I}, also proved here for the readers convenience.

\begin{proof}
Let $\tilde{\rho}_n(z)|d z|$ denote the Poincar\'e metric on the domain $\gc_{n,i}(\C{D}_n)$. 
We may decompose the map $\gc_{n,i}: (\C{D}_n, \gr_n) \to  (\C{D}_{n-1}, \gr_{n-1})$ as follows:
\begin{displaymath}
\xymatrix{
(\C{D}_n, \gr_n) \ar[r]^-{\gc_{n,i}} &\; (\gc_{n,i}(\C{D}_n), \tilde{\rho}_n) \ar@{^{(}->}[r]^-{inc.} &\;  
(\C{D}_{n-1}, \rho_{n-1}).}
\end{displaymath}

By Schwartz-Pick Lemma, the first map in the above chain is non-expanding. 
Hence, it is enough to show that the inclusion map is uniformly contracting in the respective metrics. 
For this, we use Lemma~\ref{L:well-contains}, which provides us with $\gd>0$, independent of $n$ and $i$, such that 
$\gd$ neighborhood of $\gc_{n,i}(\C{D}_n)$ is contained in $\C{D}_{n-1}$.  

To prove the uniform contraction, fix an arbitrary point 
$\xi_0$ in $\gc_{n,i}(\C{D}_n)$, and consider the map
\[H(\xi):=\xi+\frac{\delta (\xi-\xi_0)}{\xi-\xi_0+2\hat{\B{k}}+1},  \gx \in \gc_{n,i}(\C{D}_n).\]
By Proposition~\ref{P:sector-geometry}, for every $\gx \in \gc_{n,i}(\C{D}_n)$ we have $|\Re(\xi-\xi_0)| \leq \hat{\B{k}}$ 
(note that $\C{D}_h \subset \C{D}_h'$). 
This implies that $|\xi-\xi_0|<|\xi-\xi_0+2\hat{\B{k}}+1|$, and hence  $|H(\xi)-\xi| < \delta$.
In particular, $H$ is a holomorphic map from $\gc_{n,i}(\C{D}_n)$ into  $\C{D}_{n-1}$. 
By Schwartz-Pick Lemma, $H$ is non-expanding. 
In particular, at $H(\xi_0)=\xi_0$ we obtain  
\[\rho_{n-1}(\xi_0) |H'(\xi_0)|= \rho_{n-1}(\zeta_0)(1+\frac{\delta}{2\hat{\B{k}}+1}) \leq \hat{\rho}_n(\xi_0).\]
That is, 
\[\rho_{n-1}(\xi_0)\leq \big( \frac{2\hat{\B{k}}+1}{2\hat{\B{k}}+1+\delta} \big) \hat{\rho}_n(\xi_0).\]
As $\gx_0$ was arbitrary, this finishes the proof of the lemma.
\end{proof}


\subsection{Geometry of sectors on level $n$}\label{S:SectorsAtDeepLevels}
First we need the following basic property of the maps in $\IS_\ga \cup \{Q_\ga\}$. 

\begin{lem}\label{L:size-of-sigma}
There exists a constant $C_1>0$ such that for all $\ga\in (0, \ga^*]$ and $h\in \IS_\ga \cup \{Q_\ga\}$ we have 
\[ C_1^{-1} \ga \leq   |\gs_h | \leq C_1 \ga.\] 
\end{lem} 

\begin{proof}
The non-zero, or parabolic, fixed point of $Q_\ga$ is given by the formula 
$(1-e^{2\pi \ga \B{i}})\frac{16}{27} e^{-4\pi \ga \B{i}}$.
Thus, there is an explicit constant $C_1$ satisfying the inequalities for the map $Q_\ga$. 
Below, we assume that $h \in \IS_\ga$. 

For every $h\in \IS_\ga$, $0 \leq \ga\leq \ga^*$, there is a holomorphic map $u_h$, 
defined on the domain of $h$, such that $h(z)=z+ z(z-\gs_h)u_h(z)$. 
The map $u_h$ depends continuously on the map $h$. 
Consider the set 
\[A=\{u_h(0) \mid h\in \IS_\ga, 0\leq \ga \leq \ga^*\}.\] 
For $h\in \IS_\ga$ with $\ga\neq 0$, $0$ is a simple fixed point of $h$ and hence $u_h(0) \neq 0$. 
For $h\in \IS_0$, we have $u_h(0)= h''(0)/2$ where $|h''(0)|$ is uniformly bounded from above and away 
from $0$, by Lemma~\ref{L:bounds-on-second-derivative}. 
By the pre-compactness of the class $\cup_{\ga \in [0, \ga^*]}\IS_\ga$, the set $A$ is compactly 
contained in $\D{C}\setminus \{0\}$.

Differentiating equation $h(z)=z+ z(z-\gs_h)u_h(z)$ at $0$ we obtain 
$\gs_h=(1-e^{2\pi \ga \B{i}})/u_h(0)$. 
Combining with the uniform bounds from the previous paragraph, this finishes the proof of the lemma. \end{proof}

For $\ga\in (0,1/2)$ define the set 
\begin{equation}\label{E:X}
X_\ga = \left\{ \gz\in\D{C} \mid \Im \gz \geq -2, 
(\frac{1}{\ga})^{\frac{2}{3}} \leq \Re \gz \leq (\frac{1}{\ga})- (\frac{1}{\ga})^{\frac{2}{3}}  \right\}.
\end{equation}

Let $M$ be the constant obtained from Proposition~\ref{P:global-estimate} for $D=-1$, and then define the set 
 \[\hat{X}_\ga = \left\{ w \in\D{C} \mid \Im (w) \geq 2M \log \ga,  
\frac{1}{2} (\frac{1}{\ga})^{\frac{2}{3}} \leq \Re (w) \leq (\frac{1}{\ga})-\frac{1}{2}(\frac{1}{\ga})^{\frac{2}{3}} \right\}.\]

Choose $\gd_1>0$ such that  
\[\gd_1\leq 1/8, \gd_1 \leq (1/\B{k})^{3/2},\]
and for all $\ga \in (0, \gd_1]$ we have 
\begin{equation}\label{E:delta_1-introduced}
\begin{gathered}
-2-M \log (1+ 1/\ga) \geq -2 M \log (1/\ga), \\
\frac{1}{2} (1/\ga)^{\frac{2}{3}} \leq  (1/\ga)^{2/3}- M \log (1+1/\ga),\\
(1/\ga)-(1/\ga)^{2/3}+ M \log (1+1/\ga)  \leq (1/\ga)-\frac{1}{2}(1/\ga)^{2/3}
\end{gathered}
\end{equation}

\begin{lem}\label{L:step1}
For all $h\in \IS_\ga \cup \{Q_\ga\}$ with $\ga\in (0, \gd_1]$, we have $L_h(X_\ga) \subset \hat{X}_\ga$. 
\end{lem}

\begin{proof}
First note that since $\ga < 1/8$, both sets $X_\ga$ and $\hat{X}_\ga$ are non-empty. 
Also, since $\ga \leq (1/\B{k})^{3/2}$ the set $X_\ga$ is contained in $\gF_h(\C{P}_h)$, and hence 
$L_h^{-1}$ is defined on $X_h$. 

Now we use Propositio~\ref{P:global-estimate} with $D=-1$. 
Note that for every $\gz\in X_\ga$, $\Re \gz \in (0, 1/\ga -\B{k})$ and $\Im \gz \geq -2 \geq -1/\ga$. 
Hence the inequality Proposition~\ref{P:global-estimate} holds at all points in $X_\ga$, with the constant $M$.
In particular, the inequality implies that for all $\gz\in X_\ga$ we have 
\[ (1/\ga)^{2/3}- M \log (1+1/\ga) \leq  \Re L_h(\gz) \leq (1/\ga)-(1/\ga)^{2/3}+ M \log (1+1/\ga) \]
and 
\[ \Im L_h(\gz)\geq -2 - M \log (1+1/\ga)\] 
Now the lemma follows from the conditions imposed on $\gd_1$ in Equation~\eqref{E:delta_1-introduced}. 
\end{proof}

\begin{lem}\label{L:step2}
There exists $C>0$ such that for all $h\in \IS_\ga \cup \{Q_\ga \}$ with $\ga\in (0, \gd_1]$ the following holds. 
For every $z\in \D{C}$ with $\ex(z)\in T_h(\hat{X}_\ga)$, we have 
\[\Im z \geq  \frac{1}{3\pi} \log \frac{1}{\ga}-C.\] 
\end{lem}

\begin{proof}
First note that for all $w\in \hat{X}_\ga$, 
\[ -2\pi + \pi  \ga ^{1/3}  \leq  \Im (-2\pi \ga w\B{i}) \leq  - \pi \ga ^{1/3}.\]
This implies that 
\[ |\arg (e^{-2\pi \ga w \B{i}})| \leq \pi \ga ^{1/3}.\]
Using the inequality $\sin x \geq 2 x/\pi$ on the interval  $(0, \pi/2)$, as $\ga \leq 1/8$, we conclude that 
\[|1- e^{-2\pi \ga w \B{i}} | \geq \sin(\pi \ga^{1/3}) \geq  2\ga ^{1/3}.\] 
On the other hand, by Lemma~\ref{L:size-of-sigma},  $|\gs_h|\leq C_1 \ga$. 
Therefore, for all $w\in \hat{X}_\ga$, the above inequality implies that 
\[|T_h(w)|\leq \frac{C_1}{2} \ga^{2/3}.\] 
Assume for some $z\in \D{C}$, $\ex(z)\in T_h(\hat{X}_\ga)$.
By the definition of $\ex$, the above inequality implies that 
\[\Im z\geq \frac{1}{2\pi} \log \frac{8}{27 C_1} + \frac{1}{2\pi} \frac{2}{3} \log \frac{1}{\ga}.\]
\end{proof}

Note that the statement of the above lemma holds similarly if for $z\in \D{C}$ we have $\ex(z)\in s\circ T_h(\hat{X_\ga})$. 

\begin{lem}\label{L:step3}
There exists $C'>0$ such that for all $h\in \IS_\ga \cup \{Q_\ga \}$ with $\ga\in (0, \gd_1]$, we have  
\[ \frac{1}{2\pi}  \log \frac{1}{\ga} - C'  \leq   \Im \gc_{h,0} (\frac{1}{2\ga}-2 \textnormal{\B{i}}) \leq \frac{1}{2\pi} \log \frac{1}{\ga} +C'   \]
\end{lem}

\begin{proof}
Let $\gz_0=1/(2\ga)-2\B{i}$ and note that for $\ga\leq 1/(2\B{k})$, $\gz_0$ is contained in $\gF_h(\C{P}_h) \subset \Dom L_h$. 
Using Proposition~\ref{P:global-estimate} with $D=-1$, we know that 
\[|L_h(\gz_0)-\gz_0| \leq M \log (1+ 1/\ga).\] 
On the other hand, by the choice of $\gd_1$, $M\log (1+1/\ga)\leq 1/(3\ga)$. 
It follows that 
\[ -2-\frac{1}{3\ga} \leq  \Im L_h(\gz_0) \leq \frac{1}{3\ga}, \quad \frac{1}{6\ga}  \leq  \Re L_h(\gz_0)\leq  \frac{5}{6\ga}.\]
Then, 
\[ \frac{-5\pi}{3} \leq  \Im (-2\pi \ga L_h(\gz_0) \B{i}) \leq \frac{-\pi}{3}, 
\quad  \frac{-2\pi}{3} \leq \Re (-2\pi \ga L_h(\gz_0) \B{i}) \leq 4\pi \ga+\frac{2\pi}{3}.\]
Therefore, 
\[  1\leq |1- e^{-2\pi \ga L_h(\gz_0) \B{i}})|\leq 1+ e^{7\pi/6}.\] 
Combining the above bounds with the bounds on the size of $\gs_h$ in Lemma~\ref{L:size-of-sigma}, we obtain 
\[\frac{1}{C_1} \frac{1}{1+ e^{7\pi/6}} \ga \leq   |T_h(L_h(\gz_0))| \leq C_1 \ga.\]
Finally, as $\ex (\gc_{h,0}(\gz_0))= T_h(L_h(\gz_0))$, we conclude that  
\[\frac{1}{2\pi} \log \frac{1}{\ga} -C' \leq \Im \gc_{h,0}(\gz_0) \geq \frac{1}{2\pi} \log \frac{1}{\ga}+C',\]
for soma constant $C'$ depending only on $C_1$.
\end{proof}

In the following proposition, $C'$ is the constant obtained in the above Lemma.
\begin{propo}\label{P:density-high-levels-raw}
There is $\gd_0>0 $ such that for all $h\in \IS_\ga \cup \{Q_\ga \}$ with $\ga \in (0, \gd_0]$ we have 
\begin{equation}
\forall \gz\in X_\ga, \quad \Im \gc_{h,0} (\gz) \geq \frac{1}{2}\Im \gc_h (\frac{1}{2\ga}-2 \textnormal{\B{i}}) + C'. 
\footnote{Here, $C'$ may be replaced by any constant; with $\gd_0$ depending on that constant.}
\end{equation}
\end{propo}
\begin{proof}
There is $\gd_0<\gd_1$ such that for all $\ga\in (0, \gd_0]$ we have 
\[\frac{1}{12\pi} \log \frac{1}{\ga} \geq C+ 3C'/2.\] 
For all $\gz$ in $X_\ga$, by Lemma~\ref{L:step1}, $L_h(\gz)$ belongs to $\hat{X}_\ga$. 
Then, by Lemma~\ref{L:step2}  
\[\Im \gc_{h,0} (\gz) \geq \frac{1}{2\pi} \frac{2}{3} \log \frac{1}{\ga} -C.\]
On the other hand, by Lemma~\ref{L:step3}, 
\[\Im \gc_{h,0} (\frac{1}{2\ga}-2 \textnormal{\B{i}}) \leq \frac{1}{2\pi} \log \frac{1}{\ga} +C' .\]
Thus, we need 
\[\frac{1}{2\pi} \frac{2}{3} \log \frac{1}{\ga} -C \geq \frac{1}{2} \big ( \frac{1}{2\pi} \log \frac{1}{\ga} +C'\big)+ C',\]
which is guaranteed by the condition $\ga\in (0, \gd_0]$. 
This finishes the proof of the Proposition.
\end{proof}

\begin{figure}\label{F:log-of-sectors}
\begin{center}
\begin{pspicture}(12,10)  
\epsfxsize=6cm 
\rput(3,7){\epsfbox{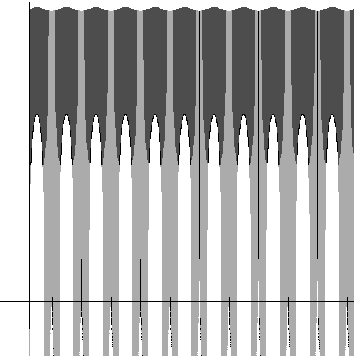}}

\pspolygon[fillstyle=solid,linecolor=grey67,fillcolor=grey67](6.4,1)(6.4,3.6)(7,3.6)(7,1)
\pspolygon[fillstyle=solid,linecolor=grey67,fillcolor=grey67](11,1)(11,3.6)(11.6,3.6)(11.6,1)
\pspolygon[fillstyle=solid,linecolor=grey30,fillcolor=grey30](7,1)(7,3.6)(11,3.6)(11,1)

\psline[linewidth=.5pt](6.4,1)(6.4,3.6)

\psline[linecolor=gray,linewidth=.3pt]{->}(5,1)(12,1) 
\psline[linecolor=gray,linewidth=.3pt]{->}(6.1,.5)(6.1,5) 

\rput(6,.8){\tiny $0$}
\rput(6.4,.8){\tiny $.5$}
\rput(7.1,.8){\tiny $(\frac{1}{\ga})^{\frac{2}{3}}$}
\rput(11.6,.8){\tiny $\frac{1}{\ga}-.5$}

\pscurve[linewidth=.5pt]{<-}(5,2.6)(4,2.9)(3.4,3.7)
\rput(3,3){\tiny $T_{Q_\ga}^{-1} \circ \ex$}

\rput(6.9,7.2){\tiny $\sim \frac{2}{3}\frac{1}{2\pi} \log \frac{1}{\ga}$}
\rput(6.9,8){\tiny $\sim \frac{1}{2\pi} \log \frac{1}{\ga}$}

\end{pspicture}
\end{center}
\caption{The lower part shows a dark grey rectangle 
$[(\frac{1}{\ga})^{\frac{2}{3}}, \frac{1}{\ga} - (\frac{1}{\ga})^{\frac{2}{3}}]+\B{i} [0, .5/\ga]$, and 
two light grey rectangles 
$[.5, \frac{1}{\ga})^{\frac{2}{3}}]+\B{i} [0, .5/\ga]$ 
and 
$[\frac{1}{\ga} - (\frac{1}{\ga})^{\frac{2}{3}}, \frac{1}{\ga} -.5]+\B{i} [0, .5/\ga]$.
The top figure shows some of the lifts of these rectangles under $ T_{Q_\ga}^{-1} \circ \ex$.   
The dark grey rectangle lifts to periodic regions above the approximate height 
$\frac{1}{2\pi}\frac{2}{3} \log \frac{1}{\ga}$, upto an additive constant.
The lift of the line segment between $(.5,0)$ and $(.5,.5/\ga)$ lifts to the black curves separating 
different lifts.  
Here, $\ga=1/1000$ and the choice of $.5$ is only to make the pictures clearer.}
\end{figure}

\begin{lem}
We have 
\[\lim _{\ga\to 0^+} \frac{\sup_{w,w'\in X_\ga} |\Re w - \Re w'|}{1/\ga } =1.\]
\end{lem}
\begin{proof}
This is straightforward, since
\[\sup_{w,w'\in X_\ga} |\Re w - \Re w'| \geq  ( \frac{1}{\ga}) - ( \frac{1}{\ga})^{\frac{2}{3}} -  ( \frac{1}{\ga})^{\frac{2}{3}} \]
and 
\[\lim_{\ga\to 0^+} \frac{(1/\ga) - 2(1/\ga)^{2/3}}{1/\ga}=1.\]
\end{proof}


\subsection{The heights of chains in the tower}
Recall the maps $\gc_{n,0} = \gc_{f_n}: \C{D}_n \to \C{D}_{n-1}$, for $n\geq 1$.
Given $n\geq 1$ and a point $\gz_n \in \C{D}_n$, define 
\[\C{C}_n(\gz_n)= \left\{ \langle \gz_j \rangle_{j=n}^0 \; \mid 
\forall j \text{ with } 0\leq j \leq n-1, \gz_j \in \C{D}_{j} \tand \gz_{j} \in \gc_{j+1,0}(\gz_{j+1})+\D{Z}.
\right\}.\]
That is, members of $\C{C}_n(\gz_n)$ are (ordered) sequences of length $n+1$ starting with $\gz_n$ and mapping an element to 
the next element in the sequence by a translation of some $\gc_{j,0}$. 
Given $\gt= \langle \gz_n, \gz_{n-1}, \dots, \gz_0 \rangle$ in $\C{C}_n(\gz_n)$, for every $k\in \{n, n-1,\dots, 0\}$ we define 
$\gt_k= \gz_k$. 
We aim to study the behavior of the sequences $\langle \Im \gt_j\rangle_{j=n}^0$, for $\gt$ in $\C{C}_{n}(t)$, as $t$ varies in 
$\C{D}_n$. 
This heavily  depends on the arithmetic of $\ga$. 

For Brjuno values of $\ga$, let $\gD(f_j)$, for $j\geq 0$, denote the Siegel disk of $f_j$. 
Then, define 
\[\tilde{\gD}(f_j)=\gF_j(\gD(f_j)\cap \p_j), \tfor j \geq 0.\]  
Also, recall the bi-sequence $B_{n,j}$ defined in Section~\ref{SS:arithmetic}.  

\begin{propo}\label{P:nearly-right-height}
There are constants $B$ and $M$ in $\D{R}$ satisfying the following. 
Let $B_{k, i}$ be the bi-sequence defined in \eqref{E:bisequence} with the constant $B$, and let $\gz_n= \frac{1}{2\ga_n} - 2 \textnormal{\B{i}}$.
Then, 
\begin{itemize}
\item[a)] for all $\ga\in \irr$ and every integer $n\geq 1$ there exists $\gt\in \C{C}_{n}(\gz_n)$  such that for all $j$ with 
$0\leq j \leq n-1$, 
\begin{equation*}
B_{n,j} \leq \Im \gt_j\leq B_{n,j}+M,\;
\end{equation*}
\item[b)] if $\ga$ is a Brjuno number, for every $j\geq 0$ we have 
\begin{equation*}
\big| \lim _{n\to\infty}B_{n,j}-\max_{w\in \partial \tilde{\gD}(f_j)} \Im w\big |<M.
\end{equation*}
\end{itemize}
\end{propo}

\begin{proof}
Part a) By the condition \eqref{E:high-type-restriction} on the rotation numbers $\ga_n$ we have 
\[\frac{1}{\ga_n} -\B{k} - \frac{1}{2 \ga_n} \geq \frac{\hat{\B{k}}}{2}+ \frac{5}{2} \geq 3. \]
Hence,   
\begin{equation}\label{E:center-points}
\gz_n= \frac{1}{2\ga_n} - 2 \B{i} \in \gF_n(\C{P}_n) \subset \C{D}_n, \forall n\geq 1.
\end{equation}
Inductively define the sequence $\gz_{j}$, for $n-1\leq j\leq 0$, 
such that 
\[\gz_{j} \in \gc_{j+1, 0}(\gz_{j+1}) +\D{Z}, \quad \Re \gz_j\in [\frac{1}{2\ga_j}-\frac{1}{2}, \frac{1}{2\ga_j}+\frac{1}{2}].\]
The above condition on the rotation implies that $\gz_j \in \C{D}_j$ for all $j$ with $0 \leq j \leq n-1$.
In particular, $\gt=\langle \gz_j \rangle_{j=n}^0$ belongs to $\C{C}_{n}(\gz_n)$. 

Note that for every $j\geq 0$, the set $\gF_j^{-1}(\C{D}_j)$ is contained in the image of $f_j$. 
That is because, $f_j(\cup_{i=0}^{k_j} f_j\co{i}(S_j^0) \cup \p_j)$ covers this set. 
By the definition of renormalization, the image of each $f_j$, $j\geq 1$, is equal to the disk of radius 
$e^{4\pi }4/27$ about zero.
In particular, $\gF_n^{-1}(\gz_n)$ is contained in the image of $f_n$. 
This implies that $\Im \gz_{n-1}\geq -2$. 
Continuing this process inductively, we conclude that for all $j$ with $0\leq j \leq n$, 
$\Im \gz_j \geq -2$.  
In particular, for all $j$ with $0\leq j \leq n$
\[\Im \gz_j\geq -1/\ga.\]

Let $L_j$ denote the map $L_{f_j}$ introduced in Equation~\eqref{E:intermediate-lift-of-F}. 
By the above inequality, we may applying Proposition~\ref{P:global-estimate} with $D=-1$ to obtain 
$M_1\in \D{R}$ such that for all $j$ with $0\leq j \leq n$, 
\[L_j(\gz_j) \in B(\gz_j, M_1 \log (1+ 1/\ga_j)).\]

Then, by a basic estimate there is a constant $M_1$ depending only on $M_1$ such that for all $j$ 
with $0\leq j \leq n$, we have 
\[M_2^{-1} e^{-2\pi \ga_j \Im \gz_j} 
\leq |1-e^{-2\pi \ga_j \B{i} L_j(\gz_j)}| 
\leq M_2 e^{-2\pi \ga_j \Im \gz_j}.\]  
Now, the uniform bounds in Lemma~\ref{L:size-of-sigma} imply that there exists a constant $M_3$ 
depending only on $M_2$ and $C_1$ such that for all $j$ with $0\leq j \leq n$, we have 
\[M_3 \ga_j  e^{-2\pi \ga_j \Im \gz_j}   \leq |\gF_j^{-1}(\gz_j)| \leq M_3 \ga_j  e^{-2\pi \ga_j \Im\gz_j}.\]
As $\ex {\gz_{j-1}}= \gF_j^{-1}(\gz_j)$ or  $\ex {\gz_{j-1}}=s\circ \gF_j^{-1}(\gz_j)$, an explicit estimate on $\ex$ implies that there is 
$M_4 \geq 0$ such that for every $j$ with $0\leq j \leq n$ we have 
\begin{equation}\label{E:actual-recursive}
\ga_j \Im \gz_j + \log \frac{1}{\ga_{j-1}} -M_4  
\leq  \Im (\gz_{j-1}) 
\leq \ga_j \Im \gz_j + \log \frac{1}{\ga_{j-1}} +M_4.
\end{equation}
Note that the constant $M_4$ depends only on the class of maps $\IS$.   

Let $B=M_4$, that is, the constant used in the definition of the bi-sequence $B_{k,j}$. 
By an inverse inductive process (from $j=n$ to $j=0$) we show that 
\begin{equation}\label{E:local-desire}
B_{n,j} \leq \Im \gz_j  \leq B_{n,j}+4 M_4 \Big( 1+ \sum_{i=j}^{n} (\prod_{l=i+1}^{n} \ga_{l-1})\Big )
\end{equation}

For $j=n$, Inequality~\eqref{E:local-desire} reduces to $B_{n,n}=-2\leq \Im \gz_n \leq -2$ 
(there is no sum), which holds by the definition of $\gz_n$. 
Assume that Inequality~\eqref{E:local-desire} holds for $j$. 
We want to prove it for $j-1$. 
Multiplying both sides of Equation~\eqref{E:local-desire}  by $\ga_j$, then 
adding $\log 1/\ga_j$, and then subtracting $B$ we come up with 
\begin{align*}
\ga_j B_{n,j}+\log \frac{1}{\ga_{j-1}} -B \leq 
\ga_j \Im \gz_j + \log \frac{1}{\ga_{j-1}} -B
\end{align*}
and 
\begin{align*} 
\ga_j \Im \gz_j + \log \frac{1}{\ga_{j-1}} -B
\leq \ga_j B_{n,j}+ \log \frac{1}{\ga_{j-1}} -B + 
4 \ga_j M_4 \big(1+ \sum_{i=j}^{n} \gb_i^{-1} \gb_n \big ),
\end{align*}
By the definition of the bi-sequence and Equation~\ref{E:actual-recursive}, the first 
inequality above provides us
\begin{align*}
B_{n,j-1}\leq \ga_j \Im \gz_j + \log \frac{1}{\ga_{j-1}} -B
\leq \Im \gz_{j-1}
\end{align*}
and, similarly, the second inequality gives us  
\begin{align*}
\Im (\gz_{j-1})  
& \leq \ga_j \Im \gz_j + \log \frac{1}{\ga_{j-1}} +M_4 \\
& \leq \ga_j \Im \gz_j + \log \frac{1}{\ga_{j-1}} - B +2M_4 \\
& \leq B_{n, j-1} + 4 \ga_j  M_4 \big (1+ \sum_{i=j}^{n} \gb_i^{-1} \gb_n \big )+ 2 M_4  \\
&= B_{n, j-1} + 4 M_4\big (\ga_j+ \sum_{i=j-1}^{n} \gb_i^{-1} \gb_n \big ) + 2M_4 \\
& \leq B_{n, j-1} + 4 M_4\big (1+ \sum_{i=j-1}^{n} \gb_i^{-1} \gb_n \big ). 
\end{align*}
In the last inequality we have used that $\ga_j+1/2\leq 1$.
This finishes the proof of the induction step.

Finally, since $\ga_j\in (0, 1/2)$, for all $j$, we have  
\[ 4 M_4 \big( 1+ \sum_{i=j}^{n} \gb_{i-1}^{-1} \gb_{n-1} \big ) 
\leq 4 M_4 \big(1+ \sum_{i=j}^{n}  (\frac{1}{2})^{n-i} \big)
\leq 4 M_4 \big(1+ \sum_{i=0}^{+\infty} (\frac{1}{2})^{i})\big)
\leq 12 M_4.\]
The above bound combined with Inequality~\eqref{E:local-desire} implies a). 
That is, we define $M=12 M_4$.

\medskip

Part b) 
By Equation~\eqref{E:bisequence}, for every $j\geq 0$, $\lim_{n\to \infty} B_{n,j}$ exists and 
\[\Big|\lim_{n\to \infty} B_{n,j} - \gb_{j}^{-1} \sum_{i=j}^{+\infty} \gb_{i}\log \ga_{i+1}^{-1} \Big| 
\leq B \gb_j^{-1} \sum_{i=j}^{+\infty} \gb_{i} \leq 2 B.\]

Each map $f_j$, $j\geq 0$, belongs to the class $\IS_{\ga_j}$ and must be univalent on the disk 
$B(0, 1/12)$. 
That is because, the polynomial $P$ is univalent on the disk $B(0, 1/3)$ and by the $1/4$-theorem 
$\gf(B(0, 1/3))$ contains $B(0,1/12)$.
By the classical result of Yoccoz~\cite{Yoc95}, there is a uniform constant $C>0$ (independent of $j$) 
such that the ball of radius $C \exp (-\gb_{j}^{-1} \sum_{i=j}^\infty \gb_{i} \log \ga_{i+1}^{-1})$ 
centered at $0$ is contained in $\gD(f_{j+1})$.
Since $\gD(f_{j+1})$ lifts under $\ex$ to the set $\tilde{\gD}(f_j)+\D{Z}$,  
\[\max_{w\in \partial \tilde{\gD}(f_j)} \Im w 
\leq \frac{1}{2\pi} \log \frac{4}{27C} +\frac{1}{2\pi} 
\gb_j^{-1} \sum_{i=j}^{+\infty} \gb_{i}\log\ga_{i+1}^{-1}\]
Indeed, one may recover the lower bound of Yoccoz from the estimates in part a), 
which we briefly sketch it here. 
In part a) we have chosen the branches of the lifts $\gc_{j,0}+\D{Z}$ (the ones at the center) 
to obtain the highest possible imaginary parts. 
That is, all other choices lead to smaller imaginary parts. 
This implies that all the sets $\gU_n$, similarly defined for $f_j$, contain the ball of radius 
$C'' \exp (-\gb_{j}^{-1} \sum_{i=j}^\infty \gb_{i} \log \ga_{i+1}^{-1})$ centered at $0$, 
for some uniform constant $C''$. 
Since Proposition~\ref{P:Siegel-disk-enclosed} holds for all maps $f_{j}$, with 
appropriately defined $\gU^n$, one concludes that $\gD(f_j)$ must contain the disk of that radius.  

On the other hand, we need to prove an upper bound on the size of the biggest ball that fits into 
the Siegel disk of $f_j$ in terms of the above infinite series.
This upper bound is proved in \cite{Ch10-I}, which we briefly present the argument 
below for the readers convenience.
There is a proof of the upper bound, which only works when $f_j$ is a quadratic polynomial, 
given in \cite{BC04}.

Fix $j \geq 0$ and let $n > j$ be an arbitrary integer. 
Let $\gt \in \C{C}_n(\gz_n)$ denote the sequence obtained in part a). 
Consider the sequence $\gt'\in \C{C}_n(\langle \frac{1}{2\ga_n} \rangle)$ defined along the same 
branches determining $\gt$. 
That is, for each $i$ with $1 \leq i\leq n$, if $\gz_{i-1}= \gc_{i,l}(\gz_i)$ for some integer $l$ 
then $\gt'_{i-1}= \gc_{i,l}(\gt'_i)$. 
The point $\gF_n^{-1}(\gt'_n)$ belongs to the forward orbit of the critical point of $f_n$. 
It follows from the prof of Proposition~\ref{P:lifts-vs-iterates} that $\gF_j^{-1}(\gt_j')$ belongs 
to the forward orbit of the critical point of $f_j$. 
In particular, $\gF_j^{-1}(\gt_j')$ does not belong to the Siegel disk of $f_j$.
By definition, $\gt_j' \notin \tilde{\gD}(f_j)$. 

By the uniform contraction in Lemma~\ref{L:contraction} and the uniform inclusion in 
Lemma~\ref{L:well-contains}, we conclude that provided  $n$ is sufficiently large, 
$|\Im \gz_j- \Im \gt'_j|\leq 1$. 
Thus, by the inequalities in part a), $\Im \gt'_j \geq B_{n,j}-1$. 
(Alternatively, one may repeated the estimates in the proof of part a) for the sequence $\gt'$ to 
obtain estimates as in Equation~\eqref{E:local-desire} for $\gt'$.)
Combining with the previous paragraph, we have 
\[\max_{w\in \partial \tilde{\gD}(f_j)} \Im w 
\geq \Im \gt'_j \geq B_{n,j}-1.\]
In particular, $\max_{w\in \partial \tilde{\gD}(f_j)} \Im w  - \lim_{n\to +\infty} B_{n,j} \geq -1$. 
This finishes the proof of part b.
\end{proof}

By Proposition~\ref{P:fine-estimate} with $D=1$ there is a constant $M_1$ such that the inequality in that proposition holds. 
In particular, for every $D \geq 1$, the inequalities hold with the constant $M_1$. 
This implies that we may choose $D\geq 1$ such that with the corresponding constant $M$ from the proposition,
for every $\ga\in (0,1/2)$ we have 
\begin{equation}\label{E:small-integral-1}
\int_{D/\ga}^\infty  M\ga e^{-2\gp \ga t}\,dt = \frac{M}{2\pi} e^{-2\pi D}  \leq  1/2, 
\end{equation}
and 
\begin{equation}\label{E:small-integral-2}
\int_{0}^{1/\ga+ \B{k}''+ \B{k}+\hat{\B{k}}+2}M\ga e^{-2\gp D}\,dt = M \ga e^{-2 \gp D} (1/\ga+ \B{k}''+ \B{k}+ \hat{\B{k}}+2)  \leq 1/2.
\end{equation}

\begin{lem}\label{L:heights-of-lifts}
With the constant $D$ obtained above we have the following. 
If for some positive integers $m<n$ there exists $\gt\in \C{C}_{n}(\frac{1}{2\ga_n}-2\mathbf{i})$       
satisfying 
\[\Im \gt_j \geq \frac{D+2}{\ga_j}, \; \forall j \; \twith  m \leq j \leq n-1,\] 
then for every $\gt'\in  \C{C}_{n}(\frac{1}{2\ga_n}-2\mathbf{i})$ we have 
\[|\Im \gt'_j - \Im\gt_j| \leq 2,\; \forall j \; \twith  m-1 \leq j \leq n-1.\] 
\end{lem}
\begin{proof}
Assume $\gt\in \C{C}_{n}(\frac{1}{2\ga_n}-2\mathbf{i})$ be a chain satisfying the hypotheses of the lemma, and
$\gt'\in  \C{C}_{n}(\frac{1}{2\ga_n}-2\mathbf{i})$ be an arbitrary chain. 
By an inductive argument we show that for all $j$ with $0 \leq j \leq n-1$ we have 
\begin{equation}\label{E:little-changes-in-height}
|\Im \gt_{j}' - \Im \gt_j | \leq 1+ \sum_{k=j+1}^{n-1} \gb_j^{-1} \gb_k  .
\end{equation}

First note that  since for all $i\geq 0$, $\ga_i\in (0, 1/2)$, we have 
\[1+ \sum_{k=j+1}^{n-1} \gb_j^{-1} \gb_k \leq 1+ \sum_{k=1}^{n-1} \gb_k  
\leq 1+ \sum_{k=1}^{\infty} \frac{1}{2^{k}} \leq 2.\]

For $j=n-1$, we have $\gt_{n-1}' \in \gc_{n,0}(\gz_n)+\D{Z}$ and hence, $\Im \gt'_{n-1}= \Im \gt_{n-1}$.  
Therefore, the inequality \eqref{E:little-changes-in-height} holds for $j=n-1$. 

Assume the inequality \eqref{E:little-changes-in-height} holds for some $j$. 
By the assumption in the lemma we have $\Im \gt_j \geq (D+2)/\ga_j$, and then by the induction assumption 
we obtain $\Im \gt_j' \geq (D+2)/\ga_j - 2 \geq D/\ga_j$. 
Let $\gga_j$ be a piecewise smooth curve in $\C{D}_j$ that lies above the horizontal line passing through 
$D\B{i}/\ga_j$ and connects $\gt_j$ to $\gt'_j$. 
We may choose $\gga_j$ such that it consists of a horizontal line segment of length at most $1/\ga_j +\B{k}''+ \B{k} + \hat{\B{k}}+2$ 
and vertical line segment of length at most $2$. 
Then, 
\begin{align*}
|\Im (\gt_{j-1}-\gt'_{j-1})|&=|\Im\int_{\gga_{j}} \gc'_{j,0}\, dz|   \\
& \leq \Im\int_{\gga_{j}} |\gc'_{j,0}-\ga_{j}|\, dz + |\Im\int_{\gga_{j}} \ga_{j}\, dz| \\
&\leq 1/2 + 1/2 + \ga_{j}|\Im (\gt_{j}-\gt'_{j})|, 
\end{align*}
where in the last inequality we have used the inequalities in Equations~\eqref{E:small-integral-1} and \eqref{E:small-integral-2}.
Combining with the induction hypothesis we obtain the inequality \eqref{E:little-changes-in-height} for $j-1$.
\end{proof} 


\subsection{Proof of Proposition~\ref{P:cremer-sectors}}
Recall the set $J_{n}$ defined in section~\ref{SS:lifts-vs-iterates} and the set $X_{\ga_n}$ defined in 
Section~\ref{S:SectorsAtDeepLevels}.
Consider the set 
\[G_{n} = \left\{ j\in \D{Z} \; \mid \; f_{n} \co{j}(J_{n}) \subset \C{P}_n, \gF_{n}(f_{n}\co{j}(J_{n})) \subset X_{\ga_{n}} \right\}.\]

\begin{lem}\label{L:density-on-hight-level}
$\forall \gep>0,\, \exists \gd>0 $ such that if $\ga_{n} \leq \gd$ for some $n\geq 1$, then  
\begin{equation}
\frac{|G_n|}{a_n +1}> 1-\gep. 
\end{equation}
\end{lem}

\begin{proof}
Recall that the set $\gF_n(J_n)$ satisfies Equation \eqref{E:width-of-J}. 
Hence, by Proposition~\ref{P:petal-geometry}, for every positive integer 
$j \leq a_n -\B{k}-\hat{\B{k}}-3$,  
$\gF_n(J_n)+j$ is contained in $\gF_n(\C{P}_n)$. 
In particular, this implies that for all such $j$, $f_n\co{j}(J_n)$ is defined and is contained in $\C{P}_n$. 

On the other hand, for all integers $j$ such that 
\[(\frac{1}{\ga_n})^{2/3} \leq j+1, \quad \tand \quad  \hat{\B{k}}+2+j \leq (\frac{1}{\ga_n}) -(\frac{1}{\ga_n})^{2/3},\]
we have $\gF_n(J_n)+j$ is contained in $X_{\ga_n}$.
Therefore, all such integers $j$ are contained in $G_n$. 
Since the number of integers in a closed interval of length $l$ is at least $l-1$, we conclude that 
\begin{align*}
|G_n| &\geq 
(\frac{1}{\ga_n}) - (\frac{1}{\ga_n})^{2/3} -\hat{\B{k}} - 2  - (\frac{1}{\ga_n})^{2/3} - 1 -1\\
& \geq (\frac{1}{\ga_n}) - 2(\frac{1}{\ga_n})^{2/3} - \hat{\B{k}} -4.
\end{align*}
Since the ratio 
\[\frac{(\frac{1}{\ga_n}) - 2(\frac{1}{\ga_n})^{2/3} - \hat{\B{k}} -4}{\langle 1/\ga_n \rangle +1} \]
tends to $+1$ as $\ga_n \to 0$, one concludes the statement of the lemma. 
\end{proof}

Recall the notation $\gb_0=1$, and $\gb_k= \prod_{i=1}^k \ga_i$, for $k\geq 1$.  
Also, recall the bi-sequence $B_{n, j}$, defined in~\refE{E:bisequence}. 
Consider the sets 
\[W^n_j=\left \{ w  \in \C{D}_j \; \big | \; \Im w\geq B_{n,j}- \frac{1}{2} \gb_j^{-1} \gb_{n-1} \log\frac{1}{\ga_n}-4\right \}.\]
Recall the set $\C{L}(\ga, T, l)$ defined in Section~\ref{SS:arithmetic}. 
In the following proposition $D$ is the constant satisfying Equations~\eqref{E:small-integral-1} and \eqref{E:small-integral-2} 
(or any constant larger than the one satisfying those equations.). 

\begin{propo}\label{P:wide-set-swallowed}
Let $T \in \D{R}$ be a constant satisfying 
\begin{equation}\label{E:T-margin}
T\geq (D+2)\frac{4\pi}{4\pi -1},
\end{equation}
and $m$ be a positive integer such that $\C{L}(\ga, T, m)$ is not empty.
Then, for all $n\in \C{L}(\ga,T,0)$ and for all integers $l_i$, for $ m+1 \leq i \leq n$, with $0 \leq l_i \leq a_{i-1} $, 
we have 
\[\gc_{m+1,l_{m+1}} \circ \gc_{m+2, l_{m+2}} \circ \dots \circ  \gc_{n, l_n}  (X_{\ga_n})  \subset W_m^n.\] 
\end{propo}

\begin{proof}
For arbitrary $w_n \in X_{\ga_n}$, define the sequence $w_n, w_{n-1}, \dots w_m$ according to 
\begin{gather*} 
w_{j-1}=\gc_{j, l_j}(w_j), \tfor m+1 \leq j \leq n.
\end{gather*}

As $n \in \C{L}(\ga, T ,m)$, by Proposition~\ref{P:nearly-right-height}-a, there exists $\gt\in \C{C}_{n}(\frac{1}{2\ga_n}-2\B{i})$ with 
\begin{equation}\label{E:wasted-height}
\Im \gt_j\geq B_{n,j}\geq \frac{T}{\ga_j},\; \tfor m+1 \leq j \leq n-1.
\end{equation}
Let $\gt'\in  \C{C}_{n}(\frac{1}{2\ga_n}-2\B{i})$ be the sequence defined along the same branches as the sequence 
$\langle w_j\rangle$, that is, 
\begin{gather*} 
\gt'_{j-1}=\gc_{j, l_j}(\gt'_j), \tfor m+1 \leq j \leq n.
\end{gather*}
By Lemma~\ref{L:heights-of-lifts}, we must have
\begin{equation}\label{E:safe-height-Cremer}
\Im \gt'_j \geq B_{n,j}-2,\; \tfor  m \leq j \leq n-1.
\end{equation}
  
By an inductive procedure we show that for all $j$ with  $m \leq  j \leq n-1$
\begin{equation}\label{E:recursive-height-Cremer}
\Im w_j \geq \Im \gt'_j - \frac{1}{4\pi} \gb_j^{-1} \gb_{n-1} \log \frac{1}{\ga_n} -  \gb_j^{-1} \sum_{k=j}^\infty \gb_k.
\end{equation}

By definition, $\Im \gt_{n-1}' = \Im \gt_{n-1} = \Im \gc_{n,0}(\frac{1}{2\ga_n} -2 \B{i})$. 
For $w_n\in X_{\ga_n}$, by Lemma~\ref{L:step3} and Proposition~\ref{P:density-high-levels-raw}, we have 
\begin{align*}
\Im w_{n-1} - \Im \gt_{n-1}' & \geq  \frac{1}{2} \Im \gt_{n-1}' +2 C'  - \Im \gt_{n-1}'    \\                         
& \geq - \frac{1}{2} \frac{1}{2 \pi } \log \frac{1}{\ga_n} - \frac{C'}{2} +C' 
 \geq - \frac{1}{4 \pi } \log \frac{1}{\ga_n}.
\end{align*}
Therefore, Equation~\eqref{E:recursive-height-Cremer} holds for $j=n-1$. 
 
Assume we have \refE{E:recursive-height-Cremer} for some $j$. We want to prove it for $j-1$.
By the hypothesis and Equation~\eqref{E:safe-height-Cremer}, $\Im w_j$ enjoys
\begin{align*}
\Im w_j &\geq \Im \gt'_j -  \frac{1}{4\pi} \gb_j^{-1} \gb_{n-1} \log \frac{1}{\ga_n} -  \gb_j^{-1} \sum_{k=j}^\infty \gb_k 
&& (\text{Eq.~\eqref{E:recursive-height-Cremer}}) \\
&\geq ( B_{n,j}-2 ) - \frac{1}{4\pi} \gb_j^{-1} \gb_{n-1} \log \frac{1}{\ga_n} - 2 
&& (\text{Eq.~\eqref{E:safe-height-Cremer} and Eq.~\eqref{E:trivial-bound}})\\
&\geq B_{n,j}  -  \frac{1}{4\pi} \gb_j^{-1} \gb_{n-1} \log \frac{1}{\ga_n} -  4 
&&\\
&\geq (1-\frac{1}{4\pi}) B_{n,j}-4 
&& (\text{Eq.~\eqref{E:arithmetic-height}})\\
&\geq  (1-\frac{1}{4\pi}) \frac{T}{\ga_j}-4 \geq  \frac{D}{\ga_j}
&& (\text{Eq.~\eqref{E:wasted-height} and  Eq.~\eqref{E:T-margin}})
\end{align*}
By the above inequality and Equations~\eqref{E:safe-height-Cremer}, \eqref{E:wasted-height}, and \eqref{E:T-margin}, 
the points $w_j$ and $\gt_j$ lie above the horizontal line $\Im w=D/\ga_j$.
On the other hand, since these points belong to the image of the same lift $\gc_{j+1,l_{i+1}}$, by Equation~\eqref{E:width-of-image-chi}
their real parts differ at most by $\hat{\B{k}}$. 
We may choose a piecewise horizontal and vertical curve $\gga_j$ connecting $w_j$ to $\gt_j'$ that lies above the horizontal line 
$\Im w=D/\ga_j$.
Hence, 
\begin{align*}
w_{j-1} -\gt_{j-1}' = \int_{\gga_j} \gc_{j,l_j}' =  \int_{\gga_j} (\gc_{j,l_j}'-\ga_j)  +  \int_{\gga_j} \ga_j  
\end{align*}
By Equations~\eqref{E:small-integral-1} and \eqref{E:small-integral-2}, $|\int_{\gga_j} (\gc_{j,l_j}'-\ga_j)|$ is uniformly bounded by 
$1$. 
By taking imaginary part of the above equation, and using Equation~\eqref{E:recursive-height-Cremer} for $j$ we obtain 
Equation~\eqref{E:recursive-height-Cremer} for $j-1$. 

Finally, by Equations \eqref{E:safe-height-Cremer}, and \eqref{E:recursive-height-Cremer} for $j=m$, as well as 
\eqref{E:trivial-bound}, imply the desired lower bound on $\Im w_m$.
That is, $w_m$ belongs to $W_m^n$.
\end{proof}

Now we are ready to prove the following stronger statement that implies \refP{P:cremer-sectors}.

\begin{propo}\label{P:cremer-sectors-stronger}
For every non-Brjuno $\ga\in \irr$ and every $f\in \IS_\ga \cup \{Q_\ga\}$ there exists a sequence of real numbers 
$\gd_i$, $i \geq 1$, converging to zero such that  
\[\limsup_{n\to\infty}  \frac{|G(n,\gd_n)|}{q_n}=1. 
\footnote{\text{It is likely that one can replace the $\limsup$ by $\lim$ in this proposition, but we do not need it here.}}\]
\end{propo}

\begin{proof}
The proof is just putting the above proposition and Lemma together. 
Define, 
\[\gd_{n+1}= \sup_{w\in W^n_0} |\gF_0^{-1}(w)|, n\geq 0.\]
Since $\ga$ is a non-Brjuno number, 
\[\lim_{n\to +\infty} B_{n,0}-\frac{1}{2}\ga_1\ga_2\dots\ga_{n-1}\log\frac{1}{\ga_n}-4= +\infty.\]
This implies that $\gd_n \to 0$ as $n \to +\infty$. 

Let $\gep>0$ be arbitrary. 
We need to find $j \in \D{N}$ such that $\gd_j \leq \gep$ and $|G(j, \gd_j)|/q_j \geq 1-\gep$. 
Let $\gd$ be the constant obtained in Lemma~\ref{L:density-on-hight-level} corresponding to this $\gep$. 

By Lemma~\ref{L:good-levels}, if $\ga$ is a non-Brjuno number, for any $T\in \D{R}$ the set $\C{L}(\ga,T,0)$ is non-empty.  
For $T > 0$, let $n=\min \C{L}(\ga, T, 0)$, where $n$ depends on $T$. 
By the definition of the set $\C{L}(\ga, T,0)$, as $T$ tends to $+\infty$, $n$ tends to $+\infty$ and $\ga_n$ tends to $0$.
That is because, if $n\in \C{L}(\ga, T, 0)$ we must have $B_{n,1} \geq T/\ga_1$ where $B_{n,1}$ is a finite number. 
Thus, as $T$ gets larger, $n$ must be larger. 
On the other hand, $n\in \C{L}(\ga, T, 0)$ requires that $B_{n,n-1} \geq T\ga_{n-1}^{-1}$, which by definition of the bi-sequence
implies that $-2 \ga_n+\log \ga_n^{-1} -B \geq T/\ga_{n-1}$. 
Hence, as $T$ tends to infinity, $\ga_n$ must tend to $0$. 
Therefore, we may choose $T>0$ large enough so that it satisfies Equation~\eqref{E:T-margin} and $n=\C{L}(\ga, T, 0)$ 
is large enough so that $\gd_n \geq \gep$ and $\ga_n \in (0, \gd)$. 

By Lemma~\ref{L:density-on-hight-level}, we have $|G_n|/q_n\geq 1-\gep$.  
And by definition, for all $j \in G_n$, $\gF_n(f_n\co {j}(J_n)) \subset X_{\ga_n}$. 
Then, by Proposition~\ref{P:wide-set-swallowed} with $m=0$, for $j\in G_n$, $\gF_n(f_n\co {j}(J_n))$ lifts to a set in $W^n_0$ under all 
possible lifts in the renormalization tower. 
By Proposition~\ref{P:lifts-vs-iterates}, for every $i \in G_n$ and every integer $k$ with $ i q_n \leq k < (i+1)q_n$, 
$f_0\co{k}(J_n)$ is contained in $B_{\gd_{n+1}}(0)$.  
Hence, 
\[\frac{|G(n+1, \gd_{n+1})|}{q_{n+1}} \geq \frac{|G_n|q_n}{q_{n+1}} \geq \frac{|G_n|q_n}{(a_n +1) q_n}\geq 1-\gep.\]
As $\gep$ was arbitrary, this finishes the proof of the proposition.
\end{proof}

\begin{rem}
One may extract an alternative proof of Proposition~\ref{P:cremer-sectors} from the above analysis. 
The steps in the above argument are set up to be used for the argument in the linearizable case.
But, briefly speaking, an alternative proof may go as follows. 
Given $\gep>0$ choose $D$ large enough so that the points $w\in \C{D}_0$ with $\Im w \geq D/\ga_0$ map into 
$B_\gep(0)$ under $\gF_0^{-1}$. 
Then, choose $T\in \D{R}$ that satisfies Equation~\eqref{E:T-margin}. 
The set $\C{L}(\ga, T,0)$ has infinite cardinality, and hence we may let $n$ tend to infinity within this set. 
We start with $\Im w_{n-1} \geq T/(2\ga_{n-1})$. 
Then, inductively, only using the five lines of inequalities in the proof of Proposition~\ref{P:wide-set-swallowed} 
to obtain the lower bound on $\Im w_j$, to show that $\Im w_0\geq D/\ga_0$. 
\end{rem}


\subsection{Proof of Proposition~\ref{P:siegel-sectors}}
Recall that for $m \geq 0$, $\gD(f_m)$ denotes the Siegel disk of $f_m$, and $\tilde{\gD}(f_m)$ 
denotes the image of $\gD(f_m) \cap \C{P}_m$ under $\gF_m$. 
We consider the following two scenarios:

\medskip

\begin{description}
\item[$\mathscr{A}$] for every $T\in \D{R}$ there exist $m_0\in \D{N}$ and infinitely many integers $m>m_0$ such 
that there is $\gt \in \C{C}_{m}(\frac{1}{2\ga_m}-2\B{i})$ satisfying $\Im \gt_j \geq T/\ga_j$, for all integers $j$ with  
$m_0+1 \leq j \leq  m-1$.

\medskip

\item[$\mathscr{B}$] there exist real constants $T$ and $C$ as well as infinitely many integers $m$ such that 
for all $w\in \C{D}_m$ with $\Im w \geq \frac{T}{\ga_m}+C$ we have $\gF_m^{-1}(w)\in \gD(f_m)$. 
\end{description}

\medskip

\begin{lem}\label{L:complementary-conditions}
For every irrational $\ga\in \irr$, at least one of $\mathscr{A}$ or $\mathscr{B}$ holds. 
\end{lem}

\begin{proof}
By Proposition~\ref{P:nearly-right-height}-a, there are constants $M$ and $B$ such that the bi-sequence $B_{n,i}$ defined 
in Section~\ref{SS:arithmetic} using $B$ satisfies the following. 
For every $n\geq 1$ and every $j$ with $0 \leq j \leq n-1$ we have $B_{n,j} \leq \Im \gt_j \leq B_{n,j}+M$.
Recall the set $\C{L}(\ga,T, l)$ defined in Section~\ref{SS:arithmetic}. 

Assume that $\mathscr{A}$ does not hold. 
That means there is $T_0\in \D{R}$ such that for every $m_0\in \D{N}$ there is $m'\in \D{N}$ such that for every $m \geq m'$ 
and every $\gt\in \C{C}_{m}(\frac{1}{2\ga_m}-2\B{i})$ there is an integer $j$ with $m-1\leq j \leq m_0+1$ such that 
$\Im \gt_j < T_0/\ga_j$. 
In particular, $\mathscr{A}$ does not hold for any constant bigger than $T_0$. 
Below we assume that $T_0>0$.

By the first paragraph above, for some choice of $\gt\in \C{C}_{m}(\frac{1}{2\ga_m}-2\B{i})$, $B_{n,j} \leq \Im \gt_j$. 
Then, by the second paragraph above, this implies that $\C{L}(\ga,T_0,l)$ is finite for every $l\in \D{N}$. 
Combining this with \text{Lemmas~\ref{L:good-levels}-c}, we conclude that there exists a sequence of integers 
$\ell_0,\ell_1,\ell_2, \dots$, tending to $+\infty$, such that  
\[\sup_{j\geq 0} \; (\lim_{m\to \infty} B_{m,\ell_j}-T_0\ga_{\ell_j}^{-1}) <\infty.\]
Then, by \refP{P:nearly-right-height}-b, we obtain
\[\sup_{j\geq 0} \; (\max_{w\in \partial (\gF_j(\gD(f_j)))} \Im w - T\ga_{\ell_j}^{-1})<\infty.\]
This implies that statement $\mathscr{B}$ must hold.  
\end{proof}

\begin{lem}\label{L:small-rotations-A}
Assume that $\ga\in \irr$ satisfies property $\mathscr{A}$ and $T$ is a real constant satisfying Equation~\eqref{E:T-margin} and  
\begin{equation}\label{E:T-small-rotations}
T\geq M+3,
\end{equation}
where $M$ is the constant in Proposition~\ref{P:nearly-right-height}.
Let $m_0\in \D{N}$ and the infinite set of integers $A$ be such that for all $m\in A$ there is 
$\gt \in \C{C}_{m}(\frac{1}{2\ga_m}-2\B{i})$ satisfying $\Im \gt_j \geq T/\ga_j$, for all integers $j$ with  $m-1\leq j \leq m_0+1$.
Then, 
\[\liminf_{m\in A, m\to +\infty} \ga_m=0.\]
\end{lem}

\begin{proof}
First we show that if some $\ga\in \irr$ satisfies property $\mathscr{A}$, then
\begin{equation}\label{E:unbounded-type}
\liminf_{i\to \infty} \ga_i=0.
\end{equation}
By property $\mathscr{A}$, for every $T\in \D{R}$ there are positive integers $m_0$ and $m\geq m_0+2$
and $\gt\in \C{C}_{m}(\frac{1}{2\ga_m}-2\B{i})$ such that $\Im \gt_{m-1}\geq T/\ga_{m-1}$.
Then, by Lemma~\ref{L:step3},
\[ T \leq \frac{T}{\ga_{m-1}} \leq \Im \gt_{m-1} \leq \frac{1}{2\pi} \log \frac{1}{\ga_m} + C'.\]
That is, for every $T>0$ there is $m\in \D{N}$ such that $\log \ga_m^{-1} \geq 2\pi(T-C')$.
This implies Equation~\ref{E:unbounded-type}.

Let $m_0$ and $A$ be as in the lemma. 
If the set $\D{N} \setminus A$ has finite cardinality, then the statement of the lemma follows from Equation~\eqref{E:unbounded-type}. 
Below we assume that $\D{N} \setminus A$ has infinite cardinality. 

Assume by contradiction that 
\[\gh=\inf \{ \ga_i \mid i\in A\} >0.\]
Let
\[\gm= \min \{\gh/2, T \big(\frac{1}{2\pi} \log \frac{1}{\gh}+C'\big)^{-1})\}.\]
By Equation~\eqref{E:unbounded-type}, there is $n\in \D{N} \setminus A$ such that $\ga_n < \gm$ and 
$n \geq m_0+2$. 
Then, let $m$ be the smallest element of $A$ bigger than $n$. 
Let $n'$ be the largest integer in the interval $[n,m]$ such that $\ga_{n'}< \gm$.
Since, $\ga_m \geq \gh \geq 2 \gm$, $n'< m$. 
By definition, for all integers $i$ with $n' +1 \leq i \leq m $ we have $\ga_i \geq \gm$.

First we note that $m \geq n'+2$.  
This is because, by Lemma~\ref{L:step3},
\[ \frac{T}{\ga_{m-1}} \leq \Im \gt_{m-1} \leq \frac{1}{2\pi} \log \frac{1}{\ga_m} + C' \leq \frac{1}{2\pi}\log \frac{1}{\gh} + C' .\]
By the definition of $\gm$, the above inequality implies that $\ga_{m-1} \geq \gm$.
Thus, since $\ga_{n'} < \gm$, $n'\neq m-1$.

Let $B$ and $M$ be the constants in Proposition~\ref{P:nearly-right-height}, and $B_{n,i}$ be the 
bi-sequence defined with the constant $B$. 
By Proposition~\ref{P:nearly-right-height}-a, there is $\gt'$ in $C_m(\frac{1}{2\ga_m}-2\B{i})$
such that $\Im \gt_j' \leq B_{m,j} +M$ for all $j$ with $0\leq j \leq m-1$. 
Then, by Lemma~\ref{L:heights-of-lifts}, $|\Im \gt_j -\Im \gt_j'|\leq 2$, for all $j$ with $m_0+1 \leq j \leq m$. 

Putting the above paragraphs together, we obtain,  
\begin{multline*}
\frac{T}{\ga_{n'}}
 \leq \Im \gt_{n'} 
 \leq \Im \gt_{n'}'+2
 \leq B_{m,n'}+M+2 \\
 \leq -2\gb_{n'}^{-1} \gb_m + \gb_{n'}^{-1} \sum_{j=n'+1}^{m} \gb_{j-1} (\log \frac{1}{\ga_j}-B) +M+2 \\
 \leq (\log \frac{1}{\gm}-B) \sum_{j=0}^{m-n'-1} \big(\frac{1}{2}\big)^{j}  +M+2
 \leq 2 \log \frac{1}{\gm} +M+2. 
\end{multline*}

Thus, as $2\log x < x$ for all $x\in (2, +\infty)$, the above inequality implies that 
\[ T \leq \ga_{n'} (2\log \frac{1}{\gm}+M+2) < 2\gm \log \frac{1}{\gm} +M+2 < M+3.\]
This contradicts the choice of $T$ in the lemma.
\end{proof}

\begin{rem}
In Lemma~\ref{L:small-rotations-A}, since the $\liminf$ is equal to zero along any infinite set $A$ 
on which the property $\mathscr{A}$ holds for the constant $T$, one may conclude that indeed the 
$\liminf$ can be replaced by limit. 
However, we have stated the least information enough to prove Proposition~\ref{P:siegel-sectors}. 
\end{rem}

\begin{proof}[Proof of Proposition~\ref{P:siegel-sectors} when $\mathscr{A}$ holds]  
The argument is similar to the one for the nonlinearizable maps. 
By Proposition~ \ref{P:nearly-right-height}-b, there is a constant $B$ such that with the corresponding bi-sequence $B_{n,j}$,
\[H=\sup_{j\in \D{N}} |\lim _{n\to\infty}B_{n,j}-\max_{w\in \partial \tilde{\gD}(f_j)} \Im w|,\]
is a finite number.

Fix an arbitrary $\gd>0$. 
Choose $\gd'>0$ such that if $w \in  B_{\gd'}(\tilde{\gD}(f_0)) \cap \C{D}_0$ then $\gF_0^{-1}(w) \in B_\gd(\gD(f_0))$. 
Recall that by Lemma~\ref{L:well-contains} for every $m\geq 0$ and every integer $j$ with $0 \leq j \leq a_m $, 
the $\gd_0$-neighborhood of $\gc_{m+1,j}(\C{D}_{m+1})$ is contained in $\C{D}_m$.
An elementary estimate shows that the Poincar\'e metric $\gr_m |dw|$ on $\C{D}_m$ satisfies 
$\gr_m(w) \leq 2/d(w, \partial \C{D}_m)$. 
This implies that for (every such $m$ and $j$, as well as) every $w \in \gc_{m+1,j}(\C{D}_{m+1})$ with 
$\Im w > \max_{w\in \partial \tilde{\gD}(f_m)} \Im w -H-6$ there is $w'\in \partial \tilde{\gD}(f_m)$ such that the hyperbolic distance 
between $w$ and $w'$ in $\C{D}_m$ is uniformly bounded from above by $2(H+6)/\gd_0$. 
Then, by Lemma~\ref{L:contraction}, there is $m_0'$ such that for every $m\geq m_0'$, every $w \in \gc_{m+1,j}(\C{D}_{m+1})$ with 
$\Im w \geq\max_{w\in \partial \tilde{\gD}(f_m)} \Im w -H-6$ and all integers 
$l_j$, $1\leq j \leq m$, with $0\leq l_j \leq a_{j-1}$, 
the point $\gc_{1,l_1}\circ \gc_{2, l_2} \circ \dots \circ \gc_{m, l_m}(w)$ belongs to $B_{\gd'}(\tilde{\gD}(f_0)) \cap \C{D}_0$. 

Let $T$ be a constant satisfying Equations~\eqref{E:T-margin} and \eqref{E:T-small-rotations}.  
Since $\mathscr{A}$ holds for $T$, there exists a positive integer $m_0$, and an infinite set of integers denoted by $A$ such that 
for all $m\in A$ there is $\gt \in \C{C}_{m}(\frac{1}{2\ga_m}-2\B{i})$ with  
\[\Im \gt_j> \frac{T}{\ga_j},\; \tfor  m_0+1 \leq j \leq m-1.\]
Note that by making $T$ larger, $m_0$ becomes larger. 
In particular, we may assume that beside satisfying Equations~\eqref{E:T-margin} and \eqref{E:T-small-rotations}, 
$T$ is large enough so that the corresponding $m_0$ is bigger than $m_0'$. 

Recall the set $G_m$ defined in Section~ \ref{S:SectorsAtDeepLevels}.  
By Lemmas~\ref{L:density-on-hight-level} and \ref{L:small-rotations-A}, we obtain 
\[\lim_{m \in A, m\to\infty} \frac{|G_m|}{a_m +1}=1.\]

The Brjuno sum for $\ga_{m_0}$, that is $\sum_{m=m_0+1}^{+\infty} \gb_{m_0}^{-1} \gb_{m-1} \log \ga_m^{-1}$, is finite. 
Also, by Lemma~\ref{L:good-levels}, $\lim_{n\to +\infty} B_{n,m_0}$ exists. 
Hence, there is $m_0''>0$ such that for all $m\geq m_0$, $\gb_{m_0}^{-1} \gb_{m-1} \log \ga_m^{-1} < 2$ and 
$|\lim_{n\to +\infty} B_{n,m_0} -B_{m,m_0}|<1$.  

Fix $m \in A$ such that $m \geq m_0''$, and let $w_m \in X_{\ga_m}$ be arbitrary, where $X_{\ga_m}$ is defined in 
Equation~\eqref{E:X}. 
Given integers $l_i$, $m \leq i \leq m_0+1$, with $0 \leq l_i \leq a_{i-1}$, 
define the sequence of points 
\[w_{i-1}=\gc_{i,l_i}(w_i), \tfor m \leq i \leq m_0+1.\] 
By Proposition~\ref{P:wide-set-swallowed}, we have 
\begin{align*}
\Im w_{m_0} 
&\geq B_{m, m_0}-\frac{1}{2} \gb_{m_0}^{-1} \gb_{m-1} \log\frac{1}{\ga_m}-4 && \\
& > \lim_{n\to +\infty} B_{n,m_0} -6. &&(\text{since }m\geq m_0'')
\end{align*}
Hence,  
\[ \Im w_0 \geq \max_{w\in \partial \tilde{\gD}(f_j)} \Im w -H -6.\]
Now since $m>m_0'$ (the argument in the second paragraph), all further lifts of $w_0$ to the level $0$ belong to the 
$\gd'$ neighborhood of $\tilde{\gD}(f_0)$. 

By definition, for all $j\in G_m$, we have $f_m\co{j}(J_m) \subset \C{P}_m$ and $\gF_m \circ f_m\co{j}(J_m) \subset X_{\ga_m}$. 
Therefore, by Proposition~\ref{P:lifts-vs-iterates}, and that $m\geq m_0'$, we conclude that 
\[1 \geq \lim_{m\to\infty, m\in A}  \frac{|H(m_i+1, \gd)|}{q_{m_i+1}} 
\geq \lim_{m\to\infty, m\in A} \frac{|G_m|}{ a_m +1}=1.\]

This finishes the proof of the Proposition when $\mathscr{A}$ holds.
\end{proof}

As a corollary of the above proof we state the following property for future reference. 

\begin{cor}\label{C:tips-near-Siegel-disk}
Assume $\ga\in \irr$ satisfies property $\mathscr{A}$. 
Then for every $\gd'>0$, there are infinitely many $m\in \D{N}$ such that for all $\gt\in \C{C}_m(\frac{1}{2\ga_m}+2\textnormal{\B{i}})$, 
we have $\gt_0\in B_{\gd'}(\tilde{\gD}(f_0)) \setminus \tilde{\gD}(f_0)$.
\end{cor}

\begin{proof}
By virtue of the proof of Proposition~\ref{P:siegel-sectors} when $\mathscr{A}$ holds, we only need to 
show that for all $m\in \D{N}$, $\gt_0 \nin \tilde{\gD}(f_0)$. 
This is because, since $\Im \gt_m=-2 $, $\gt_m$ does not belong to $\tilde{\gD}(f_m)$, 
and since the changes of coordinates preserve the Siegel disks, $\gt_0$ may not belong to $\tilde{\gD}(f_0)$. 
\end{proof}

\begin{proof}[Proof of Proposition~\ref{P:siegel-sectors} when $\mathscr{B}$ holds]
For $n\geq 0$, consider the sets 
\[E_n=\bigcup_{i=0}^{a_n} \gc_{n+1,i}(\C{D}_{n+1}).\]
By Lemma~\ref{L:well-contains}, for every $n\geq 0$, $\gd_0$-neighborhood of $E_n$ is contained in $\C{D}_n$. 

Fix an arbitrary $\gd>0$. 
Choose $\gd'>0$ such that if $w \in  B_{\gd'}(\tilde{\gD}(f_0)) \cap \C{D}_0$ then $\gF_0^{-1}(w) \in B_\gd(\gD(f_0))$. 
As discussed in the proof of Proposition~\ref{P:siegel-sectors} when $\mathscr{A}$ holds, 
Lemma~\ref{L:well-contains} implies that for every $H>0$ there is $m_0'\geq 1$ satisfying the following property. 
For every $m\geq m_0'$ and all integers $l_j$, $1\leq j \leq m$, with $0\leq l_j \leq a_{j-1}$, 
as well as all $w\in E_m$ with either Euclidean distance $d(w, \partial \tilde{\gD}(f_m))\leq H$ or hyperbolic distance 
$d_{\textnormal{hyp}}(w, \partial \tilde{\gD}(f_m))\leq H$, 
$\gc_{1,l_1}\circ \gc_{2, l_2} \circ \dots \circ \gc_{m, l_m}(w)$ belongs to $B_{\gd'}(\tilde{\gD}(f_0)) \cap \C{D}_0$. 

Let $T$, $C$, and the sequence of integers $m_1 < m_2 < m_3< \dots $ be the data obtained from property $\mathscr{B}$.
We break the proof into two cases.

\medskip

{\em I)} We have $\limsup_{i\to \infty} \ga_{m_i}>0$. 

\medskip 

Let $\gh>0$ and integers $n_1< n_2 < n_3< \dots $ be such that $\ga_{n_i}>\gh$ for all $i\geq 1$. 
It follows from property $\mathscr{B}$ that for all $i\geq 1$ and all $w\in E_{n_i}$ there is $w'\in \partial \tilde{\gD}(f_{n_i})$ with 
$d(w,w') \leq T/\gh+C$. 
Therefore, by the above paragraph, there is $n'$ in the sequence $n_i$, such that for all integers $l_j$, $1 \leq j \leq n'$, 
with $0\leq l_j \leq a_{j-1}$,  
$\gc_{1,l_1}\circ \gc_{2, l_2} \circ \dots \circ \gc_{n', l_{n'}}(E_{n'})$ is contained in $B_{\gd'}(\tilde{\gD}(f_0)) \cap \C{D}_0$.
Here, we have used that $\gc_{n,i}(\tilde{\gD}(f_n)) \subset \tilde{\gD}(f_{n-1})$, for all $n\geq 1$ and $i$ with 
$0\leq i \leq a_{i-1}$, which follows from Propositions~\ref{P:Siegel-disk-enclosed} and \ref{P:invariance-of-pc}.
Thus, by Proposition~\ref{P:lifts-vs-iterates}, for all $m \geq n'+1$, we have 
\[\frac{|H(m, \gd)|}{q_{m}}=1.\]
As $\gd$ was arbitrary, this finishes the proof of the proposition in this case.

\medskip

{\em II)} We have $\lim_{i\to \infty} \ga_{m_i}=0$. 

\medskip

Recall the constant $\hat{\Bk}$ obtained in Lemma~\ref{P:sector-geometry}. 
Consider the sets  
\[E'_n= \{w \in E_n \mid  \frac{1}{4\ga_n}-\hat{\Bk}\leq \Re w\leq \frac{3}{4\ga_n}+\hat{\Bk}\}.\]
Sine the sequence $\ga_i$ tends to $0$, there is $i_0\geq 1$ such that for all $i \geq i_0$, $E_n' \subset E_n$. 

By Proposition~\ref{P:fine-estimate}, for every $i \geq i_0$ the derivative of $\gc_{m_i,0}$ is comparable to $1/\ga_{m_i}$ on 
$E'_{m_i}$, with constants independent of $i$. 
Therefore, there is a constant $H>0$ such that for all $i \geq i_0$ and all integers $l_i$ with 
$0 \leq l_i \leq a_{m_i-1}$, the set $\gc_{m_i,l_i}(E'_n)$ has Euclidean diameter bounded from above by $H$.
Thus, by the second paragraph in this proof, there exists $i_1 \geq i_0$, such that for all $i\geq i_1$ all integers $l_j$, 
$1 \leq j \leq m_i$, with $0\leq l_j \leq a_{j-1}$,  
$\gc_{1,l_1}\circ \gc_{2, l_2} \circ \dots \circ \gc_{m_i, l_{m_i}}(E_{m_i})$ is contained in $B_{\gd'}(\tilde{\gD}(f_0)) \cap \C{D}_0$.

Each set $E'_n$ contains a vertical strip of width $1/(2\ga_n)+2\hat{\B{k}}$. 
Thus, by Lemma~\ref{L:well-contains}, for all $i\geq i_0$, $\gc_{m_i+1,l}(\C{D}_{m_i+1})$ is contained in $E'_{m_i}$, 
for at least half of the integers $l$ with $0 \leq l \leq a_{m_i}-1$ . 

Fix an arbitrary $\gep>0$. 
Choose an integer $i_2\geq i_1$ such that $(1/2)^{i_2-i_1} <\gep$.
Now consider an integer $m\geq m_{i_2}+1$. 
From all the lifts $\gc_{m_{i_2}}, l_{i_2} \circ \dots \circ \gc_{m, l_m}(\C{D}_m)$, at least $1/2$ of them lie in 
$E_{m_{i_2}}$ and hence, their further lifts lie in $B_{\gd'}(\tilde{\gD}(f_0)) \cap \C{D}_0$. 
From all the remaining lifts up to level $m_{i_2}$, at least half of the lifts 
$\gc_{m_{i_2}-1}, l_{i_2-1} \circ \dots \circ \gc_{m, l_m}(\C{D}_m)$ lie in $E_{m_{i_2}-1}$ and hence, their further lifts 
lie in $B_{\gd'}(\tilde{\gD}(f_0)) \cap \C{D}_0$, and so on. 
Thus, putting all this together, Proposition~\ref{P:lifts-vs-iterates} implies that for all $m\geq m_{i_2}+1$, we have 
\[\frac{|H(m+1, \gd)|}{q_{m+1}} \geq 1-\gep.\]
This finishes the proof of the proposition in this case.
\end{proof}

\begin{proof}[Proof of Theorem~\ref{T:small-cycle-property}]
For $m\geq 0$ let $\gs_m$ denote the non-zero fixed point of $f_m$ that lies on the boundary of $\C{P}_m$.
Recall that by Lemma~\ref{L:conjugacy}, $\gY_m: \C{P}_m \to \C{P}_0$ conjugates $f_m$ on $\C{P}_m' \subset \C{P}_m$
to $f_0\co{q_m}$ on $\gY_m(\C{P}'_m)$, and $\gy_m: \C{P}_m \to \C{P}_{m-1}$ conjugates $f_m$ on $\C{P}_m'$
to some iterate of $f_{m-1}$ on $\gy_m(\C{P}'_m)$.   
Since, $\C{P}_m$ is bounded by piecewise analytic curves, $\gy_m$ and $\gY_m$ extend to some continuous maps on 
the closure of $\C{P}_m$. 
In particular, $\gY_m$ maps $\gs_m$ to a periodic point of $f_0$ of period $q_m$. 
Similarly, $\gy_m$ maps $\gs_m$ to a periodic point of $f_{m-1}$ of period say $b_{m-1}$.
We want to show that every neighborhood of $\gD(f_0)$ contains the cycle of infinitely many periodic points of type 
$\gY_m(\gs_m)$, for some $m \in \D{N}$.

For $m \geq 1$, define 
\[\C{O}_{m-1}=\{\gF_{m-1}(\gy_m(\gs_m))+j \mid j\in \D{Z}, 0\leq j \leq b_{m-1}-1\}.\]
By definition of $\gy_m$ and $\ex$, the set $\C{O}_{m-1}$ projects under $\ex$ onto the non-zero fixed point of $\C{R}(f_{m-1})$, 
which is either $\gs_m$ or $s(\gs_m)$. 
Lemma~\ref{L:size-of-sigma} implies that there is a constant $C_1'>0$, independent of $m$, such that 
\begin{equation}\label{E:height-O}
\big|\Im \C{O}_{m-1} - \frac{1}{2\pi} \log \frac{1}{\ga_m}\big| \leq C'_1,
\end{equation}
and by Lemma~\ref{L:well-contains}, $\gd_0$ neighborhood of $\C{O}_{m-1}$ is contained in $\C{D}_{m-1}$.  

Fix an arbitrary $\gd>0$, and choose $\gd'>0$ such that for every $w\in \C{D}_0 \cap B_{\gd'} (\tilde{\gD}(f_0))$ 
we have $\gF_0^{-1}(w) \in B_\gd(\gD(f_0))$.
We break the rest of the argument into two cases. 

\medskip

{\em I)} Assume that statement $\mathscr{A}$ holds for $\ga$. 

\medskip

\noindent By Lemma~\ref{L:step3}, for every $\gt\in \C{C}_m(\frac{1}{2\ga_m}-2\B{i})$ we have 
\[ \big|\Im \gt_{m-1} - \frac{1}{2\pi} \log \frac{1}{\ga_m}\big| \leq C'.\]
Thus, by Lemma~\ref{L:well-contains}, for every element of $\C{O}_{m-1}$ there is an integer $l$ with 
$0 \leq l \leq a_{m-1}$ such that the hyperbolic distance between that element and 
$\gc_{m,l}(\frac{1}{2\ga_m}-2\B{i})$ is uniformly bounded from above by a constant say $H$, 
depending only on $C'$, $C'_1$, and $\gd_0$. 
By the uniform contraction in Lemma~\ref{L:contraction}, there is $m_0\geq 1$ such that if $m\geq m_0$ then
any two points at hyperbolic distance bounded by $H$ are mapped to two points at Euclidean distance bounded from above 
by $\gd'/2$. 

On the other hand, by Corollary~\ref{C:tips-near-Siegel-disk} for $\gd'/2$, there are infinitely many $m$ such that for all 
$\gt \in \C{C}_m(\frac{1}{2\ga_m}+2 \B{i})$ we have $\gt_0 \in B_{\gd'/2}(\tilde{\gD}(f_0))$. 
Combining this with the above paragraph, we conclude that for all such $m\geq m_0$ all lifts of elements of 
$\C{O}_{m-1}$ to the level $0$ are contained in $B_{\gd'}(\tilde{\gD}(f_0)$.
Therefore, by Proposition~\ref{P:lifts-vs-iterates}, the cycle of $\gY_m(\gs_m)$ is contained in $B_\gd(\gD(f_0))$.

\medskip

{\em II)} Assume that statement $\mathscr{B}$ holds for $\ga$.

\medskip

\noindent Let $T$, $C$, and $m_1< m_2 < m_3 < \dots$ be the quantities provided by property $\mathscr{B}$. 
By definition, for all $i \geq 1$, $\Im \tilde{\gD}(f_{m_i}) \subset [-2, T/\ga_{m_i}+C]$. 
In particular, there is $w_{m_i}  \in \partial \tilde{\gD}(f_{m_i})$ with $\Re w_{m_i}=1/(2\ga_{m_i})$ and 
$\Im w_{m_i} \in [-2, T/\ga_{m_i}+C]$. 
By Proposition~\ref{P:invariance-of-pc}, $\gc_{m_i,l}(w_{m_i}) \in \partial \tilde{\gD}(f_{m_i-1})$, 
for all $l$ with $0\leq l \leq a_{m_i-1}$.
Moreover, by Proposition~\ref{P:fine-estimate}, there is a constant $C_2>0$, independent of $i$ and $l$, such that 
\[|\Im \gc_{m_i, l}(\frac{1}{2\ga_{m_i}}+2\B{i}) - \Im \gc_{m_i,l}(w_{m_i})| \leq C_2.\] 
Thus, by Lemma~\ref{L:step3}, we have 
\[|\Im \gc_{m_i,l}(w_{m_i}) - \frac{1}{2\pi}\log \frac{1}{\ga_{m_i}} | \leq C_2+C'.\] 

Combining the above inequality with Equation~\eqref{E:height-O}, and using Lemma~\ref{L:well-contains}, 
we conclude that every element of $\C{O}_{m_i}$ lies within uniformly bounded (depending on $C'_1+C_2+C'$ and $\gd_0$) 
hyperbolic distance from $\partial \tilde{\gD}(f_{m_i-1})$. 
Then, by the contraction in Lemma~\ref{L:contraction}, for sufficiently large $i$, $\C{O}_{m_i}$
lifts into a set of points in $B_{\gd'}(\tilde{\gD}(f_0))$. 
Finally, Proposition~\ref{P:lifts-vs-iterates}, implies that the cycle of $\gY_{m_i}(\gs_{m_i})$ is contained in $B_\gd(\gD(f_0))$.
\end{proof}
\section{Unique ergodicity}\label{S:Unique-Ergodicity}
We work with the following equivalent definition of unique ergodicity, see \cite[Theorem 9.2]  {Man87}. 
A continuous map $f:X \to X$, where $X$ is a compact metric space, is uniquely ergodic if
for every continuous function $\gf :X \to \D{R}$ and every $x\in X$ the limit of the Birkohff averages 
\[\lim_{n\to\infty} \frac{1}{n}\sum_{j=0}^{n-1} \gf (f\co{j}(x))\]
exists and is independent of $x$.

Since we have a simple candidate for the unique invariant measure for the non-linearizable maps, that is the Dirac 
measure at $0$, the proof of Theorem~\ref{T:unique-ergodicity} becomes slightly simpler for non-linearizable maps. 
Also, it conveys the idea of proof for the linearizable ones where the invariant measure is more complicated. 
Hence, although both cases may be treated simultaneously, we present the proof for the non-linearizable maps first.   

\subsection{Non-linearizable maps}

\begin{proof}[Proof of  \refT{T:unique-ergodicity} when $\ga$ is non-Brjuno] 
We assume $N$ is the integer determined in Section~\ref{S:continued-fraction--renormalizations} and 
is subject to the inequality in Equation~\eqref{E:high-type-restriction}. 

Let $\gf:\pc(f_0)\to \D{R}$ be a continuous function with the supremum norm 
\[M=\max_{z\in \pc(f_0)} |\gf(z)|.\]
Let $\gep>0$ be an arbitrary constant. 
There is $\gd>0$ such that 
\[\forall w\in B_\gd(0), |\gf(w)-\gf(0)|<\gep.\] 
Recall the set $G(n,\gd)$ defined in Section~\ref{S:geometry-arithmetic}. 
By Proposition~\ref{P:cremer-sectors}, there is an integer $n>0$ such that 

\[\frac{|G(n,\gd)|}{q_n}\geq 1-\gep.\]

If $z=0$, then clearly the Birkhoff average of $\gf$ along the orbit of $0$ is the constant sequence with terms equal to $\gf(0)$.
We need to show that these averages along the orbit of every point in $\pc(f_0)$ is equal to $\gf(0)$. 

Recall that by proposition~\ref{P:pc-neighbor}, $\pc(f_0)$ is contained in $\gU^n$. 
Given a non-zero $\gz \in \pc(f_0)$, by Proposition~\ref{P:orbits-organized}, there is a one-to-one map 
$\gt: \{i\in \D{Z}\mid 0\leq i\leq q_n-1\} \to \{i\in \D{Z}\mid 0\leq i\leq q_n-1\}$, depending on $\gz$, 
such that $f_0\co{i}(\gz)\in f_0\co{\gt(i)}(I^n)$, for all $i$ with $0\leq i\leq q_n-1$.  
Then, 
\begin{equation}\label{E:average-one-turn}
\begin{aligned}
\bigg|\frac{1}{q_n}\sum_{k=0}^{q_n-1} & \gf(f_0\co{k}(\gz))  - \gf(0) \bigg | \\
&   \leq \frac{1}{q_n} \sum_{\substack{0\leq k\leq q_n-1\\ f_0\co{k}(\gz)\in G(n, \gd)}}  |\gf(f_0\co{k}(\gz))  - \gf(0) | 
 + \frac{1}{q_n} \sum_{\substack{0\leq k\leq q_n-1 \\ f_0\co{k}(\gz) \nin G(n, \gd)}} | \gf(f_0\co{k}(z)) - \gf(0)| \\
&  \leq \frac{1}{q_n} |G(n,\gd)| \cdot \gep + \frac{1}{q_n} (q_n- |G(n, \gd)|) \cdot 2 M \\
& \leq \gep + \gep 2 M = \gep (1+ 2M).
\end{aligned}
\end{equation}

Fix an arbitrary $z$ in $\pc(f_0)$. 
Let $N$ be a positive integer. 
Dividing $N-1$ by $q_n$ we obtain non-negative integers $m$ and $r$ such that $N-1= m q_n+r$ with $0\leq r \leq q_n-1$.
In particular, there is $N_0>0$ such that for all $N\geq N_0$ we have  
\[\frac{r}{N} 2 M \leq \gep.\] 
Applying Equation~\eqref{E:average-one-turn}, to the points $f_0\co{iq_n}(z)$, for $i\geq 0$, we conclude that for every $N\geq N_0$
we have 
\begin{align*}
\bigg|\frac{1}{N}\sum_{k=0}^{N-1}& \gf(f_0\co{k}(z))-\gf(0)\bigg| \\
& \leq \frac{1}{N} \sum_{i=0}^{m-1} \sum_{k=i q_n}^{(i+1)q_n-1} |\gf(f_0\co{k}(z))-\gf(0)| 
+  \frac{1}{N} \sum_{k=mq_n}^{N-1} |\gf(f_0\co{k}(z))-\gf(0)|\\
& = \frac{m q_n}{N} \frac{1}{m}\sum_{i=0}^{m-1}  \frac{1}{q_n} \sum_{k=i q_n}^{(i+1)q_n-1} |\gf(f_0\co{k}(z))-\gf(0)| 
+  \frac{1}{N} \sum_{k=mq_n}^{N-1} |\gf(f_0\co{k}(z))-\gf(0)|\\   
& \leq 1 \cdot \frac{1}{m} \cdot m \cdot \gep (1+ 2M)+ \frac{1}{N} \cdot r \cdot 2 M \\
& \leq \gep (1+ 2M)+ \gep = \gep (2+ 2M).
\end{align*}
As $\gep$ was chosen arbitrarily, we conclude that 
\[\lim_{N\to\infty}  \frac{1}{N}\sum_{k=0}^{N-1} \gf(f_0\co{k}(z))=\gf(0) \]
\end{proof}

\subsection{Linearizable maps}

\begin{proof}[Proof of  \refT{T:unique-ergodicity} when $\ga$ in Brjuno] 
The integer $N$ is determined in Section~\ref{S:continued-fraction--renormalizations} and is subject 
to Equation~\eqref{E:high-type-restriction}. 

Let $\gf:\pc(f_0)\to \D{R}$ be a continuous map with 
\[M=\max_{z\in \pc(f_0)} |\gf(z)|< +\infty.\]

Fix arbitrary $z\in \pc(f_0)$ and $\gep>0$. 
There is $\gd'>0$ such that for all $w$, $w'$ in $\pc(f_0)$ with $|w'-w|<\gd'$ we have $|\gf(w')-\gf(w)|<\gep$. 

Recall that by Proposition~\ref{P:pc-neighbor}, $\pc(f_0)$ is contained in every $\gU^n$, $n\geq 0$. 
Also, recall the definition of the sets $H(n, \gd)$ from Section~\ref{SS:nearby-orbits}. 
Using \refP{P:siegel-sectors}, there are integer $n_0$ and real $\gd_0>0$ such that 
\[\forall n\geq n_0, \forall \gd\leq \gd_0, \forall k \in H(n,\gd), \diam (f_0\co{k}(I^n)\setminus \gD(f_0))<\gd'.\]
By the same proposition, there is an integer $n \geq n_0$ such that 
\[\frac{|H(n, \gd)|}{q_n}\geq 1-\gep.\]
From Proposition~\ref{P:orbits-organized}, there is a permutation of the set $\{i\in \D{Z} \mid  0\leq i\leq q_n-1\}$ such that 
$f_0\co{i}(z)\in f_0\co{\gt(i)}(I^n)$, for all $i$ with $0\leq i\leq q_n-1$. 
Given $w \in \pc(f_0)$, let $\gr$ be a permutation of the set $\{i\in \D{Z} \mid  0\leq i\leq q_n-1\}$, obtained 
from Proposition~\ref{P:orbits-organized}, such that $f_0\co{i}(w) \in f_0\co{\gr(i)}(I^n)$. 
Note that for all $i$ with $0\leq i \leq q_n-1$, we have $f_0\co{\gr^{-1}(\gt(i))}(w) \in f_0\co{\gt(i)}(I^n)$. 
We have 

\begin{equation}\label{E:average-one-turn-Brjuno}
\begin{aligned}
\bigg|\frac{1}{q_n}\sum_{k=0}^{q_n-1} \gf(f_0\co{k}(z)) &- \frac{1}{q_n}\sum_{k=0}^{q_n-1} \gf(f_0\co{k}(w))\bigg | \\
& \leq \frac{1}{q_n} \sum_{k=0}^{q_n-1} | \gf(f_0\co{k}(z)) - \gf (f_0\co{\gr^{-1}(\gt(k))}(w)) | \\ 
&= \frac{1}{q_n} \sum_{\substack{0\leq k\leq q_n-1 \\ k \in H(n, \gd)}} | \gf(f_0\co{k}(z)) - \gf (f_0\co{\gr^{-1}(\gt(k))}(w)) |  \\
& \qquad   + \frac{1}{q_n} \sum_{\substack{0\leq k\leq q_n-1 \\ k \nin H(n, \gd)}} | \gf(f_0\co{k}(z)) - \gf (f_0\co{\gr^{-1}(\gt(k))}(w)) | \\ 
& \leq \frac{|H(n, \gd)|}{q_n} \cdot \gep + \frac{1- |H(n, \gd)|}{q_n} \cdot 2 M \leq \gep + \gep \cdot 2M.
\end{aligned}
\end{equation}

For every positive integer $N$ there are non-negative integers $m$ and $r$ with $N-1=m q_n+r$ and $0\leq r \leq q_n-1$. 
Let us assume that $N$ is large enough so that 
\[ \frac{r}{N} 2M \leq \gep.\]
Applying the estimate in Equation~\eqref{E:average-one-turn-Brjuno} to the points $w=f_0\co{i q_n}(z)$, $i\geq 0$, we get

\begin{align*}
\bigg|\frac{1}{N} &\sum_{k=0}^{N-1} \gf(f_0\co{k}(z))- \frac{1}{q_n}\sum_{k=0}^{q_n-1} \gf(f_0\co{k}(z)) \bigg| \\
& \leq \frac{1}{N} \bigg | \sum_{i=0}^{m-1} \sum_{k=i q_n}^{(i+1)q_n-1} \gf(f_0\co{k}(z)) 
+  \frac{1}{N} \sum_{k=mq_n}^{N-1} \gf(f_0\co{k}(z))- \frac{1}{q_n}\sum_{k=0}^{q_n-1} \gf(f_0\co{k}(z)) \bigg| \\
& = \frac{m q_n}{N} \frac{1}{m}\sum_{i=0}^{m-1} \big | \frac{1}{q_n} \sum_{k=i q_n}^{(i+1)q_n-1} \gf(f_0\co{k}(z))- 
\frac{1}{q_n}\sum_{k=0}^{q_n-1} \gf(f_0\co{k}(z)) \big|  
+  \frac{1}{N} \sum_{k=mq_n}^{N-1} |\gf(f_0\co{k}(z))-\gf(0)|\\   
& \leq 1 \cdot \frac{1}{m} \cdot m \cdot \gep (1+ 2M)+ \frac{1}{N} \cdot r \cdot 2 M \\
& \leq \gep (1+ 2M)+ \gep = \gep (2+ 2M).
\end{align*}

As $\gep$ was chosen arbitrarily, the above bound implies that the sequence of Birkhoff averages along the orbit of $z$ forms a Cauchy 
sequence. 
In particular, the sequence of averages along the orbit of every $z\in \pc(f_0)$ is convergent. 

On the other hand, by the abvoe argument, for every $z$ and $w$ in $\pc(f_0)$ and every $\gep>0$ there is $n\geq 0$ such that 
the inequality in Equation~\eqref{E:average-one-turn-Brjuno} holds. 
This implies that the limit of the Birkhoff averages along the orbits of $z$ and $w$ are equal. 
That is, the limit of the sequence is independent of $z$. 
\end{proof}

\begin{proof}[Proof of \refC{C:physical}]
As mentioned in the introduction, the limit set of Lebesgue almost every point in $J(f_0)$ is contained in $\pc(f_0)$. 
(Indeed, in \cite{Ch10-II} we show that for $Q_\ga$ with $\ga\in \irr$ the limit set of the orbit of almost every point in the Julia set 
is equal to $\pc(f_0)$.)
For any such point $z$, every convergent subsequence of the sequence of measures
\[\gm_n= \frac{1}{n}\sum_{k=0}^{n-1}\gd_{f_0\co{k}(z)}.\]
is an $f_0$-invariant probability supported on $\pc(f_0)$. 
However, by \refT{T:unique-ergodicity} there is only one invariant probability measure supported 
on $\pc(f_0)$.
Hence, the above sequence of measures is convergent and converges either to Dirac measure at $0$ 
or the harmonic measure on the boundary of the Siegel disk, depending on the type of $\ga$. 
\end{proof}


\subsection{Hedgehogs and the postcritical set}\label{SS:hedgehogs-and-PC}
Theorem~\ref{C:hedgehog-dynamics} follows from the following proposition and 
Theorem~\ref{T:unique-ergodicity}.

\begin{propo}\label{P:Hedgehogs-vs-postcritical}
Let $f \in \IS_\ga \cup\{Q_\ga\}$ with $\ga \in \irr$ and $K$ be a Siegel compacta of $f$. 
Then, either 
\begin{itemize}
\item[-] $\partial K$ is an invariant analytic curve in the Siegel disk of $f$, or
\item[-] $\partial K$ is contained in the postcritical set of $f$.
\end{itemize}
\end{propo}

The above proposition is proved for the quadratic maps $Q_\ga$ which are not linearizable at 
$0$ in \cite{Chi08}, where only the second possibility may arise.  
We break the proof of the proposition into several lemmas. 

Given $f\in \IS_\ga$ with $\ga\in \irr$, let $f_j$, for $j\geq 0$, denote the sequence of maps 
defined in Section~\ref{S:continued-fraction--renormalizations} with $f_j'(0)=e^{2\pi \B{i}\ga_j }$.
Recall the sets $\C{C}_n^{-k_n}$, $n \geq 0$, defined in Section~\ref{SS:change-coordinates}. 

In the next lemma $b_n=(k_n+a_n-\B{k}-1)q_n +q_{n-1}$.

\begin{lem}\label{L:critical-puzzles-and-returns}
For every $n\geq 0$ there is an integer $l_n$ with $0\leq l_n\leq k_nq_n+ q_{n-1}$ such that 
$f_0\co{l_n}(\gY_n(\C{C}_n^{-k_n}))$ contains the critical point of $f_0$.
Moreover,
\begin{gather*}
\lim_{n\to +\infty} \diam f_0\co{l_n}(\gY_n(\C{C}_n^{-k_n})) =0, \\
\lim_{n\to +\infty} \sup \, \big\{ |f_0\co{b_n}(z)-z| : 
z\in f_0\co{l_n}(\gY_n(\C{C}_n^{-k_n})) \cap \C{PC}(f_0)\big\}=0.
\end{gather*}
\end{lem}
 
\begin{proof}
For every $n\geq 0$, the map $f_n\co{k_n}: \C{C}_n^{-k_n} \to \C{C}_n$ has a unique critical point. 
By Lemma~\ref{L:conjugacy}, $\gY_n \circ f_n\co{k_n}= f_0\co{k_n q_n +q_{n-1}} \circ \gY_n$ 
on $\C{C}_n^{-k_n}$. 
This implies that the map $f_0\co{k_n q_n + q_{n-1}}$ has a critical point in $\gY_n(\C{C}_n^{-k_n})$. 
Since, $f_0$ has a unique critical point in its domain of definition, we conclude that there must be 
an integer $l_n$ with $0 \leq l_n \leq k_n q_n + q_{n-1}$ such that $f_0\co{l_n}(\gY_n(\C{C}_n^{-k_n}))$ 
contains the critical point of $f_0$. 

By Proposition~\ref{P:well-contained-in-domain} there is $\gd>0$ such that for all $n\geq 0$, 
$\gd$-neighborhood of $\C{C}_n^{-k_n}$ is contained in $\Dom f_n \setminus \{0\}$, and
$\diam \C{C}_n^{-k_n} \leq 2/\gd$. 
This implies that for every $n\geq 0$ there is a simply connected region $E_n \subset 
\Dom f_n\setminus \{0\}$ such that the conformal modulus of $(E_n \setminus \C{C}_n^{-k_n})$ is uniformly bounded from below independent of $n$. 
As $\gF_{n-1}(S_{n-1}^0)$ projects onto $\Dom f_n\setminus \{0\}$ under $\ex$, or $s \circ \ex$, 
we conclude that $\gc_{n,0} \circ \gF_n(\C{C}_n^{-k_n})$ has uniformly bounded hyperbolic diameter 
in $\C{D}_{n-1}$. 
On the other hand, the uniform bound in Proposition~\ref{P:turning} implies that 
$\inf \{\Re \C{C}_n^{-k_n}\} \geq a_n +k_n -\B{k}-2- \B{k}''$.
By a similar argument, it follows that for every $z'\in \C{C}_n^{-k_n}$ the hyperbolic distance between 
$\gc_{n,0} \circ \gF_n(z')$ and $\gc_{n,0} \circ \gF_n(f_n\co{a_n+k_n-\B{k}-1}(z'))$ in $\C{D}_{n-1}$ 
is uniformly bounded from above by a constant independent of $n$ and $z'\in \C{C}_n^{-k_n}$. 

Recall from Sections~\ref{SS:petals} and \ref{SS:lifts-vs-iterates} that 
$\gy_n(\C{C}_n^{-k_n}) \subset  J_{n-1}$ and $\gY_n(\C{C}_n^{-k_n}) \subset I_n$. 
It follows from Proposition~\ref{P:lifts-vs-iterates} that there are integers 
$i_j$ with $0 \leq i_j \leq a_{j-1}$, for $1 \leq j \leq n-1$, such that 
\[f_0\co{l_n}(\C{C}_n^{-k_n})=
\gF_0^{-1}\circ \gc_{1, i_1} \circ \gc_{2, i_2} \circ \gc_{n-1, i_{n-1}} (\gF_{n-1}(\gy_n(\C{C}_n^{-k_n}))).\]

Since the image of $\gc_{1, i_1}$ is well contained in $\C{D}_0$, see Lemma~\ref{L:well-contains}, 
the hyperbolic metric on $\C{D}_0$ and the Euclidean metric on $\C{D}_0$ are comparable on the 
set $\gc_{1, i_1}(D_{1})$.
Now, the uniform contraction of the changes of coordinates $\gc_{j, i_j}$ with respect to the 
hyperbolic metrics in Lemma~\ref{L:contraction} implies that the Euclidean diameters must shrink to 
zero. 
\end{proof}

The second limit in the above proposition is a special case of a more general statement.
It is proved in \cite{Ch10-I} that the sequence of maps $f_0\co{q_n}$ converge to the identity map 
on certain sets containing $\C{PC}(f_0)$ (and $f_0\co{l_n}(\gY_n(\C{C}_n^{-k_n}))$). 
But, we do not need this stronger statement here. 

Recall that $S_n^0= \C{C}_n^{-k_n}\cup (\Csh_n)^{-k_n}$.
We break each set $\gU^n$ into two sets as follows.  
Define the sets 
\[ A^n_a= \bigcup_{i=0}^{k_n+a_n-\B{k}-2} f_0\co{(i q_n)} (\gY_n((\Csh_n)^{-k_n})), 
\quad A^n_b= f_0\co{q_{n-1}}(A^n_a), 
\quad A^n= A^n_a \cup A^n_b\]
and
\[B^n_a= \bigcup_{i=0}^{k_n+a_n-\B{k}-2} f_0\co{(i q_n)} (\gY_n(\C{C}_n^{-k_n})), 
\quad B^n_b= f_0\co{q_{n-1}}(B^n_a), 
\quad B^n= B^n_a \cup B^n_b. 
\]
For $n\geq 1$, let 
\[\C{A}^n= \bigcup_{i=0}^{q_n-1}f_0\co{i}(A^n) \cup \{0\}, 
\quad \C{B}^n=\bigcup_{i=0}^{q_n-1}f_0\co{i}(B^n).\]
For every $n\geq 1$, we have $\gU^n= \C{A}^n \cup \C{B}^n$. 
The sets $\C{A}^n$, $\C{B}^n$, and $\gU^n$ are bounded by piecewise smooth curves 
(see Lemma~\ref{L:petal}).
The set $(\Csh_n)^{-k_n}$ is bounded by three (closed) smooth curves. 
Let us denote these curves by $\gga_n$, $\gn_n$, and $\gh_n$, such that   
\begin{gather*}
\gF_n(f_n\co{k_n}(\gga_n(t)))=1/2+(-2+t)\B{i}, \forall t\in [0, +\infty)\\
\gF_n(f_n\co{k_n}(\gn_n(t)))=3/2+(-2+t)\B{i}, \forall t\in [0, +\infty)  \\
\gF_n(f_n\co{k_n}(\gh_n(t)))=1/2+t-2\B{i}, \forall t\in [0,1]. 
\end{gather*}

\begin{lem}\label{L:boundary-curves}
For every $n\geq 0$, we have 
\begin{itemize}
\item[a)] $\gga_n$ is contained in the interior of $\cup_{m=k_n}^{k_n+a_n - \B{k}-2}(f_n\co{m}(S_n^0))$;
\item[b)] $f_n\co{(k_n+a_n - \B{k}-2)}(\gn_n)$ is contained in the interior of 
$\cup_{m=0}^{k_n-1}(f_n\co{m}(S_n^0))$.
\end{itemize}
\end{lem}

\begin{proof}
Recall from Section~\ref{SS:IS-class} that the ellipse $E$ is contained in $B(0,2)$. 
By a simple calculation, $B(0, 8/9) \subset U$. 
On can verify that the polynomial $P$ is one-to-one on the ball $B(0,1/3)$.
By 1/4-theorem, $\gy (U)$ contains the ball $B(0,2/9)$ and $\gy(B(0, 1/3))$ contains $B(0, 1/12)$.  

By Theorem~\ref{T:Ino-Shi2} and the above paragraph, $\C{R}(f_n)$ is univalent on $B(0, 1/12)$. 
In particular, $B(0, \frac{4}{27} e^{-4\pi})$ is contained in $B(0,1/12)$, 
and by Koebe distortion theorem, 
\[\C{R}(f_n)^{-1}(B(0, \frac{4}{27}e^{-4\pi}) \subset B(0, \frac{4}{27}e^{+4\pi}).\] 
Therefore, by the definition of renormalization,
\[\ex(\gF_n(\Csh_n)^{-k_n}) \subset \ex(\gF_n(f_n\co{k_n}(S_n))).\]
The above equation implies the first part of the lemma. 

On the other hand, $\Im \gF_n(f_n\co{(k_n+a_n - \B{k}-2)}(\gn_n))\geq 2$ and 
$\C{R}(f_n)$ is defined on $B(0, \frac{4}{27} e^{-4\pi})$. 
Since $\Re \gF_n(f_n\co{(k_n+a_n - \B{k}-2)}(\gn_n)) = a_n-\B{k}_n-1/2$, the 
second part of the lemma follows.
\end{proof}

\begin{lem}\label{L:Upsilon-boundary}
For every $n\geq 1$, the closure of $\C{A}^n$ is contained in the interior of $\gU^n$. 
In particular, the boundary of $\gU^n$ is contained in the closure of $\C{B}^n$. 
\end{lem}

\begin{proof}
Fix an integer $n\geq 1$ and a point $z$ in the closure of $\C{A}^n$. 
The point $z=0$ belongs to the interior of $\gU^n$, so below we assume that $z$ is nonzero.

By the definition of the set $\C{A}^n$, there is an integer $l$ of the form $iq_n +j$ or $iq_n+q_{n-1}+j$, 
with $0 \leq i \leq k_n+a_n-\B{k}-2$ and $0 \leq j \leq q_n-1$, such that 
$z \in f_0\co{l}(\gY^n(\Csh_n)^{-k_n})$. 

If $z$ belongs to the interior of $f_0\co{l}(\gY_n(\Csh_n)^{-k_n})$ then it belongs to the interior of 
$\gU^n$ and we are done. 
If $z$ belongs to the curve $f_0\co{l}(\gY_n(\gh_n))$ minus its end points, then it belongs to the 
interior of the set $f_0\co{l}(S_0^n)$.
Hence, $z$ belongs to the interior of $\gU^n$. 
It remains to consider the situation where $z$ belongs to the boundary curves 
$f_0\co{l}(\gY_n(\gga_n \cup \gn_n))$.

First assume that $z\in f_0\co{l}(\gY_n(\gn_n))$, and choose $z'\in \gn_n$ with $z= f_0\co{l}(\gY_n(z'))$. 
We consider three cases.

\begin{itemize}
\item[(a)] $i< a_n+k_n-\B{k}-2$;
\item[(b)] $i= a_n+k_n-\B{k}-2$ and $l=i q_n+q_{n-1} + j$;
\item[(c)] $i= a_n+k_n-\B{k}-2$ and $l=i q_n + j$.
\end{itemize}
 
Assume that (a) holds. 
There is $w\in A^n_a \cup A^n_b$ such that $z=f_0\co{j}(w)$ and $w=f_0\co{(l-j)}(\gY_n(z'))$, where 
$l-j$ is either equal to $i q_n$ or $iq_n+q_{n-1}$. 
Note that $z'$ belongs to the interior of $S_n^0\cup f_n(S_n^0)$.  
Also, $\gY_n$ has a univalent extension onto $f_n(S_n)$ through the compositions of the lifts $\gc_{.,0}$. 
Then, $f_0\co{l}(\gY_n(f_n(S_n^0)))=f_0\co{l+q_n}(\gY_n(S_n^0)) \subset \gU^n$.
By the open mapping property of holomorphic and anti-holomorphic maps, this implies that 
$w$ must belong to the interior of 
$f_0\co{l-j} (\gY_n(S_n^0 \cup f_n(S_n^0))) \subset A^n_a\cup A^n_b$. 
Hence, $z=f_0\co{j}(w)$ must belong to the interior of $\gU^n$. 

Assume that (b) holds. 
There is $w\in A^n_b$ such that $z=f_0\co{j}(w)$ and $w=f_0\co{(l-j)}(\gY_n(z'))$.
By Lemma~\ref{L:boundary-curves}-b), $f_n\co{k_n+a_n-\B{k}-2}(\gn_n)$ is contained in the interior of 
$\cup_{m=0}^{k_n-1}(S_n^0)$. 
By Lemma~\ref{L:conjugacy}, this implies that $f_0\co{(l-j)}(\gY_n(\gn_n))$
is contained in the interior of $\cup_{m=0}^{k_n-1} f_0\co{(mq_n)}(\gY_n(S_n^0))\subset A^n_a$. 
By the open mapping property of $f_0\co{j}$, we conclude that $z$ lies in the interior of 
$\gU^n$.

Assume that (c) holds. 
Choose $w\in f_0\co{iq_n}(\gY_n(\gn_n))$ such that $z=f_0\co{j}(w)$. 
We have 
\[f_0\co{(iq_n)}(\gY_n(\gn_n))=
f_0 \co{(q_n-q_{n-1})} \circ f_0\co{q_{n-1}} \circ f_0\co{((i-1)q_n}(\gY_n(\gn_n)).\]
Let $w'\in f_0\co{q_{n-1}} \circ  f_0\co{((i-1)q_n}(\gY_n(\gn_n))$ be such that 
$f_0\co{(q_n-q_{n-1})} (w')=w$. 
By the argument in case a), $w'$ belongs to the interior of 
$A^n_a\cup A^n_b= A^n$. 
As $f_0\co{(q_n-q_{n-1})}$ maps open sets to open sets, $w$ belongs to the interior of 
$f_0\co{(q_n- q_{n-1})}(A^n)$. 
In particular, if $j\leq q_{n-1}-1$, we conclude that $z$ belongs to the interior of $\gU^n$. 

On the other hand, if $j\geq q_{n-1}$, by case b) above,  $f_0\co{q_{n-1}}(w)$ belongs to the interior 
of $A^n_b$. 
Therefore, for every $j$ with $q_{n-1} \leq j \leq q_n-1$, $z$ belongs to the interior of 
$f_0\co{(j-q_{n-1})}(A^n_b) \subset \gU^n$. 

Now assume $z\in f_0\co{l}(\gY_n(\gga_n))$. 
Choose $z'\in \gga_n$ with $z= f_0\co{l}(\gY_n(z'))$ and consider the following three cases.
\begin{itemize}
\item[(a)] $i\neq 0$;
\item[(b)] $i= 0$ and $l=i q_n+q_{n-1} + j$;
\item[(c)] $i= 0$ and $l=i q_n + j$.
\end{itemize}
The arguments in these cases are similar to the above ones, except that one uses part a) of Lemma~\ref{L:boundary-curves} instead of part b). 
We leave further details to the reader.
\end{proof}

Assume that $W\ni 0$ is a Jordan domain such that $f\in \IS_\ga$ and $f^{-1}$ are defined 
and univalent on a neighborhood of the closure of $W$. 
Let $K$ denote the invariant Siegel compacta of $f$ associated to the domain $W$. 
Recall that if $\gep_0=+1$ then $f_0=f$, and if $\gep_0=-1$ then $f_0=s\circ f \circ s$.
Define the domain $W'$ as $W$ if $\gep_0=+1$, and $W'=s(W)$ if $\gep_0=-1$. 
Then the Siegel compacta of $f_0$ in the closure of $W'$, denoted by $K'$, is equal to either $K$ or 
$s(K)$, depending on the sign of $\gep_0$. 

\begin{lem}\label{L:hedgehogs-postcritical}
Let $K'$ be a Siegel compacta of $f_0$. 
There is an integer $n_0\geq 0$ such that for all $n\geq n_0$, the set $K'$ does not intersect the 
closure of $\C{B}^n$.
\end{lem}

\begin{proof}
Let $\cp_0$ denote the critical point of $f_0$.
As $f_0$ is univalent on a neighborhood of $K'$, there is $\gd>0$ such that 
$B(\cp_0, \gd)\cap K'=\emptyset$. 
By Lemma~\ref{L:critical-puzzles-and-returns}, there is $n_0$ such that for all $n\geq n_0$, we have 
$\diam f_0\co{l_n}(\gY_n(\C{C}_n^{-k_n}))\leq \gd/3$ and for all $z\in \gY_n(\C{C}_n^{-k_n})$, 
$|f_0\co{b_n}(z)-z|\leq \gd/3$. 
In particular, for all $n\geq n_0$, $K'$ does not intersect the closure of 
$f_0\co{l_n}(\gY_n(\C{C}_n^{-k_n}))$.

As $f_0(K')=K'$, for all $n\geq n_0$ and all $i$ with $0 \leq i \leq l_n$, 
$K'$ may not intersect the closure of $f_0\co{i}(\gY_n(\C{C}_n^{-k_n}))$. 
However, we cannot immediately use the backward invariance of $K'$ to conclude the same statement 
for the other values of $i$. 
That is because, $K'$ is only fully invariant when $f_0$ is restricted to the domain $W'$, and 
there is no relation between $W'$ and $\gU^n$.

Assume that there is $n\geq n_0$ and an integer $i$ with 
$l_n \leq i \leq  b_n-1$ such that $K'$ intersects the closure of $f_0\co{i}(\gY_n(\C{C}_n^{-k_n}))$. 
Let $z'\in K'$ belong to the closure of $f_0\co{i}(\gY_n(\C{C}_n^{-k_n}))$, and choose $z$ in the 
closure of $\gY_n(\C{C}_n^{-k_n})$ such that $f_0\co{i}(z)=z'$.
Then, by the invariance of $K'$, $f_0\co{(b_n-i)}(z') \in K'$, and 
\begin{equation*}
d(\cp_0, f_0\co{(b_n-i)}(z')) 
\leq d(\cp_0, z)+ d(z, f_0\co{b_n}(z))
\leq \gd/3+ \gd/3.
\end{equation*}
That is, there is an element of $K'$ within $2\gd/3$-neighborhood of $\cp_0$,  
contradicting the choice of $\gd$. 
\end{proof}

\begin{proof}[Proof of Proposition~\ref{P:Hedgehogs-vs-postcritical}]
Since $f_0$ is conjugate to $f$, it is enough to prove the proposition for $f_0$ and $K'$.
As $K'$ is connected  and contains $0$, by the previous lemma, 
$K' \subseteq \cap_{n\geq n_0} \gU^n$. 
On the other hand, by Proposition~\ref{P:Siegel-disk-enclosed}, 
$\cap_{n\geq n_0} \gU^n$ is equal to $\gD(f_0) \cup \C{PC}(f_0)$ and 
$\partial \gD(f_0) \subseteq \C{PC}(f_0)$.
Thus, $K'$ is a connected invariant region in $\gD(f_0) \cup \C{PC}(f_0)$. 
This implies that either $K$ is equal to the region bounded by an analytic curve in $\gD(f_0)$, 
or it contains $\gD(f_0)$. 
\end{proof}

As a corollary of the above lemmas we conclude the following result.  

\begin{thm}\label{T:one-to-0ne}
For every $f \in \IS_\ga \cup \{Q_\ga\}$ with $\ga \in \irr$, $f : \C{PC}(f) \to \C{PC}(f)$ 
is one-to-one.
\end{thm}

\begin{proof}
For every $n\geq 1$, if there are $z$ and $z'$ in $\gU^n$ with $f_0(z)=f_0(z')$ then 
$z$ and $z'$ must belong to $f_0\co{l_n}(\C{C}_n^{-k_n})$ where $l_n$ is the integer introduced in 
Lemma~\ref{L:critical-puzzles-and-returns}. 
Since the diameters of $f_0\co{l_n}(\C{C}_n^{-k_n})$ tend to $0$ as $n$ tends to $+\infty$, we conclude 
that $f_0$ is one-to-one on $\cap_{n\geq 1} \gU^n$. 
By Proposition~\ref{P:pc-neighbor}, the postcritical set of $f_0$ is contained in this intersection. 
\end{proof}
\bibliographystyle{amsalpha}
\bibliography{/Users/Davoud/Work/My-Work/Data}
\end{document}